\documentclass{amsart}

\usepackage{amsmath,amsthm,amsfonts,amssymb,bm,graphicx, mathrsfs}

\usepackage{comment}
\usepackage{tikz-cd}
\usepackage
{hyperref}

\hypersetup{colorlinks=true,citecolor=blue,linkcolor=blue,urlcolor=blue,
pdfstartview=FitH }

\newcommand{\bG}{\mathbf{G}}

\newcommand{\bB}{\mathbf{B}}
\newcommand{\bPB}{\mathbf{PB}}
\newcommand{\bT}{\mathbf{T}}

\newcommand{\SO}{\textrm{SO}}
\newcommand{\C}{{\mathbb C}}
\newcommand{\N}{{\mathbb N}}
\newcommand{\R}{{\mathbb R}}
\newcommand{\Z}{{\mathbb Z}}
\newcommand{\Q}{{\mathbb Q}}

\newcommand{\A}{{\mathbb A}}

\newcommand{\bs}{\backslash}

\newcommand{\SL}{\operatorname{SL}}
\newcommand{\GL}{\operatorname{GL}}
\newcommand{\PGL}{\operatorname{PGL}}

\renewcommand{\Re}{\operatorname{Re}}

\newcommand{\gf}{\mathfrak{g}}

\theoremstyle{plain}
\newtheorem{lemma}{Lemma}
\newtheorem{theorem}[lemma]{Theorem}
\newtheorem*{theorem*}{Theorem}

\newtheorem*{conj*}{Conjecture}

\newtheorem{cor}[lemma]{Corollary}

\numberwithin{equation}{section}
\newtheorem*{question*}{Question}
\theoremstyle{definition}

\newtheorem*{example*}{Example}
\newtheorem{remark}{Remark}

\begin{document}

\author{Valentin Blomer}
\author{Farrell Brumley}
 
\address{Mathematisches Institut, Endenicher Allee 60, 53115 Bonn, Germany}
\email{blomer@math.uni-bonn.de}

\address{LAGA, Institut Galil\'ee, 99 avenue Jean Baptiste Cl\'ement, 93430 Villetaneuse, France}
\email{brumley@math.univ-paris13.fr}
  
\title[Simultaneous equidistribution]{Simultaneous equidistribution of toric periods and fractional moments of $L$-functions}

\thanks{The first author was supported in part by the DFG-SNF lead agency program grant BL 915/2-2}

\begin{abstract} The embedding of a torus into an inner form of  $\PGL_2$ defines an adelic toric period. A general version of Duke's theorem states that this period equidistributes as the discriminant of the splitting field tends to infinity. In this paper we consider a torus embedded diagonally into two distinct inner forms  of $\PGL_2$. Assuming the Generalized Riemann Hypothesis (and some additional technical assumptions), we show simultaneous equidistribution as the discriminant tends to infinity, with an effective logarithmic rate. Our proof is based on a probabilistic approach to estimating fractional moments of $L$-functions twisted by extended class group characters.  
\end{abstract}


\maketitle

\section{Introduction}
A well-known theorem of Legendre (proved by Gau\ss) states that a positive integer $d$ is a sum of three integer squares precisely when $d$ is not of the form $d=4^a(8b+7)$, for $a, b \in \N \cup \{0\}$. Let $\Z^3_{\rm prim}$ denote the subset of $\Z^3$ whose coordinates are relatively prime. Then the Legendre--Gau{\ss} theorem can be reformulated to state that the set
\begin{equation}\label{11}
\mathscr{R}_d=\big\{ x \in \Bbb{Z}_{\rm prim}^3:  x_1^2 + x_2^2 + x_3^2 = d\big\}
\end{equation}
is non-empty precisely when $d$ lies in $\mathbb{D}=\{d\in\N: d\not\equiv 0,4,7\}$, the set of locally admissible values. Gau{\ss} proved in fact much more; quite remarkably, he established an exact formula for the cardinality $|\mathscr{R}_d|$ in terms of class numbers of quadratic orders in the number field $E=\Q(\sqrt{-d})$ (\cite[Section 291]{Ga}). For example, if $\mathbb{D}_{\rm fund}=\mathbb{D}\cap\mathbb{F}$, where $\mathbb{F}$ denotes the set of square-free integers, then for $d>3$ in $\mathbb{D}_{\rm fund}$ he showed that
\begin{equation}\label{Gauss-formula}
|\mathscr{R}_d|=
\begin{cases} 
24 |{\rm Cl}_E|, & d \equiv 3\pmod{8};\\
12 |{\rm Cl}_E|, & d\equiv 1, 2\pmod{4},
\end{cases}
\end{equation}
where ${\rm Cl}_E$ is the class group of the ring of integers of $E$. In particular, it follows from the work of Dirichlet and Siegel that $|\mathscr{R}_d|=d^{1/2+o(1)}$ for $d\in \mathbb{D}_{\rm fund}$ (in fact for all $d\in\mathbb{D}$). 

With the issues of cardinality settled, one can then investigate the distribution of the points in $\mathscr{R}_d$. Using his ergodic method, Linnik \cite{Li2} in the late 1950's proved that the projection $d^{-1/2}\mathscr{R}_d$ equidistributes on $S^2$ -- the Euclidean sphere in $\R^3$ -- with respect to the rotationally invariant Lebesgue probability measure $m_{S^2}$, for $d\rightarrow\infty$ along a sequence in $\mathbb{D}$ such that
\begin{equation}\label{Linnik-cond}
-d\textit{ is a non-zero square modulo } p,
\end{equation}
where $p$ is a fixed odd prime. This \textit{Linnik condition}, as it is called,   is equivalent to the splitting at $p$ of a stabilizer subgroup of an acting orthogonal group, which in turn allows for the use of measure classification results arising from homogeneous dynamics. The rate of convergence in Linnik's proof can be made sufficiently uniform in $p$ so as to allow for the following reformulation: the set $d^{-1/2}\mathscr{R}_d$ equidistributes on $S^2$, provided that $d\rightarrow\infty$ along a sequence in $\mathbb{D}$ for which there exists a prime $p$ which splits in $E=\Q(\sqrt{-d})$ such that $\log p\ll \log d/\log\log d$. The existence of such small split primes is in fact known to hold under the Generalized Riemann Hypothesis (GRH) for $L$-functions of quadratic Dirichlet characters.  

Gau{\ss}' formula \eqref{Gauss-formula} hints at an interesting  structural relation between $\mathscr{R}_d$ and the class group ${\rm Cl}_E$. In fact it was Venkov \cite{Ve} who 
first explicated Gau{\ss}' formula in terms of quaternion algebras, and we shall adopt this perspective. Let $\bB = \bB^{(2, \infty)}$ denote the rational quaternion algebra ramified at $2$ and $\infty$. A solution $x = (x_1, x_2, x_3) \in \mathscr{R}_d$ can then be identified with a trace-zero integral quaternion $x_1i + x_2j + x_3k$ of reduced norm $d$. When $d\in \mathbb{D}_{\rm fund}$, the choice of a base point $x_0\in\mathscr{R}_d$ yields an optimal embedding $\iota: E \rightarrow \textbf{B}(\Bbb{Q})$, $a + b \sqrt{-d} \mapsto a + b x_0$, relative to the maximal order $\mathscr{O}$ of Hurwitz quaternions. Since $\mathscr{O}$ is principal, if $\mathfrak{a}$ is a fractional ideal in $E$, the $\Z$-module $\iota(\mathfrak{a})\mathscr{O}$ is a principal ideal $(q)$ in $\mathscr{O}$. Letting $\Gamma$ denote the order 12 group of projective units $\mathscr{O}^\times/\{\pm 1\}$, which acts on the coordinate lines in $\R^3$ via even permutations, we may define an action of ${\rm Cl}_E$ on the quotient $\mathscr{R}^*_d=\Gamma\backslash \mathscr{R}_d$ by $[\mathfrak{a}] \cdot x = q^{-1} x q$. This action is free and has one or two orbits, according to the two types of congruence classes in \eqref{Gauss-formula}.

The  problem of equidistribution of integer points on the sphere admits many variants, such as the distribution of Heegner points on the modular curve \cite{Li1}, packets of closed geodesics \cite{Sk} on the modular curve, as well as the supersingular reduction of CM elliptic curves \cite{Mi}. We will review these examples later in Section \ref{sec:examples}. In  each case, the underlying set (or ``packet'') of arithmetic objects admits an action by the class group of a quadratic order. From a modern perspective, what they have in common is the equidistribution of the adelic quotient $[\bT]=\bT(\Q)\bs\bT(\A)$ of an algebraic torus $\bT$ of large discriminant inside the automorphic space $[\bG]=\bG(\Q)\bs\bG(\A)$ of an inner form $\bG$ of $\PGL_2$.

Several decades after Linnik's fundamental contributions, Iwaniec \cite{Iw} developed an innovative technique to bound the Fourier coefficients of half integral weight modular forms which paved the way for removing Linnik's condition \eqref{Linnik-cond}, \textit{unconditionally} on GRH. By extending Iwaniec's methods, Duke, in his famous paper \cite{Du}, proved the equidistribution of $d^{-1/2}\mathscr{R}_d$ on $S^2$ for all $d\in\mathbb{D}_{\rm fund}$, with a power savings rate of convergence. This result, and the others he treated in \cite{Du}, now go collectively under the name of \textit{Duke's Theorems} and cover in particular the case of Heegner points, closed geodesics and integer points on general ellipsoids \cite{DSP}. The method is very different from Linnik's ergodic approach and based on harmonic analysis and automorphic forms. The most uniform treatment uses  Waldspurger's theorem on toric periods \cite{Wa}  and subconvex bounds on quadratic character twists of degree two $L$-functions, due to Duke, Friedlander, and Iwaniec \cite{DFI} when the base field is $\Q$, and to Michel-Venkatesh \cite{MV2} in general. This family of results belongs to the landmark achievements in analytic number theory in the past 30 years. 

\subsection{Simultaneous equidistribution}

A natural generalization of the above setting, first put forward by Michel and Venkatesh in their 2006 ICM proceedings \cite{MV}, is to consider the diagonal embedding of a quadratic number field $E$ into two distinct quaternion algebras $\bB_1$ and $\bB_2$. A dynamical analogue of Goursat's lemma would suggest that, since the projective unit groups $\bG_1=\bPB_1^\times$ and $\bG_2=\bPB_2^\times$ are non-isomorphic, and the adelic quotient $[\bT]$ of the projective torus $\bT$ defined by $E$ equidistributes on each $[\bG_i]$ by Duke's theorem, then the diagonal embedding of $[\bT]$ should equidistribute on the product $[\bG_1]\times [\bG_2]$, as the discriminant  of $E$ gets large.

This problem was approached through ergodic theory by Einsiedler and Lindenstrauss, by bootstrapping Duke's theorem in each $\bG_i$ by means of their joinings theorem \cite{EL}. However, their method, like that of Linnik, requires an auxiliary congruence condition on the allowed set of discriminants, similar to \eqref{Linnik-cond} but with \emph{two} auxiliary primes. This \textit{double Linnik condition} guarantees the existence of an action by a higher rank torus, which can be shown to enjoy decisive measure rigidity properties. Arithmetic applications of the joinings theorem have been explicated in \cite{AES, ALMW, AEW, Kh}, with many still to come.

There is an important quantitative distinction between Linnik's result and that of Einsiedler and Lindenstrauss. While Linnik's condition \eqref{Linnik-cond} can be removed under GRH, the joint equidistribution statements of Einsiedler and Lindenstrauss are ineffective (without a rate of convergence), and it is presently unknown whether their methods can be strengthened to allow for a replacement of the double Linnik condition by the assumption of GRH.   

The main result of this paper is a proof of this conjecture conditionally on the Generalized Riemann Hypothesis (and some minor simplifying assumptions). Like Duke, we approach the problem through automorphic forms and $L$-functions, using Waldspurger's formula as a crucial input. Unlike his setting, subconvexity is not enough, and one must estimate fractional moments of a certain family of Rankin-Selberg $L$-functions. The assumption of GRH allows us to use methods in probabilistic number theory pioneered by Soundararajan \cite{So}. The analytic heart of the paper is Theorem \ref{thm3} in Section \ref{sec2}.  A special case of this result can be rephrased as a non-trivial bound for Bessel periods of Yoshida lifts on ${\rm GSp}(4)$, cf.\ Corollary \ref{cor4}.

We delay the full statement of our main theorems until Section \ref{main}, since they require substantial notational preliminaries. Instead, we provide an illustrative special case, building upon Linnik's sphere problem. More examples will be given in Section \ref{sec:examples}. 

\medskip

\noindent {\sc Example.} We let be $\bB$ one of the five quaternion algebras over $\Q$ having class number 1 (cf.\ \cite[Theorem 25.4.1]{Vo}), with discriminants 2, 3, 5, 7, or 13.  Let $Q={\rm Nm}|_{\bB^0}$ be the reduced norm form ${\rm Nm}$ restricted to the trace-zero elements $\bB^0$ of $\bB$. Choose a maximal order $\mathscr{O}$ in $\bB(\Q)$ and write $\mathscr{O}^0=\mathscr{O}\cap\bB^0(\Q)$. If $\mathscr{O}_{\rm prim}^0$ denotes the subset of primitive elements we put, for any $d\in\N$,
\[
\mathscr{R}_d(Q)=\{x\in\mathscr{O}_{\rm prim}^0: Q(x)=d\}.
\]
Let $\mathbb{D}(Q)=\{d\in\N: \mathscr{R}_d(Q)\neq\emptyset\}$ and $\mathbb{D}_{\rm fund}(Q)=\mathbb{D}(Q)\cap\mathbb{F}$. Shemanske \cite{Sh} extended the method of Venkov to show that for $d\in\mathbb{D}_{\rm fund}(Q)$, with $d>3$, the class group $\text{Cl}_E$ of $E = \Bbb{Q}(\sqrt{-d})$ acts freely with one, two, four, or eight orbits on $\mathscr{R}_d^*(Q)=\Gamma\backslash\mathscr{R}_d(Q)$, where $\Gamma$ is the group of projective units $\mathscr{O}^\times/\{\pm 1\}$.

Define the ellipse $V_Q= \{x \in \Bbb{R} \mid Q(x) = 1\}$, endowed with the probability measure $m$ induced by assigning to any $\Omega\subset V_Q$ the Lebesgue volume of $\cup_{x\in\Omega}[0,1]x$. Duke and Schulze-Pillot \cite{DSP} proved that $\mathscr{R}_d^*(Q)$ equidistributes on $\Gamma\bs V_Q$ relative to the measure $m$ as $d\rightarrow\infty$ in $\mathbb{D}_{\rm fund}(Q)$.

Now take two \textit{distinct} quaternion algebras $\bB_1$ and $\bB_2$  over $\Q$ of class number 1. Fixing base points $x_i \in \mathscr{R}_d^*(Q_i)$ we can consider the joint orbit
\begin{equation}\label{joint-orbit}
\Delta\mathscr{R}_d^*(Q_1,Q_2) =   \big\{(tx_1,  tx_2) \mid t \in \text{Cl}_E\big\}\subseteq \mathscr{R}_d^*(Q_1) \times \mathscr{R}_d^*(Q_2).
\end{equation}
A special case of Theorem \ref{thm1} below is the following result.  \textit{Assume the Generalized Riemann Hypothesis.  Then $d^{-1/2}\Delta\mathscr{R}_d^*(Q_1,Q_2)$ equidistributes in $\Gamma_1\bs V_{Q_1}\times\Gamma_2\bs V_{Q_2}$ with respect to the product measure $m_1\times m_2$, as $d \rightarrow \infty$ in $\mathbb{D}_{\rm fund}(Q_1)\cap \mathbb{D}_{\rm fund}(Q_2)$.} 
 
\subsection{The conjecture of Michel-Venkatesh}\label{sec:M-V}
We now pass to the adelic language. 
To set up the stage, we wish to review the most general form of Duke's theorems, as they have been refined and generalized in recent years, most notably by Einsiedler, Lindenstrauss, Michel, and Venkatesh.

Let $\bB$ be a quaternion algebra defined over a number field $F$ and write $\bPB^\times$ for its group of projective units. A \textit{homogeneous toral subset} inside the automorphic quotient space $[\bPB^\times]=\bPB^\times(F)\bs\bPB^\times(\A_F)$ is a set of the form
\[
[\bT_\iota.{\rm g}]= \bT_\iota (F)\bs\bT_\iota (\A_F).{\rm g},
\]
where $\bT_\iota\subset\bPB^\times$ is the image under a rational embedding $\iota:\bT\hookrightarrow\bPB^\times$ of an anisotropic algebraic torus over $F$ and ${\rm g}\in\bPB^\times(\A_F)$. One can associate with $[\bT_\iota.{\rm g}]$ two important objects:
\begin{enumerate}
\item a positive number ${ D}$ called the \textit{discriminant}, defined in \cite[\S 4.2]{ELMV2}, which encodes the arithmetic complexity of $[\bT_\iota.{\rm g}]$;
\item a ${\rm g}^{-1}\bT_\iota(\A){\rm g}$-invariant probability measure $\mu$ on $[\bT_\iota.{\rm g}]$, given by the push-forward under $t\mapsto \iota(t){\rm g}$ of the normalized Haar measure on $\bT_\iota (F)\bs\bT_\iota (\A_F)$.
\end{enumerate}

Now let $[\bT_{\iota_n}.{\rm g}_n]$ be a sequence of homogeneous toral subsets, with associated probability measures $\mu_n$, such that ${ D}_n\rightarrow\infty$. The following theorem, which was stated in \cite[Theorem 4.6]{ELMV2}, gives a minimal set of conditions under which $[\bT_{\iota_n}.{\rm g}_n]$ equidistributes with respect to a convex combination of homogeneous probability measures on $[\bPB^\times]$.

\begin{theorem*}
Any weak-* limit of $\mu_n$ is a homogeneous probability measure on $[\bPB^\times]$ invariant under $\bPB^\times(\A_F)^+$. Here, $\bPB^\times(\A_F)^+$ is the image of $\bB^{(1)}(\A_F)\rightarrow\bPB^\times(\A_F)$, where $\bB^{(1)}$ is the simply connected cover of $\bPB^\times$. Moreover, the rate of convergence can be quantified, with an error term of the form $O({ D}^{-\delta})$ for some $\delta>0$. 
\end{theorem*}

In their 2006 ICM address, Michel and Venkatesh \cite[\S 2.3, \S 6.4.1]{MV} considered the following simultaneous equidistribution problem. Let $\bB_1$ and $\bB_2$ be two non-isomorphic quaternion algebras over $F$. For $j  = 1, 2$ let 
$\bG_j=\bPB_j^\times$, and 
let $\bG=\bG_1\times\bG_2$. Let $\bT$ be an anisotropic algebraic torus over $F$ equipped with rational embeddings $\iota_j: \bT\hookrightarrow\bG_j$, for $j=1,2$. We may define a \textit{diagonal homogeneous toral subset} to be the subset of the product space $[\bG]=[\bG_1]\times [\bG_2]$ given by
\[
[\Delta\bT_\iota.{\rm g}]= \Delta\bT_\iota (F)\bs \Delta\bT_\iota (\A_F).{\rm g},
\]
where $\Delta\bT_\iota\subset\bG$ is the image under the diagonal embedding $\iota=(\iota_1,\iota_2):\bT\hookrightarrow\bG$ and ${\rm g}\in\bG(\A_F)$. As before, one may associate with $[\Delta\bT_\iota.{\rm g}]$ its discriminant ${ D}=\min ({ D}_1,{ D}_2)$ and a natural probability measure $\Delta\mu$ on $[\bG]$. Let $\bG(\A_F)^+=\bG_1(\A_F)^+\times\bG_2(\A_F)^+$.

\begin{conj*}[Michel--Venkatesh]
Let $[\Delta\bT_{\iota_n}.{\rm g}_n]$ be a sequence of diagonal homogeneous toral subsets satisfying $D_n\rightarrow\infty$. Then any weak-* limit of $\Delta\mu_n$ is a homogeneous probability measure on $[\bG]$ invariant under $\bG(\A)^+$.
\end{conj*}

By choosing a finite level structure $K_f\subset \bG(\A_f)$ and a compact subgroup $K_\infty\subset \bG(\R)$, we can reinterpret the above conjecture more classically in the double quotient space
\[
[\bG]_K=\bG(F)\bs\bG(\A_F)/K\qquad (K=K_fK_\infty)
\]
by considering the distribution of the images
\[
[\Delta\bT_\iota.{\rm g}]_K=\bG(F) \Delta\bT_\iota(\A_F){\rm g} K
\]
of $[\Delta\bT_\iota.{\rm g}]$ under the natural projection $[\bG]\rightarrow [\bG]_K$. We shall call these projections \textit{diagonal packets}, extending the terminology \cite{ELMV2} to this setting. 

\subsection{Main results}\label{main}
We are now ready to state our principal result. In it, we establish the conjecture of Michel--Venkatesh for diagonal packets, under the assumption (most notably) of the Generalized Riemann Hypothesis. To simplify the presentation, we have made  further restrictions, including maximal level structure, optimal embeddings, and the number field $\Q$. Moreover, we test convergence only against functions in the discrete spectrum, an assumption only relevant if one of $\bB_j$ is the matrix algebra (for more on this assumption, see Remark   \ref{rem4}). More precisely, we give ourselves the following data.

Let $\bB_1,\bB_2$ be non-isomorphic quaternion algebras over $\Q$  and write $\bG_j=\bPB_j^\times$ for $j = 1, 2$. Let $\mathscr{O}_j$ be a maximal order in $\bB_j(\Q)$. For a prime $p$ let $\bG_j(\Z_p)$ denote the projective unit group of the local maximal order $\mathscr{O}_{j,p}=\mathscr{O}_j\otimes\Z_p$. Let $\bG_j(\widehat{\Z})=\prod_p \bG_j(\Z_p)$. Let $\bG_j(\A_\Q)$ denote the adelic points relative to the subgroups $\bG_j(\Z_p)$.  Let $K_{\infty,j}\subset \bG_j(\R)$ be %
\begin{itemize}
\item  either $\bG_j(\R)$ itself if $\bB_j\otimes_\Q\R$  is non-split,
\item  or a maximal compact torus if $\bB_j\otimes_\Q\R$  is non-split or non-split. 
\end{itemize}
When $\bB_j$ is split at infinity, we fix an isomorphism of $\bB_j(\R)$ with $M_2(\R)$ which induces an isomorphism of $\bPB^\times_j(\R)$ with $\PGL_2(\R)$ sending $K_{\infty,j}$ to ${\rm PSO}(2)$.

Let $E$ be a quadratic field extension of $\Q$ with ring of integers $\mathcal{O}_E$. Let $\iota_j:E\hookrightarrow\bB_j$ be an  optimal embedding of $\mathcal{O}_E$ into some maximal order $\mathbb{O}_j$ of $\bB(\Q)$, depending on $\iota_j$. This induces optimal local embeddings at each prime $p$, in the following sense. Let $v$ be a finite place of $E$ lying over $p$. Let $E_v$ be the $v$-adic completion of $E$ (a quadratic \'etale algebra) and write $\mathcal{O}_{E,v}$ for its maximal order. Let $\mathbb{O}_{j,p}=\mathbb{O}_j\otimes\Z_p$. Then $\iota_j(\mathcal{O}_{E_v})=\iota_j(E_v)\cap\mathbb{O}_{j,p}$. Since $\mathscr{O}_j$ is everywhere locally isomorphic to $\mathbb{O}_j$, there are ${\rm g}_{j,p}\in\bG_j(\Q_p)$ such that $\mathscr{O}_{j,p}={\rm g}_{j,p}\mathbb{O}_{j,p}{\rm g}_{j,p}^{-1}$. In this way we obtain
\begin{equation}\label{eq:optimal}
\iota_j(\mathcal{O}_{E_v})=\iota_j(E_v)\cap {\rm g}_{j,p}^{-1}\mathscr{O}_{j,p} {\rm g}_{j,p}.
\end{equation}
 Let ${\rm g}_{j,f}=({\rm g}_{j,p})_p\in \bG_j(\A_f)$. Let $\bT=({\rm Res}_{E/\Q}\mathbb{G}_m)/\mathbb{G}_m$. Then $\iota_j$ induces an embedding $\iota_j:\bT\hookrightarrow\bG_j$  which, in view of \eqref{eq:optimal}, satisfies 
\begin{equation}\label{Tf-optimality}
{\iota_j}(\bT(\widehat{\Z}))=\iota_j(\bT(\A_f))\cap {\rm g}_{j,f}^{-1}\bG_j(\widehat{\Z}){\rm g}_{j,f}. 
\end{equation}
We write  $\bT_{\iota_j}=\iota_j(\bT)\subset\bG_j$. Assume furthermore that ${\rm g}_{j, \infty}\in\bG_j(\R)$ is such that
\begin{equation}\label{conjugation-Tinfty}
\begin{cases}
&\!\!\!\!{\rm g}_{j,\infty}\bT_{\iota_j}(\R){\rm g}_{j,\infty}^{-1}\subset K_{\infty, j},\;\text{when}\; \bT(\R)\;\text{is anisotropic};\\
&\!\!\!\!{\rm g}_{j,\infty} \bT_{\iota_j}(\R){\rm g}_{j,\infty}^{-1}\;\text{is, under the fixed identification of $\bPB^\times(\R)$ with $\PGL_2(\R)$ above,}\\ &\hspace{2,3cm} \text{the group of projective diagonal matrices, when}\; \bT(\R)\;\text{is isotropic}.
\end{cases}
\end{equation}
Put ${\rm g}_j=({\rm g}_{j,f},{\rm g}_{j,\infty})\in\bG(\A_\Q)$.

Put $\bG=\bG_1\times\bG_2$ and let $dg$ be the right $\bG(\Bbb{A}_{\Bbb{Q}})$-invariant probability measure. Let $\bG(\widehat{\Z})=\bG_1(\widehat{\Z})\times\bG_2(\widehat{\Z})$ and $K_\infty=K_{\infty,1}\times K_{\infty,2}$. Put $K=\bG(\widehat{\Z})K_\infty$. Let ${\rm g}=({\rm g}_1,{\rm g}_2)\in \bG(\A_\Q)$. Let $[\Delta\bT_\iota.{\rm g}]$ be a diagonal homogeneous toral subset in $[\bG]$, where $\iota=(\iota_1,\iota_2):\bT\hookrightarrow\bG$. Then $[\Delta\bT_\iota.{\rm g}]$ is endowed with its invariant probability measure $\Delta \mu$, which we shall write simply by $dt$. 
Denote by ${ D}$ the discriminant of the packet $[\Delta\bT_\iota.{\rm g}]_K$. The conditions \eqref{eq:optimal} and \eqref{conjugation-Tinfty}, together with suitable choices of archimedean metric normalizations, imply that ${ D}$ is the absolute value of the discriminant of $E$; see Section \ref{sec:discriminant} for details.

Let $C^{\infty}_c([\bG]_K)$ denote the space of right $K$-invariant compactly supported functions $f : [\bG]\rightarrow \Bbb{C}$ such that $g\mapsto f(xg)$ is in $C_c^\infty(\bG(\R))$ for every $x\in [\bG]$. Let $C^{\infty}_{c,{\rm disc}}([\bG]_K)= C^{\infty}_c([\bG]_K)\cap L^2_{\rm disc}([\bG]_K)$. We have $C^{\infty}_{c,{\rm disc}}([\bG]_K) = C^{\infty}_{c}([\bG]_K)$  unless one of $\bG_j$ is ${\rm PGL}_2$. For $f\in C^{\infty}_c([\bG]_K)$ let $\mathcal{S}_{\infty,d}(f)=\sum_{{\rm ord}(\mathcal{D})\leq d}\|\mathcal{D}f\|_{L^\infty([\bG]_K)}$, where $\mathcal{D}$ runs over monomials of degree at most $d$ in a fixed basis of ${\rm Lie}(\bG(\R))$.

\begin{theorem}\label{thm1}
Let the notations be as above. Then, {\rm under the Generalized Riemann Hypothesis}, the diagonal packets $[\Delta\bT_\iota.{\rm g}]_K$ equidistribute on $[\bG]_K$ relative to $C_{c,{\rm disc}}^\infty([\bG]_K)$, with an effective rate of convergence of the form $O_\varepsilon((\log { D})^{-1/4+\varepsilon})$, as ${ D}\rightarrow\infty$. More precisely, there is $d\in\N$ such that for every 
$f\in C_{c,{\rm disc}}^\infty([\bG]_K)$, and every $\varepsilon>0$, we have
\[
\int_{[\Delta\bT_\iota]} f(t{\rm g})dt=\int_{[\bG]}f(g)dg+O_\varepsilon\Big(\mathcal{S}_{\infty,d}(f)(\log { D})^{-1/4+\varepsilon}\Big).
\]
\end{theorem}

A variety of situations in which our theorem applies will be given in Section \ref{sec:examples}.  Our effective error term also allows  applications to equidistribution on (very slowly) shrinking subsets of $[\bG]_K$.  

We may also prove an equidistribution statement in the case $\bB_1= \bB_2$ (or otherwise) if we twist the diagonal embedding by allowing each component embedding to travel through $\bT$ with different  speeds:
\[
\iota_{\alpha,\beta}:\bT \rightarrow \bPB_1^\times\times \bPB_2^\times, \qquad t\mapsto (\iota_1(t)^\alpha,\iota_2(t)^\beta),
\]
where $\alpha,\beta\in\N$ are distinct integers. Let $\Delta_{\alpha,\beta}\bT_\iota$ denote the image of $\bT$ under $\iota_{\alpha,\beta}$ and write $[\Delta_{\alpha,\beta}\bT_\iota.{\rm g}]_K$ for the image of $\Delta_{\alpha,\beta}\bT_\iota(\Q)\bs\Delta_{\alpha,\beta}\bT_\iota(\A_\Q).{\rm g}$ in $[\bG]_K$.

\begin{theorem}\label{thm2} Let $\bB_1$ and $\bB_2$ be quaternion algebras over $\Q$ (not necessarily distinct), and otherwise keep the assumptions and notations of Theorem \ref{thm1}. Let $\alpha,\beta\in \N$ be two distinct integers. Suppose that the  class group ${\rm Cl}_E$ of the field $E$ associated with the torus $\bT$ has no $p$-torsion, for all\footnote{By symmetry, we can also assume this for all $p \mid 2\alpha$.} $p\mid 2\beta$.  Then, {\rm under the Generalized Riemann Hypothesis}, there is $d\in\N$ such that for every  
 $f\in C_{c,{\rm disc}}^\infty([\bG])$, and every $\varepsilon > 0$, we have
\[
\int_{[\Delta_{\alpha,\beta}\bT_\iota]} f(t{\rm g})dt  =\int_{[\bG]} f(g) dg + O_{\varepsilon, \alpha, \beta}\Big(\mathcal{S}_{\infty,d}(f)(\log { D})^{-1/4 + \varepsilon}\Big).
\]
\end{theorem}


Several assumptions in Theorems \ref{thm1} and \ref{thm2} can be relaxed, using the same methods, but at the cost of a greater technical effort. For instance, instead of maximal orders one can take Eichler orders (and even more general orders, such as in \cite[Example 10.5]{ALMW}), the difficulty being in treating oldforms. 
The assumption in Theorem \ref{thm2} on the $2\beta$ torsion can also be relaxed, for instance it would suffice for its cardinality to be bounded independently of ${ D}$, cf.\   Section \ref{sec:Parseval} and Remark \ref{rem6}.  Note that if $E$ is real, a conjecture of Gau{\ss} (quantified by Hooley \cite{Hoo}) says that for ``many'' discriminants the group $\rm Cl_E$ is trivial, in particular torsionfree.  

\subsection{Beyond sparse equidistribution}
As discussed in the opening paragraphs, the modern approach to proving Duke's theorem passes through Waldspurger's formula, which relates the square of the toric period of an automorphic form to an associated $L$-function. A separate problem is then to prove subconvex bounds on these $L$-functions, a time-honored subject in analytic number theory. In fact, one can consider a natural refinement of Duke's theorem, where one seeks to prove the equidistribution of the orbit of a subgroup of the class group ${\rm Cl}_E$ of large enough index. This type of problem has been referred to in the literature as \textit{sparse equidistribution} \cite{Ven}, and it is solved by again appealing to Waldspurger's formula and subconvex bounds, this time on $L$-functions twisted by class group characters \cite{Mi}.

A fundamental property underlying the proof of Waldspurger's formula is the fact that the subgroup pair $(\bT,\bPB^\times)$ is a strong Gelfand pair. This is not the case with $(\Delta\bT,\bPB_1^\times\times \bPB_2^\times)$, as the diagonal torus $\Delta\bT$ is too small relative to the product group $\bPB_1^\times\times \bPB_2^\times$, and one no longer expects the corresponding diagonal period  to be directly related to a single $L$-function. Following Bernstein and Reznikov (see \cite{Rez} for an overview) one can nevertheless form the following \textit{Gelfand formation}:
\[
\begin{tikzcd}
 \bPB_1^\times\times \bPB_2^\times\ar[d,dash]  \\
\bT\times \bT \ar[d,dash] \\
\Delta \bT
\end{tikzcd}
\]
in which the intermediate subgroup pairs $(\Delta\bT,\bT\times\bT)$ and $(\bT\times\bT,\bPB_1^\times\times \bPB_2^\times)$ are strong Gelfand pairs. In such a situation the diagonal period should be related to a \textit{family of twisted $L$-functions}. We establish this link more precisely in Section \ref{sec2}.

The equidistribution problems of Theorems \ref{thm1} and \ref{thm2} therefore go beyond even the sparse equidistribution refinements of Duke's theorem. From this perspective it is perhaps less surprising that one should need the deeper statistical information provided by GRH, which we take as a working assumption. The analytic tools we develop for families of $L$-functions, as expressed in our main analytic number theoretic result Theorem \ref{thm3}, should (we believe) provide a new paradigm for treating equidistribution problems in the absence of direct period formulae.

\subsection{The plan of the paper} In Section \ref{sec:adelic-to-classical} we take some time to translate the content of Subsections \ref{sec:M-V} and  \ref{main} in classical language and in particular discuss general versions of Duke's theorem. This prepares the ground to give classical applications of Theorems \ref{thm1} and \ref{thm2} in Section \ref{sec:examples}
\begin{itemize}
\item on simultaneous equidistribution on pairs of quadrics,
\item on simultaneous equidistribution by genus classes, 
\item and on simultaneous supersingular reduction of CM elliptic curves.
\end{itemize}
Section \ref{sec2} uses Waldspurger's theorem and Parseval to reduce the proof of Theorems \ref{thm1} and \ref{thm2} to the proof of Theorem \ref{thm3}, a mean value estimate for fractional moments of twisted $L$-functions. It offers an independent consequence, stated as Corollary \ref{cor4}, on Bessel periods of Yoshida lifts. Before we start with the proof of Theorem \ref{thm3}, we give a heuristic argument in Section  \ref{sec5}. Section  \ref{sec6} compiles general results on $L$-functions. The combinatorial input of the proof of Theorem \ref{thm3} is provided in Section  \ref{sec7}, while the analytic input is the content of Sections  \ref{sec8} and  \ref{sec9}. \\

\noindent \textbf{Acknowledgements:} We would like to thank Gergely Harcos, Kimball Martin, and Abhishek Saha  for helpful comments.

\section{Converting from adelic to classical language}\label{sec:adelic-to-classical}

We begin by converting from the adelic language in which we have expressed Theorems \ref{thm1} and \ref{thm2} to more classical language. Throughout this section we shall give ourselves only \textit{one} quaternion algebra; concrete examples of Theorems \ref{thm1} and \ref{thm2} with two (distinct) quaternion algebras will be given in Section \ref{sec:examples}.

The following notation will be in place: let $\bB$ be a quaternion algebra over $\Q$. Let $\bPB^\times$ be its group of projective units. Let $\mathscr{O}\subset\bB(\Q)$ be a maximal order and write $\bPB^\times(\Z)$, resp. $\bPB^\times(\widehat\Z)$ for the projective unit group of $\mathscr{O}$, resp. $\widehat{\mathscr{O}}=\mathscr{O}\otimes\widehat\Z$.  Write $K=\bPB^\times(\widehat\Z)K_\infty$, where $K_\infty$ is a maximal compact torus of $\bPB^\times(\R)$.

\subsection{Viewing $[\bT]_K$ classically}\label{sec:ACG}

Let $E$ be quadratic field extension of $\Q$, with ring of integers $\mathcal{O}_E$, which is not split wherever $\bB$ is ramified. This condition assures that $E$ embeds into $\bB(\Q)$  \cite[Prop.\ 14.6.7]{Vo}. Let $\bT=({\rm Res}_{E/\Q}\mathbb{G}_m)/\mathbb{G}_m$ and let $\iota:\bT\hookrightarrow\bPB^\times$ be an optimal embedding of $\mathcal{O}_E$ into $ \mathbb{O}_\iota$.  Denote by $K_{\bT_\iota,\infty}$ the maximal compact subgroup of $\bT_\iota(\R)$.

The map 
\[
\bT_\iota(\A_{\Q})\rightarrow [\bT_\iota.{\rm g}]_K, \quad t\mapsto \bPB^\times(\Q) t{\rm g}K,
\]
induces a  bijection   $\bT_\iota(\Q)\bs\bT_\iota(\A_\Q)/(\bT_\iota(\A_\Q)\cap {\rm g}K{\rm g}^{-1}) \xrightarrow{\sim} [\bT_\iota.{\rm g}]_K$. 
By \eqref{Tf-optimality} and the hypothesis (which can be deduced from \eqref{conjugation-Tinfty}) that $\bT_\iota(\R)\cap g_\infty^{-1}K_\infty g_\infty=K_{\bT_\iota,\infty}$, the preceding adelic double quotient may be re-written as $\bT_\iota(\Q)\bs \bT_\iota(\A_\Q)/\iota(\bT(\widehat{\Z}))K_{\bT_\iota,\infty}$. 

We now observe that the latter group can naturally  be identified with the \textit{Arakelov class group} $\widetilde{\rm Cl}_E$, in the sense of \cite{EV}, of the ring of integers $\mathcal{O}_E$ of $E$. Indeed, we put $\widetilde\bT={\rm Res}_{E/\Q}\mathbb{G}_m$ and let $K_{\widetilde\bT,\infty}$ denote the maximal compact subgroup of $\widetilde\bT(\R)$ and recall that
$$
\widetilde{\rm Cl}_E=E^\times\bs\A_E^\times/\widehat{\mathcal{O}}_E^\times K_{\widetilde{\bT},\infty}$$
 is the usual class group ${\rm Cl}_E$ of $E$ if $E$ is imaginary, and  the extension of ${\rm Cl}_E$ by the circle $ \R^\times\bs E_{\infty}^\times/\mathcal{O}_E^{\times}$ if $E$ is real. We see that 
\begin{equation}\label{eq:cl-gp}
\widetilde{\rm Cl}_E=\widetilde\bT(\Q)\bs \widetilde\bT(\A_\Q)/\widetilde\bT(\widehat{\Z})K_{\widetilde\bT,\infty}\xrightarrow{\sim} \bT_\iota(\Q)\bs \bT_\iota(\A_\Q)/\iota(\bT(\widehat{\Z}))K_{\bT_\iota,\infty}  \cong [\bT_{\iota}.g]_K\end{equation}
the middle arrow being an isomorphism, since the kernel is $\Q^\times\bs\A_f^\times/\widehat{\Z}^\times$ and $\Q$ has class number 1. 

We now put a finite measure on $\widetilde{\rm Cl}_E$ whose volume behaves regularly in the discriminant of $E$. If $E$ is real, we have
\[
1\rightarrow \mu_2\rightarrow \bT(\R)=\R^\times\bs E_\infty^\times\xrightarrow{\log |x_1/x_2|}\R\rightarrow 1
\]
and $\mathcal{O}_E^\times\simeq \mu_2 \times \langle\log \epsilon\rangle$, where $\epsilon>1$ is a totally positive fundamental unit. We deduce an isomorphism between $\bT(\R)/\mathcal{O}_E^\times$ and $\R/\langle\log \epsilon\rangle$, with which we transport the Lebesgue measure on the circle of arclength $\log \epsilon$ to $\bT(\R)/\mathcal{O}_E^\times$. Using the counting measure on ${\rm Cl}_E$ we obtain a measure on $\widetilde{\rm Cl}_E$ with total volume
\[
{\rm vol}(\widetilde{\rm Cl}_E)=
\begin{cases}
|{\rm Cl}_E|, & E\; \textrm{ imaginary quadratic};\\
|{\rm Cl}_E| \log\epsilon, & E\; \textrm{ real quadratic}.
\end{cases}
\]
 Let $\eta_E$ be the quadratic character of conductor ${D}$ associated to $E/\Bbb{Q}$ by class field theory. By the Dirichlet class number formula we have
\begin{equation}\label{vol1}
\text{vol}(\widetilde{\rm Cl}_E) =cL(1, \eta_E) {D}^{1/2},
\end{equation}
where $c>0$ is an absolutely bounded constant depending only on the signature of $E$ at infinity and the number of roots of unity of $\mathcal{O}_E$. 

Since $\widetilde{\rm Cl}_E$ is compact, its dual $\widetilde{\rm Cl}_E^\vee$ is discrete: finite and equal to ${\rm Cl}^{\vee}_E$ if $E$ is imaginary, and infinite and isomorphic to ${\rm Cl}^{\vee}_E \times \Bbb{Z}$ if $E$ is real. \color{black}

\subsection{Viewing the discriminant $D$ classically}\label{sec:discriminant}
Next we show that the discriminant $D$ of the homogeneous torus subset $[\bT_\iota.{\rm g}]$ is the absolute value of the discriminant of $E$. We shall use an equivalent description of $D=\prod_v D_v$ from \cite[\S 4.2]{ELMV2} as described in \cite[\S 2.4.4]{Kh}.  

\medskip

\noindent\textit{The discriminant at finite places.}
 In this case $D_v$ is the discriminant of the maximal quadratic order $\iota(\mathcal{O}_{E,v})=\iota(E_v)\cap {\rm g}_v^{-1}\mathscr{O}_v {\rm g}_v$ inside $\iota(E_v)$, where we have used the optimality assumption \eqref{eq:optimal}. The latter discriminant is equal to  the discriminant of $\mathcal{O}_{E,v}$ inside $E_v$.

\medskip

\noindent\textit{The discriminant at the archimedean place.} 

\medskip

\noindent\textit{When $\bB(\R)$ is indefinite.} When $\bB(\R)$ is indefinite, we follow \cite[\S 6.1]{ELMV2}, which explicates the case of $\PGL_n(\R)$. We use the fixed isomorphism of $\bPB^\times(\R)$ with $\PGL_2(\R)$ from Section \ref{main} to identify the Lie algebra of $\bPB^\times(\R)$ with the quotient $\gf=M_2(\R)/\R$. Let $\|\cdot\|_\infty^2$ denote the norm on $\mathfrak{g}$ which descends from the norm ${\rm tr}(X^2)/2$ on $M_2(\R)$.

For any quadratic \'etale subalgebra $\tilde{\mathfrak{h}}$ of $M_2(\R)$ let $\{1,\tilde{f}\}$ be an $\R$-basis for $\tilde{\mathfrak{h}}$ which is orthonormal with respect to $\|\cdot\|_\infty$. Let $\mathfrak{h}$ be the image of $\tilde{\mathfrak{h}}$ in $\gf$ and $f$ the image of $\tilde{f}$ in $\gf$. Then we put $D_\infty(\mathfrak{h})=\|f\|_\infty^{-2}$. Note that when $\mathfrak{h}=\mathfrak{a}$ is the diagonal subalgebra of $\gf$ or when $\mathfrak{h}=\mathfrak{k}$ is the Lie algebra of ${\rm PSO}(2)$, then $D_\infty(\mathfrak{h})=1$.

 Following \cite[p.\ 841]{ELMV2} we put $D_\infty=D_\infty({\rm g}_\infty \iota(E_\infty){\rm g}_\infty^{-1})$. From \eqref{conjugation-Tinfty}, and the above remark, we see that $D_\infty=1$.

\medskip

 \noindent\textit{When $\bB(\R)$ is definite.} When $\bB(\R)$ is definite we follow the discussion in \cite[p.\ 166]{Kh}. We let $\|\cdot \|_\infty^2$ be the reduced norm ${\rm Nm}$ on $\bB(\R)$ and write
\[
\mathscr{O}_\infty=\{g \in\bB(\R): \|g\|_\infty\leq 1\}.
\]
We fix the volume form ${\rm vol}_\infty$ on ${\rm g}_\infty \iota(E_\infty){\rm g}_\infty^{-1}$ induced by the metric $|\cdot |_\infty$ for which ${\rm g}_\infty\bT(\R){\rm g}_\infty^{-1}$ acts by isometries, and normalized so that the unit disc is of volume 1. Let 
\[
\Lambda_\infty={\rm g}_\infty \iota(E_\infty) {\rm g}_\infty^{-1}\cap \mathscr{O}_\infty.
\]
Then we set ${D}_\infty={\rm vol}_\infty(\Lambda_\infty)^2$.

 We now calculate $D_\infty$, given the above choice of data. The essential point is that $K_\infty$ preserves $\|\cdot \|_\infty$, since it sits inside the projective image in $\bPB^\times(\R)$ of $\mathscr{O}_\infty^\times=\{g\in\bB^\times(\R): \|g\|_\infty=1\}$, which acts by conjugation on $\bB(\R)$ as the full group of orientation preserving isometries. 
From this it follows, using \eqref{conjugation-Tinfty}, that ${\rm g}_\infty\bT(\R){\rm g}_\infty^{-1}$ preserves $\|\cdot\|_\infty$, so that that the restriction of $\|\cdot \|_\infty$ to ${\rm g}_\infty \iota(E_\infty){\rm g}_\infty^{-1}$ is $|\cdot |_\infty$. Thus $\Lambda_\infty$ is the unit disc for $|\cdot |_\infty$, proving $D_\infty=1$.

\subsection{Viewing Duke's theorems classically}\label{sec:unions-quadrics} 

We now explicate Duke's theorems, converting from the adelic language of Section \ref{sec:M-V} to the classical arithmetic setting of integral points on (unions of) quadrics, as in the papers of Linnik \cite{Li1} and Skubenko \cite{Sk}. This will be helpful for generating examples of Theorems \ref{thm1} and \ref{thm2} in the next section.

Let $Q={\rm Nm}|_{\bB^0}$ be the restriction of the reduced norm ${\rm Nm}$ to the trace zero quaternions $\bB^0$. For a non-zero integer $d$ let $\mathbf{X}_{Q,d}$ denote the level set $\{x\in\bB^0: Q(x)=d\}$. This is an affine $\bPB^\times$-variety, under the action of conjugation.

Let $\mathbb{D}(Q)$ denote the set of non-zero integers which are everywhere locally integrally represented by $Q$, and let $\mathbb{D}_{\rm fund}(Q)=\mathbb{D}(Q)\cap\mathbb{F}$. For $d\in\mathbb{D}_{\rm fund}(Q)$ we let $E=\Q(\sqrt{-d})$ and choose an optimal embedding $\iota$ of  $\mathcal{O}_E$ into a maximal order $\mathbb{O}$ of $\bB$,  as in Section \ref{main}. Since $\iota$ preserves the trace and the norm, the point $$x_0=\iota(\sqrt{-d})$$ lies in $\mathbf{X}_{Q,d}(\Q)$. Note that the stabilizer of $x_0$ in $\bPB^\times$ consists of (projective) invertible elements of the form $a+bx_0$; this stabilizer is then seen to be $\bT_\iota=\iota(\bT)$, where $\bT=({\rm Res}_{E/\Q}\mathbb{G}_m)/\mathbb{G}_m$. By Witt's theorem we have $\mathbf{X}_{Q,d}(\Q)=\bPB^\times (\Q).x_0$, so that $\mathbf{X}_{Q,d}(\Q)$ is identified with the quotient $\bPB^\times (\Q)/\bT_\iota(\Q)$ through the orbit map on $x_0$.

We shall be interested in the distribution of the \textit{integral points} of $\mathbf{X}_{Q,d}$. To this end, let $\bPB^\times(\Q)\bs\bPB^\times(\A_f)/\bPB^\times(\widehat{\Z})$ be the \textit{class set} of $\bPB^\times$, for which we fix representatives $\{b_1=[e],\ldots ,b_h\}$. Then $\{ \mathbb{O}_i=b_i \mathbb{O} b_i^{-1}\cap\bB(\Q)\}$ forms a complete set of representatives for the $\bPB^\times(\Q)$-conjugacy classes of maximal orders of $\bB(\Q)$. When $Q$ is indefinite, Eichler's theorem (along with the fact that $\Q$ has class number one) implies that the class set of $\bPB^\times$ is a singleton. 

Define the lattice $\Lambda_i= \mathbb{O}_i\cap\bB^0(\Q)$ and let $d_i={\rm disc}(\Lambda_i)$ be its discriminant. The restriction of $d_i.Q$ to $\Lambda_i$ yields integral ternary quadratic forms $q_i$, forming a full $\bPB^\times$-genus class \cite[Ch.\ 22]{Vo}, which we denote by ${\rm Gen}(Q)$. Let $\mathscr{R}_d(q_i)=\bPB^\times (\Q).x_0\cap \Lambda_i$ be the $\Lambda_i$-integral points of $\mathbf{X}_{Q,d}$, the set of all $\Lambda_i$-integral representations of $d$ by $q_i$. Considering them all together yields
\[ 
\mathscr{R}_d({\rm gen}_Q)=\coprod_{q_i\in {\rm Gen}(Q)} \mathscr{R}_d(q_i).
\]

The \textit{equidistribution problems of Linnik's type} are the study of the distribution of the projection $|d|^{-1/2}\mathscr{R}_d({\rm gen}_Q)$ on the union of quadrics
\begin{equation}\label{union-of-quadrics}
\coprod_{{\rm Gen}(Q)} \mathbf{X}_{Q,{\rm sgn}(d)}(\R),
\end{equation}
as $d\rightarrow\pm\infty$ in $\mathbb{D}_{\rm fund}(Q)$, where ${\rm sgn}(d)$ is $+1$ or $-1$ according to the sign of $d$. To formulate this precisely one must prescribe the relevant measures. The quadric $\mathbf{X}_{Q,{\rm sgn}(d)}(\R)$ has a unique (up to non-zero scaling) $\bPB^\times(\R)$-invariant measure coming from its structure as a homogeneous space for $\bPB^\times(\R)$. When $Q$ is indefinite, $\mathbf{X}_{Q,{\rm sgn}(d)}(\R)$ is non-compact and the measure will then be infinite. In this case, we let $\mu_{Q,{\rm sgn}(d)}$ be any such choice, and equidistribution relative to this measure is taken in the sense of \cite[(1.1)]{MV}. In the definite case we let $\mu_{Q,{\rm sgn}(d)}$ be the probability measure which assigns to the $i$th copy of $\mathbf{X}_{Q,{\rm sgn}(d)}(\R)$ the $\bPB^\times(\R)$-invariant measure with volume $|{\rm Aut}(q_i)|^{-1}$.

Then Duke's theorem \cite{Du, DSP} states that $d^{-1/2}\mathscr{R}_d({\rm gen}_Q)$ equidistributes on \eqref{union-of-quadrics} relative to $\mu_{Q,{\rm sgn}(d)}$ as $d\rightarrow\pm\infty$ in $\mathbb{D}_{\rm fund}(Q)$. In particular (in the definite case), every large enough $d\in\mathbb{D}_{\rm fund}(Q)$ is integrally represented by every genus member, solving (up to issues of effectivity) the last remaining case of Hilbert's 11th problem over $\Q$ (see \cite{BH} for the number field case). 

\subsection{Finiteness of equivalence classes}

In this subsection we explicate the action of the class group of $\bT_\iota$ on certain equivalence classes of integral representations $\mathscr{R}_d({\rm gen}_Q)$.

We write $\Gamma=\bPB^\times(\Z)$. Then ${\rm Gen}(\Gamma)=\{\Gamma_1=\Gamma,\ldots ,\Gamma_h\}$, where $\Gamma_i$ denote the projective units of the maximal order $\mathscr{O}_i$, is the \textit{genus class} of $\Gamma$. Since $\Lambda_i$ is stable under the action of $\Gamma_i$ by conjugation, we may form the quotient $\mathscr{R}^*_d(q_i)=\Gamma_i\backslash\mathscr{R}_d(q_i)$ and put
\[
\mathscr{R}_d^*({\rm gen}_Q)=\coprod_{q_i\in {\rm Gen}(Q)} \mathscr{R}^*_d(q_i). 
\]
We now give an adelic parametrization of $\mathscr{R}_d^*({\rm gen}_Q)$, for $d\in\mathbb{D}_{\rm fund}(Q)$. This will show, in particular, that $\mathscr{R}_d^*({\rm gen}_Q)$ is finite, a fact which is not \textit{a priori} evident when $\bB$ is indefinite.

To this end, we introduce purely local analogues of the global problem of parametrizing $\Lambda$-integral points on $\mathbf{X}_d$, where $\Lambda=\Lambda_1$. Namely, with $q=q_1$ and $\widehat{\Lambda}=\Lambda\otimes_\Z\widehat\Z$, we put
\[
\bPB^\times(\A_f)^{\rm int}=\{g\in\bPB^\times(\A_f): g.x_0\in\widehat{\Lambda}\}.
\]
In this way, the local analogues $\bPB^\times(\A_f).x_0\cap \widehat{\Lambda}$ and $\bPB^\times(\widehat{\Z})\bs (\bPB^\times(\A_f).x_0\cap \widehat{\Lambda})$ are identified, under the orbit map through $x_0$, with the group quotients
\[
\bPB^\times(\A_f)^{\rm int}/\bT_\iota(\Q)\quad\textrm{ and }\quad\bPB^\times(\widehat\Z)\bs \bPB^\times(\A_f)^{\rm int}/\bT_\iota(\Q),
\]
respectively. To compute the latter, one must determine the orbit structure of  $\bG(\Z_p)=\mathscr{O}^\times_p/\{\pm 1\}$ acting by conjugation on the $\Z_p$-solutions $\mathbf{X}_d(\Z_p)=\{x_p\in\bB^0(\Z_p): {\rm Nm}(x_p)=d\}$ for every $p$. This has been done in \cite[Prop 10.1]{EMV}. In particular, for almost all $p$ the action is transitive.

With this notation in place, it follows from \cite[\S 3.2.2]{EV08} that $\mathscr{R}_d^*({\rm gen}_Q)$ is naturally identified with
\[
\bigcup_{[{\rm h}_f]\in \bPB^\times(\widehat\Z)\bs \bPB^\times(\A_f)^{\rm int}/\bT_\iota(\Q)} \bT_\iota(\Q) \bs \bT_\iota(\A_f)/\left(\bT_\iota(\A_f)\cap {\rm h}_f^{-1}\bPB^\times(\widehat{\Z}){\rm h}_f \right).
\]
Each of the above quotients is a finite group, being the class group of a quadratic order in $E$.

\begin{remark}
The lack of non-transitivity (for certain congruence classes mod   $d$) of the class group action, as described in the introduction in the context of Linnik's sphere problem, is ``explained'' by the local parametrizing set $\mathscr{R}^*_d(x^2+y^2+z^2)_{\rm loc}$. In many other sources, such as \cite{EMV}, the integral structure for $\bPB^\times$ is induced by the orthogonal group for the maximal order, rather than its projective units, via the $\Q$-isomorphism of $\bPB^\times$ with $\SO_Q$. The resulting quotient then receives a transitive action of ${\rm Cl}_E$ for \textit{all} congruence classes, but possibly with non-trivial stabilizers of finite order.
\end{remark}

Note that condition \eqref{eq:optimal} implies $x_0\in \mathbf{X}_{Q,d}(\Q)\cap {\rm g}_f^{-1} \widehat{\Lambda} {\rm g}_f$. Thus ${\rm g}_f. x_0 \in {\rm g}_f\mathbf{X}_{Q,d}(\Q){\rm g}_f^{-1}\cap  \widehat{\Lambda}$, so that ${\rm g}_f\in \bPB^\times(\A_f)^{\rm int}$. Similarly to \eqref{eq:cl-gp}, we may identity ${\rm Cl}_E$ with $\bT_\iota (\Q) \bs \bT_\iota (\A_f)/\bT_\iota(\widehat\Z)$, where $\bT_\iota(\widehat\Z)$ satisfies \eqref{Tf-optimality}. Taking $[{\rm h}_f]=[{\rm g}_f]$, we get an embedding
\begin{equation}\label{class-gp-embedding}
{\rm Cl}_E\rightarrow \mathscr{R}_d^*({\rm gen}_Q),
\end{equation}
given by the orbit map through $x_0$.

\subsection{The dual picture and the modular formulation of Duke's theorem}\label{sec:dual-picture}
When $\bB$ is indefinite, it is more convenient to formulate (and prove!) Duke's theorem using a dual formulation \cite{MV} involving packets of Heegner points or closed geodesics on a Shimura or modular curve. 

In this section we shall take $\bB$ indefinite, in which case   the genus class of $Q$ is a singleton.  In particular, one can assume that $\mathbb{O}=\mathscr{O}$ and the element ${\rm g}_f$ from Section \ref{main} can be taken to be the identity.  Let $H<\bPB^\times(\R)$ be either $K_\infty$ or the pullback of the projective diagonal matrices via the fixed isomorphism of $\bB(\R)$ with $M_2(\R)$ from Section \ref{main}. We fix $x_\infty\in\mathbf{X}_{Q,{\rm sgn}(d)}(\R)$ such that $H={\rm Stab}_{\bPB^\times(\R)}(x_0)$. We then identify $\mathbf{X}_{Q,{\rm sgn}(d)}(\R)$ with $\bPB^\times(\R)/H$ via the orbit map through $x_\infty$. The image of \eqref{class-gp-embedding} is given by a finite union of $\bPB^\times(\Z)$-orbits
\[
\coprod_{\tau\in {\rm Cl}_E} \bPB^\times(\Z) x_\tau.
\]
We project these onto the quadric $\mathbf{X}_{Q,{\rm sgn}(d)}(\R)$ by rescaling by $|d|^{-1/2}$. In this way we produce elements $g_\tau\in\bPB^\times(\R)$ satisfying $g_\tau.x_\infty=|d|^{-1/2} x_\tau$. 

The following equidistribution statements are equivalent \cite[Prop. 2.1]{ELMV1}:
\begin{enumerate}
\item \textit{(arithmetic equidistribution statement)} the finite union of left $\bPB^\times(\Z)$-orbits
\[
\coprod_{\tau\in {\rm Cl}_E} \bPB^\times(\Z) g_\tau H/H
\]
equidistribute on the quadric $\mathbf{X}_{Q,{\rm sgn}(d)}(\R)$;
\item\label{modular-statement} \textit{(modular equidistribution statement)} the finite set of right $H$-orbits
\[
\left\{ y_\tau H: \tau\in {\rm Cl}_E\right\}, \quad\text{where}\quad y_\tau=\bPB^\times(\Z)\bs \bPB^\times(\Z) g_\tau,
\]
equidistribute on $\bPB^\times(\Z)\bs\bPB^\times(\R)$.
\end{enumerate}

We remark that by \cite[\S 2.4.1]{ELMV3} the right $H$-orbits in \eqref{modular-statement} are periodic.

Let us now examine the modular equidistribution statement \eqref{modular-statement}. Let $\mathbb{H}^\pm= \mathbb{H}\cup -\mathbb{H}$ be the union of the upper and lower half-planes. We identify $\bPB^\times(\R)$ with $\PGL_2(\R)$ as in Section \ref{sec:discriminant} so that
\[
\bPB^\times(\R)/K_\infty\simeq \PGL_2(\R)/{\rm PSO}(2)\simeq \mathbb{H}^\pm.
\]
Then, with $K=\bPB^\times(\widehat{\Z}) K_\infty$, we have
\[
[\bPB^\times]_K=\bPB^\times(\Q)\bs\bPB^\times(\A_\Q)/K=\bPB^\times(\Z)\bs\bPB^\times(\R)/K_\infty\simeq X_\Gamma,
\]
where $X_\Gamma=\bPB^\times(\Z)\bs\mathbb{H}^\pm$ is a Shimura curve when $\bB$ is a division algebra and the modular curve $\PGL_2(\Z)\bs\mathbb{H}^\pm={\rm PSL}_2(\Bbb{Z}) \backslash \Bbb{H}$ when $\bB$ is the matrix algebra. In the modular statement, we may identify the quotient $\bPB^\times(\Z)\bs\bPB^\times(\R)$ with the unit tangent bundle $T^1(X_\Gamma)$ of $X_\Gamma$, equipped with the Liouville measure. However, in view of the archimedean restrictions in Theorems \ref{thm1} and \ref{thm2} we wish rather to examine \eqref{modular-statement} on $X_\Gamma$ itself.\footnote{The stronger version of Duke's original theorem, which upgrades $X_\Gamma$ to $T^1(X_\Gamma)$, was proved by Chelluri in \cite{Che}.} In this case, $X_\Gamma$ is equipped with the Poincar\'e measure, normalized to have volume one.

Now when $H=K_\infty$ is compact (so that $E=\Q(\sqrt{-d})$, with $d>0$, is imaginary quadratic), \eqref{modular-statement} is equivalent to the equidistribution of the \textit{Heegner points}
\[
\mathscr{H}_d(X_\Gamma)=\left\{z_\tau=y_\tau K_\infty/K_\infty: \tau\in {\rm Cl}_E\right\}
\]
on $X_\Gamma$. For example, when $\bB=M_2$ is the matrix algebra so that $Q(a,b,c)=b^2 - 4ac$ is the discriminant form, we obtain the Heegner points 
\[
\mathscr{H}_d({\rm PSL}_2(\Bbb{Z}) \backslash \Bbb{H})={\rm PSL}_2(\Bbb{Z})\backslash \Big \{ \frac{-b + \sqrt{-d}}{2a} : b^2 \equiv  -d\, (\text{mod } 4a)\Big\}
\]
on the modular surface ${\rm PSL}_2(\Bbb{Z}) \backslash \Bbb{H}$. When $H$ is non-compact (so that $E=\Q(\sqrt{-d})$, with $d<0$, is real quadratic), \eqref{modular-statement} states that the packet of \textit{closed geodesics}
\[
\mathscr{G}_d(X_\Gamma)=\left\{\gamma_\tau=y_\tau H/K_\infty: \tau\in {\rm Cl}_E\right\}
\] 
equidistributes on $X_\Gamma$. Taking the modular surface again as our example, each $\gamma_\tau$ is the projection to ${\rm PSL}_2(\Bbb{Z}) \backslash \Bbb{H}$ of the unique geodesic in $\mathbb{H}$ with endpoints $\frac{-b + \sqrt{-d}}{2a}$, where $b^2 \equiv  -d\, (\text{mod } 4a)$ cf. \cite[\S 2.3]{ELMV3}.

\section{Instances of Theorems \ref{thm1} and \ref{thm2}}\label{sec:examples}

We now give explicit examples of pairings which fit into the simultaneous equidistribution statement of our Theorems \ref{thm1} and \ref{thm2}.

\medskip

\noindent {\sc Example 1:} \emph{Simultaneous equidistribution on pairs of quadrics.}  We begin by giving an example with class number one algebras; we shall emphasize equidistribution across genus classes shortly.

Let $\bB_1$ be the split algebra $M_2$ and $\bB_2=\bB^{(2,\infty)}$. Any imaginary quadratic extension $E$ of $\Q$ of discriminant $-D$ with $D \not\equiv 7$ (mod 8) embeds diagonally into $\bB_1\times\bB_2$. Theorem \ref{thm1} states that, under GRH, the corresponding packets of pairings, as in \eqref{joint-orbit}, between Heegner points $\mathscr{H}_D$ on the modular curve and the integral points $D^{-1/2} \mathscr{R}_D$ on the sphere (cf.\ \eqref{11}) equidistribute, relative to the discrete spectrum,  to the product of the uniform probability measures on ${\rm SO}^+_3(\Bbb{Z})\backslash S^2\times {\rm PSL}_2(\Z)\bs\mathbb{H}$ as $D\rightarrow\infty$.

 If we replace $\bB_1 = M_2$ with an indefinite division algebra, we obtain the same result with Heegner points on a Shimura curve. Since the automorphic spectrum of the latter is discrete, the restriction to test functions in the discrete spectrum in Theorem \ref{thm1} holds automatically in this case.

\medskip

\noindent {\sc Example 2:} \emph{A variation.}  The previous example has an interesting geometric variation, treated in \cite{AES}. To each point in $ x\in \mathscr{R}_D$ we can associate the rank 2 lattice $x^\perp\cap\Z^3$; after rotating to a fixed reference plane and rescaling by the volume, it can be viewed as a Heegner point of discriminant $D$ via the isomorphism of ${\rm PSL}_2(\Z)\bs\mathbb{H}$ with isometry classes of unimodular lattices in $\R^2$. As explained in \cite[Section 5.2]{EMV}, this corresponds to the twisted diagonal embedding $\iota_{\alpha, \beta}$ with $\alpha = 1$, $\beta = 2$; indeed, the projection of the image onto the second factor meets only one coset of squares in the class group (cf.\  footnote 11 on p.\ 151 of \textit{loc.\ cit.}). We may deduce from Theorem \ref{thm2} that, under GRH, the set of pairs $$(x/\|x\|,x^\perp\cap\Z^3)$$ equidistributes (relative to the discrete spectrum) to the product of the uniform measures on ${\rm SO}^+_3(\Bbb{Z})\backslash S^2\times {\rm PSL}_2(\Z)\bs\mathbb{H}$ as $D\rightarrow\infty$ through prime discriminants (so that by Gau{\ss}' genus theory the 2-torsion is trivial). 

\begin{remark}
This orthogonal complement construction works in greater generality. For instance, as explained in Section \ref{sec:dual-picture}, if $\bB=M_2$, a primitive integral element $x$ on the discriminant variety $\mathbf{X}_{Q,d}$, associated with the determinant form $Q(a,b,c)=b^2-4ac$, corresponds to a Heegner point or a closed geodesic on the modular curve, according to whether $d<0$ or $d>0$, respectively. The restriction of $Q$ to $\langle x\rangle^{\perp}$ then has signature $(2, 0)$ if $d < 0$ and signature $(1, 1)$ if $d > 0$, yielding a positive, respectively indefinite, binary lattice of discriminant $d$. 
\end{remark}

\medskip

\noindent {\sc Example 3:} \emph{Simultaneous representation by genus classes.} 
 Next we assume $\bB$ is definite and let $K_\infty=\bPB^\times(\R)$. In this case the adelic double quotient $[\bPB^\times]_K$ is simply a finite union of singletons, indexed by the class set of $\bPB^\times$. We give two  very concrete examples of Theorem \ref{thm1} in this setting in this and the following subsection. 

Let $\bB_1=\bB^{(11,\infty)}$ and $\bB_2=\bB^{(19,\infty)}$, both genus 2 definite quaternion algebras. For $i=1,2$, let $Q_i={\rm Nm}|_{\bB_i^0}$ be the trace zero norm form and choose two non-conjugate maximal orders $ \mathbb{O}_i^{(1)},  \mathbb{O}_i^{(2)}$ in $\bB_i(\Q)$. Restricting $Q_i$ to the trace zero dual lattices $\Lambda_i^{(1)}, \Lambda_i^{(2)}$ yields, as in Section \ref{sec:unions-quadrics}, a pair of genus equivalent,  integral, positive definite ternary quadratic forms $Q_i^{(1)}, Q_i^{(2)}$. Explicitly, we have
\[
{\rm gen}_{Q_1}=\{Q_1^{(1)},Q_1^{(2)}\}:\quad\begin{cases}
Q_1^{(1)}=3x^2+11y^2+xz+z^2,& |{\rm Aut}(Q_1^{(1)})|=4;\\
Q_1^{(2)}=3x^2+4y^2+4z^2+2xy+2xz - 3 yz,& |{\rm Aut}(Q_1^{(2)})|=6,
\end{cases}
\]
and
\[
{\rm gen}_{Q_2}=\{Q_2^{(1)},Q_2^{(2)}\}:\quad\begin{cases}
Q_2^{(1)}=x^2+5y^2+19z^2 +xy,& |{\rm Aut}(Q_2^{(1)})|=8;\\
Q_2^{(2)}=4x^2+5y^2+6z^2 +5yz + xz + 2xy,& |{\rm Aut}(Q_2^{(2)})|=4.
\end{cases}
\]
We define probability measures $m_i$ on ${\rm gen}_{Q_i}$ by weighting each form by the reciprocal of the order of its automorphism group. Thus
\[
m_1=\frac{3}{5}\delta_{Q_1^{(1)}}+\frac{2}{5}\delta_{Q_1^{(2)}}\quad\text{and}\quad m_2=\frac{1}{3}\delta_{Q_2^{(1)}}+\frac{2}{3}\delta_{Q_2^{(2)}}.
\]

It follows from the theorem of Duke and Schulze-Pillot \cite{DSP} that any large enough $d\in\mathbb{D}_{\rm fund}(Q_i)$ is integrally represented both by $Q_i^{(1)}$ and $Q_i^{(2)}$, and that the relative share of such representations as $d\rightarrow\infty$ is governed by $m_i$. Theorem \ref{thm1} states that, under GRH, any large enough $d\in\mathbb{D}_{\rm fund}(Q_1)\cap\mathbb{D}_{\rm fund}(Q_2)$ is represented \textit{simultaneously} by every pair $(Q_1^{(i)},Q_2^{(j)})$, and that the relative share of such representations as $d\rightarrow\infty$ is governed by the product measure $m_1\times m_2$.

\medskip

\noindent {\sc Example 4:} \emph{Simultaneous supersingular reduction of CM elliptic curves.} While the above example featured two explicit definite quaternion algebras ramified at a single prime, this was only for computational simplicity. By contrast, the next example exploits the special arithmetic significance of the quaternion algebra $\bB^{(p,\infty)}$. 

In this case, the class set of $\bB^{(p,\infty)}$ can be identified with the set $\mathscr{E}_p^{\rm ss}$ of isomorphism classes of supersingular elliptic curve defined over $\overline{\Bbb{F}}_p$. Under this identification, $\mathscr{E}_p^{\rm ss}$ is endowed with a natural probability measure $m_p$ which, as in the previous example assigns to each $e\in \mathscr{E}_p^{\rm ss}$ a weight proportional to $| \text{End}(e)^\times|^{-1}$.

Let $E$ be an imaginary quadratic with ring of integers $\mathcal{O}_E$ such that $p$ is inert in $E$; then $E$ embeds into $\bB^{(p,\infty)}$. Let ${\rm Ell}_{\mathcal{O}_E}^{\rm CM}$ be the set of elliptic curves defined over $\Bbb{C}$ with complex multiplication by $\mathcal{O}_E$. Any $e\in {\rm Ell}_{\mathcal{O}_E}^{\rm CM}$ is defined over the Hilbert class field $H_E$. Fix a prime $\mathfrak{p} \mid p$ of $H_E$. Then the reduction of $e\in{\rm Ell}_{\mathcal{O}_E}^{\rm CM}$ modulo $\mathfrak{p}$ is a supersingular elliptic curve defined over $\overline{\Bbb{F}}_p$ \cite[p.\ 41]{De}. In this way we obtain a reduction map
\[
{\rm red}_p: {\rm Ell}_{\mathcal{O}_E}^{\rm CM}\rightarrow  \mathscr{E}_p^{\rm ss},\qquad e\mapsto e\!\!\!\mod \mathfrak{p}.
\]

It was observed by Michel \cite{Mi} that this example fits into the paradigm of Duke's theorems of the previous examples. (We refer to the recent preprint \cite{ALMW} for a detailed description of the relation with toric packets.) Namely, for fixed $p$, the map ${\rm red}_p$ is surjective if the discriminant of $E$ is large enough, and in fact one has the equidistribution statement
\[
\frac{|\{e \in{\rm Ell}_{\mathcal{O}_E}^{\rm CM}: \text{red}_p(e) = e_0\}|}{|{\rm Ell}_{\mathcal{O}_E}^{\rm CM}|} \rightarrow m_p(e_0)\qquad (e_0\in \mathscr{E}_p^{\rm ss}),
\]
as $E$ varies along the imaginary quadratic fields for which $p$ is inert. 

Having reprised Duke's theorem in this setting, we now take an imaginary quadratic extension $E$ of $\Q$, in which two distinct odd primes $p,q$ are inert. Then $E$ embeds diagonally into the product $\bB^{(p,\infty)}\times\bB^{(q,\infty)}$. We may then consider the simultaneous supersingular reduction map 
\[
{\rm red}_{p,q}: {\rm Ell}_{\mathcal{O}_E}^{\rm CM}\rightarrow\mathscr{E}_p^{\rm ss}\times  \mathscr{E}_q^{\rm ss},\qquad e\mapsto (e\!\!\!\mod \mathfrak{p},e\!\!\!\mod \mathfrak{q}).
\]
Our main result, Theorem \ref{thm1}, shows that, under GRH, ${\rm red}_{p,q}$ is surjective for $E$ of large enough discriminant $D$, and that the pushforward of the counting measure on ${\rm Ell}_{\mathcal{O}_K}^{\rm CM}$ tends to the product measure $m_p\times m_q$ on $\mathscr{E}_p^{\rm ss}\times  \mathscr{E}_q^{\rm ss}$ as $D\rightarrow\infty$. Note that, in contrast to the spectral approach we have adopted, dynamical methods can handle $n\geq 2$ distinct copies of $\bB_i$.  

\section{Reduction to half-integral mixed moment}\label{sec2}

In this subsection we reduce the proofs of Theorem \ref{thm1} and \ref{thm2} to statements about a half-integral mixed moment of $L$-functions. Throughout, we retain the notation and hypotheses of Theorems \ref{thm1} and \ref{thm2}. 
For $f \in L^2([\bG]_{K})$ let 
\begin{equation}\label{period}
\mathcal{P}_{\Delta\bT}(f)=\int_{[\Delta\bT_\iota]} f(t) dt\qquad\text{and}\qquad \mathcal{P}_{\Delta_{\alpha, \beta}\bT}(f)=\int_{[\Delta_{\alpha, \beta}\bT_\iota]} f(t) dt
\end{equation}
be the global toric period integral relative to the subgroups $\Delta\bT_\iota\subset\bG$ and $\Delta_{\alpha, \beta}\bT_\iota\subset\bG$. When $\alpha=\beta=1$ these two subgroups and their corresponding periods coincide. We recall that the measures on $[\Delta\bT_\iota]$ and $[\Delta_{\alpha,\beta}\bT_\iota]$ are normalized to have volume 1. 

Let $L^2_0([\bG]_K)$ denote the orthocomplement of the character spectrum in $L^2_{\rm disc}([\bG]_K)$. Since $K_j$ is the projective unit group of the maximal order $\mathscr{O}_j$, the character spectrum of $K_j$, and thus of $K=K_1\times K_2$, is trivial (see \cite[\S 9.2-9.3]{ALMW}, where this is shown to hold for Eichler orders). 

\subsection{Setting up Weyl's criterion}\label{sec:Weyl}

It suffices, for proving Theorems \ref{thm1} and \ref{thm2}, to study the periods \eqref{period} when $f$ is an element of a given orthonormal basis of $L^2_0([\bG]_K)$. We now describe a particularly nice such orthonormal basis, coming from the theory of new vectors for inner forms of $\GL_2$.

We begin with the group $\bG_j$, for $j=1,2$. Let $L^2_0([\bG_j]_{K_j})$ denote the orthocomplement of the trivial character $1_{\bG_j}$ in $L^2_{\rm disc}([\bG_j]_{K_j})$. We have the Hilbert space direct sum decomposition 
\begin{equation}\label{L2-decomp-Gj}
L^2_0([\bG_j]_{K_j})=\bigoplus_{\sigma_j\subset L^2_0([\bG_j])}\sigma_j^{K_j}
\end{equation}
into irreducible discrete automorphic subrepresentations having non-trivial invariants under $K_j=\bG_j(\widehat{\Z})K_{j,\infty}$, where we have used the Multiplicity One theorem. Since $K_j$ is a maximal compact subgroup at each finite place, and either a maximal compact or maximal compact proper subgroup (if $\bB$ is definite) at infinity, we have $\dim \sigma_j^{K_j}\leq 1$. For each $\sigma_j$ appearing in the  decomposition \eqref{L2-decomp-Gj}, there is therefore a unique up to unitary scaling choice of an $L^2$-normalized vector $\phi_{\sigma_j}$ in the line of $K_j$-invariants, and we let $\mathcal{B}_{0,j}=\{\phi_{\sigma_j}\}$ be the resulting orthonormal basis of $L^2_0([\bG_j]_{K_j})$.

Note that, in the case of $\bB_j$ definite and $K_{j,\infty}$ a maximal  {compact torus}, each $\phi_{\sigma_j}$ is an eigenfunction for the sphere Laplacian with an eigenvalue of the form $k(k+1)$. When $\bB_j$ is indefinite, each $\phi_{\sigma_j}$ is an eigenfunction for the hyperbolic Laplacian of the form $1/4 + t^2$. In either case, we denote the Laplacian eigenvalue by  $\lambda_{\sigma_j}^2$ and call $\lambda_{\sigma_j}$ the \emph{spectral parameter}. Note that $\lambda_{\sigma_j} > 1/3$ is bounded away from 0  by \cite{KS} and comparable in size to the archimedean conductor of $\sigma_j$. In the remaining case where $\bB_j$ is definite and $K_{j,\infty}=\bG_j(\R)$, we simply put $\lambda_{\sigma_j}=1$. 

We now return to the product group $\bG=\bG_1\times\bG_2$. Let $L^2_{00}([\bG]_K)=L^2_0([\bG_1]_{K_1})\otimes  L^2_0([\bG_2]_{K_2})$. Then
\[
L^2_0([\bG]_K)=\big(\C 1_{\bG_1}\otimes L^2_0([\bG_2]_{K_2})\big)\bigoplus \big(L^2_0([\bG_1]_{K_1}\otimes \C 1_{\bG_2}\big) \bigoplus L^2_{00}([\bG]_K).
\]
We deduce that $L^2_0([\bG]_K)$ admits an orthonormal basis of the form $\mathcal{B}_0=\mathcal{B}_{01}\bigcup \mathcal{B}_{02}\bigcup \mathcal{B}_{00}$, where
\[
\mathcal{B}_{01}=\mathcal{B}_{0,1}\otimes 1_{\bG_2},\qquad \mathcal{B}_{02}=1_{\bG_1}\otimes\mathcal{B}_{0,2},\quad\text{and}\quad \mathcal{B}_{00}=\mathcal{B}_{0,1}\otimes\mathcal{B}_{0,2}.
\]
For each $\phi\in\mathcal{B}_0$ we let $\lambda_\phi$ denote $\lambda_{\sigma_j}$ or $\lambda_{\sigma_1}\lambda_{\sigma_2}$ according to whether $\phi\in\mathcal{B}_{0j}$ or $\phi\in\mathcal{B}_{00}$.

Now let $f \in C^{\infty}_c([\bG]_K)\cap L^2_0([\bG]_K)$ with $L^2$-spectral expansion
\[
f=\sum_{\phi\in\mathcal{B}_0}\langle f,\phi\rangle \phi.
\]
This expansion is finite whenever both $\bB_j$ are definite and $K_{j,\infty}=\bG_j(\R)$. Otherwise the expansion is absolutely convergent, and by self-adjointness of the Laplace operator, the $L^2$-inner products verify $\langle f,\phi\rangle \ll_d \mathcal{S}_{\infty, 2d}(f)\lambda_\phi^{-d}$ for all $d\in\N$. In all cases, we deduce that Theorems \ref{thm1}  and \ref{thm2} follow from 
\begin{equation}\label{period-bd2}
\mathcal{P}_{\Delta_{\alpha, \beta}\bT}( {\rm g}.\phi)\ll (\log {D})^{-1/4+\varepsilon} 
\end{equation}
for all $\phi \in \mathcal{B}_0$, where the implied constant may depend on $\alpha, \beta, \varepsilon$ and \emph{polynomially} on $\lambda_\phi$. Note that 
\[
\mathcal{P}_{\Delta_{\alpha,\beta}\bT}( {\rm g}.(\phi_{\sigma_1}\otimes 1_{\bG_2}))=\int_{[{\bT}]}\phi_{\sigma_1}(t^{\alpha}{\rm g}_1)dt\quad\textrm{and}\quad \mathcal{P}_{\Delta_{\alpha,\beta}\bT}( {\rm g}.(1_{\bG_2}\otimes\phi_{\sigma_2}))=\int_{[{\bT}]}\phi_{\sigma_2}(t^{\beta}{\rm g}_1)dt.
\]
Thus the bound \eqref{period-bd2} for test functions in $\mathcal{B}_{0j}$ is covered by Duke's theorem as stated in Section \ref{sec:M-V}, even with a power saving rate, so that it suffices to prove \eqref{period-bd2} on the basis elements in $\mathcal{B}_{00}$.

\subsection{Parseval and absolute values}\label{sec:Parseval}

For $\phi_j \in L^2([\bG_j]_{K_j})$ and $\chi \in \widetilde{\rm Cl}_E^\vee$ we define 
\begin{equation}\label{twisted}
\mathcal{P}_{\bT}^\chi (\phi_j)=  \int_{[\bT_{\iota_j}]} \phi_j(t)  \chi(t) dt 
\end{equation}
to be the global $(\bT_{\iota_j},\chi)$-period  integral. We now convert the problem of estimating of the diagonal period $\mathcal{P}_{\Delta_{\alpha,\beta}\bT}$ to one of bounding an average of these twisted periods. 

In the context of Theorem \ref{thm1} this is a straightforward task.  We recall from Section \ref{sec:ACG} that $[\bT_{\iota_j}]_{K_j}$ for $j= 1, 2$ can naturally be identified with the Arakelov class group $\widetilde{\text{Cl}}_E$. Using the identification \eqref{eq:cl-gp}, we may view the integral over $[\Delta\bT_\iota]_K$ in \eqref{period} as an inner product over the Arakelov class group $\widetilde{\rm Cl}_E$. Plancherel's identity then gives
\begin{equation}\label{planch}
\mathcal{P}_{\Delta\bT}(\phi_\sigma^\circ)=\int_{[\bT]} \phi_{\sigma_1}^\circ( \iota_1(t)) \overline{\phi_{\sigma_2}^\circ(\iota_2(t))} dt= \sum_{\chi \in \widetilde{\rm Cl}_E^\vee} \mathcal{P}_{\bT}^\chi(\phi_{\sigma_1}^\circ) \overline{\mathcal{P}_{\bT}^\chi(\phi_{\sigma_2}^\circ)},
 \end{equation}
 where $\phi_\sigma^\circ={\rm g}.\phi_\sigma$. More generally, working in the context of Theorem \ref{thm2}, we apply Fourier inversion to obtain
\begin{displaymath}
\begin{split}
\mathcal{P}_{\Delta_{\alpha, \beta}\bT}(\phi_\sigma^\circ) = \int_{[\bT]} \Big(\sum_{\chi \in \widetilde{\rm Cl}_E^\vee} \mathcal{P}_{\bT}^\chi (\phi_{\sigma_1}^\circ) \overline{\chi^{\alpha}(t)}\Big)
\overline{\phi_{\sigma_2}^\circ(\iota_2(t)^{\beta})} dt = \sum_{\chi \in \widetilde{\rm Cl}_E^\vee} \mathcal{P}_{\bT}^\chi (\phi_{\sigma_1}^\circ) \int_{[\bT]}  \overline{\phi_{\sigma_2}^\circ(\iota_2(t)^{\beta})   \chi^{\alpha}(t)} dt. 
\end{split}
\end{displaymath}
Since by assumption $ {\rm Cl}_E$ has no $\beta$-torsion\footnote{The assumption of no $2$-torsion is used later, in Remark \ref{rem6}.}, the integral vanishes unless $\chi^{\alpha} = \psi^\beta$ for some $\psi \in \widetilde{\rm Cl}_E^\vee$ (which is unique). 
Writing $\alpha = \alpha'\delta$, $\beta = \beta'\delta$ with $(\alpha', \beta') = 1$, we see that $\chi = \psi^{\beta'}$ for some $\psi \in \widetilde{\rm Cl}_E^\vee$, so $\chi^{\alpha}(t) = \psi^{ \alpha'}(t^{\beta})$. 
Changing variables $t \leftarrow t^{\beta}$ in the $t$-integral, we obtain
\[
\mathcal{P}_{\Delta_{\alpha, \beta}\bT}(\phi_\sigma^\circ)=\sum_{\chi \in \widetilde{\rm Cl}_E^\vee} \mathcal{P}_{\bT}^{\chi^{\beta'}}(\phi_{\sigma_1}^\circ) \overline{\mathcal{P}_{\bT}^{\chi^{\alpha'}}(\phi_{\sigma_2}^\circ)},
\]
which specializes to \eqref{planch} when $\alpha=\beta=1$.

Having no way of accessing the sign of the periods $\mathcal{P}_{\bT}^\chi(\phi_{\sigma}^\circ)$ as $\chi$ varies, we now apply the triangle inequality to the $\chi$-sum, to get
\begin{equation}\label{plan}
|\mathcal{P}_{\Delta_{\alpha, \beta}\bT}(\phi_\sigma^\circ)|\leq \sum_{\chi \in \widetilde{\rm Cl}^\vee_E}\big|\mathcal{P}_{\bT}^{\chi^{\beta'}}(\phi_{\sigma_1}^\circ)\mathcal{P}_{\bT}^{\chi^{\alpha'}}(\phi_{\sigma_2}^\circ)\big|. 
\end{equation}

\begin{remark}
The bound \eqref{plan} sacrifices all cancellation in the $\chi$-sum. Something similar was done in the pioneering work of Holowinsky \cite{Ho}, and later Lester and Radziwi\l\l\, \cite{LR}, with respect to unipotent periods, in which $\bT$ is replaced by the unipotent subgroup $\mathbf{N}$ of upper triangular matrices in $\PGL_2$. At first sight this looks like a hopeless gambit, in view of the loss of information incurred, but these breakthrough papers demonstrated that it is reasonable nonetheless to hope for some small savings.
\end{remark}

\subsection{Fractional moments of $L$-functions}\label{sec:fractional}
We now convert the right-hand side of \eqref{plan} to a fractional moment of $L$-functions using an explicit form of Waldspurger's theorem \cite{Wa}, which relates the twisted period \eqref{twisted} to the central Rankin--Selberg $L$-value $L(1/2, \pi_j \times \chi)$, where $\pi_j$ on ${\rm PGL}_2(\Bbb{A})$ is associated with $\sigma_j$ by the Jacquet--Langlands correspondence and $\chi$ is viewed as the automorphic induction to $\GL_2$ over $\Q$. This will lead us to the statement of our main analytic result Theorem \ref{thm3}, which will be seen to imply both Theorems \ref{thm1} and \ref{thm2}. 

The dual group $\widetilde{\rm Cl}_E^{\vee}$ consists of everywhere unramified Hecke characters, trivial at infinity if $E$ is imaginary and totally even at infinity if $E$ is real. Let $\psi\in \widetilde{\rm Cl}_E^\vee$ denote either of the characters $\chi^{\alpha'}$ or $\chi^{\beta'}$ appearing in \eqref{plan}. We may assume that ${\rm Hom}_{\bT_{\iota_j}(\A_\Q)}(\sigma_j,\psi)\neq 0$ since otherwise the twisted period $\mathcal{P}_{\bT_j}^\psi |_{\sigma_j}$ in \eqref{plan} vanishes. Moreover, since $\psi$ is unramified, and $\phi_{\sigma_j}^\circ$ is invariant under ${\rm g}_f^{-1}\bG_j(\widehat{\Z}){\rm g}_f$ in which $\bT$ is optimally embedded, the vector $\phi_{\sigma_j}^\circ$ is the global Gross--Prasad \cite{GP} vector. This observation allows us to use the explicit Waldspurger formula from \cite[Theorem 1.1]{FMP}, which states that
\begin{equation}\label{twisted-Walds}
|\mathcal{P}_{\bT_j}^\psi(\phi_{\sigma_j}^\circ)|^2= C_{\bG_j} C_{\rm Ram}(\pi_j, \psi)\frac{1}{\sqrt{D} }\frac{1}{L(1, \eta_E)^2}\frac{L(1/2, \pi_j \times \psi)}{L(1,\text{{\rm Ad }} \pi_j)}F(\pi_{j,\infty},\psi_\infty),
\end{equation}
for a constant $C_{\bG_j}>0$ depending only on $\bG_j$, a constant $C_{\rm Ram}(\pi,E)>0$ depending only on $E$ and on the local ramified local components of $\pi$, and a function $F$ of the spectral parameters of $\pi_\infty$ and $\psi_\infty$. 

In Appendix \ref{App-A} we prove bounds on $C_{\rm Ram}(\pi_j, \psi)$ and $F(\pi_{j,\infty},\psi_\infty)$. To state them here, let $N_{\pi_j}$ denote the (arithmetic) conductor of $\pi_j$ and write $\lambda_{\pi_j}$ for the spectral parameter $\lambda_{\sigma_j}$ from Section \ref{sec:Weyl}. If $\chi_\infty$ is the archimedean component of $\chi$, we let $\lambda_\chi\in\R$ be its frequency as a character on $\bT(\R)/K_\bT$. In particular, when $E$ is imaginary quadratic, we have $\chi_\infty=1$ and thus $\lambda_\chi=0$. If $E$ is real, then $\chi_{\infty}(x_1, x_2) = |x_1|^{i\lambda_{\chi}} |x_2|^{-i\lambda_{\chi}}$ with $\lambda_{\chi} \in \frac{\pi}{\log \epsilon}\Bbb{Z}$. We show in Appendix \ref{App-A}  that
\begin{equation}\label{boundF}
C_{\rm Ram}( \pi_j, \psi)\ll_\varepsilon N_{\pi_j}^{\varepsilon}\quad\textrm{and}\quad F(\pi_{j,\infty},\psi_\infty)\ll \exp \big(-  c_0|\lambda_{\psi}|/\lambda_{\pi_j}\big),
\end{equation}
for some absolute constant $c_0 > 0$.

We now take \textit{positive} square-roots in \eqref{twisted-Walds} and insert this into the right-hand side of \eqref{plan}, which we are to bound. Using \eqref{vol1}, as well as the bound $1/L(1,\eta_E)\ll\log\log D$, which is known to hold under GRH (see \eqref{titch2} below), we obtain
 \begin{align*}
\mathcal{P}_{\Delta\bT_{\alpha,\beta}}(\phi_\sigma^\circ)
\ll_\varepsilon \frac{N_
\pi^{\varepsilon} \log\log D}{{\rm vol}(\widetilde{\rm Cl}_E)} \sum_{\chi \in \widetilde{\rm Cl}^\vee_E}\exp\Big(- \frac{c_0|\lambda_{\chi}|}{\lambda_\pi}\Big)
\left(\frac{L(1/2, \pi_1 \times \chi^{\beta'})L(1/2, \pi_2 \times \chi^{\alpha'})}{L(1,\text{{\rm Ad }} \pi_1)L(1,\text{{\rm Ad }} \pi_2)}\right)^{1/2},
 \end{align*}
for some $c_0 > 0$ (which may depend on $\alpha, \beta$), where we have put $N_\pi=N_{\pi_1}N_{\pi_2}$, and $ \lambda_\pi = \lambda_{\pi_1} \lambda_{\pi_2}$. For notational simplicity we drop the prime in the above estimate and relabel $\alpha' \rightarrow \alpha$, $\beta' \rightarrow \beta$. 

Now that we have converted the original problem to one on $L$-functions, the key observation is that on average over $\chi$, the central values $L(1/2, \pi_1 \times \chi^\alpha)$ and $L(1/2, \pi_2 \times \chi^\beta)$ are a little less than 1, and that these small values occur \textit{independently of each other}. This will just suffice to obtain the equidistribution results in Theorems \ref{thm1} and \ref{thm2}. The following result makes this precise. 

\begin{theorem}\label{thm3} Let $E$ be a quadratic field extension of $\Q$, of discriminant $D$. Let $\pi_1, \pi_2$ be irreducible cuspidal automorphic representations on ${\rm PGL}_2(\A_\Q)$, of square-free level. Let $\lambda_{\pi_j}$ be the spectral parameter of $\pi_j$ and put $\lambda_\pi= \lambda_{\pi_1} \lambda_{\pi_2}$.  Let $c_0 > 0$. Assume the Generalized Riemann Hypothesis.
\begin{enumerate}
\item[(a)]\label{3a} If $\pi_1 \neq \pi_2$ then
\[
\frac{1}{{\rm vol}(\widetilde{\rm Cl}_E)} \sum_{\chi \in \widetilde{\rm Cl}^\vee_E} \exp\Big(-\frac{c_0|\lambda_{\chi}|}{ \lambda_\pi}\Big)\Big(\frac{L(1/2, \pi_1 \times \chi)L(1/2, \pi_2 \times \chi)}{L(1,\text{{\rm Ad }} \pi_1)L(1,\text{{\rm Ad }} \pi_2) }\Big)^{1/2} \ll_{\varepsilon, c_0} (\log D)^{-1/4 + \varepsilon}
\]
for any $\varepsilon > 0$. 

\item[(b)] Let $\alpha, \beta \in \Bbb{N}$ be two distinct positive integers. Assume that $\widetilde{\rm Cl}_E$ has no $2\beta$ torsion. Then
\[
\frac{1}{{\rm vol}(\widetilde{\rm Cl}_E)} \sum_{\chi \in \widetilde{\rm Cl}^\vee_E} \exp\Big(-\frac{c_0|\lambda_{\chi}|}{ \lambda_\pi}\Big)\Big(\frac{L(1/2, \pi_1 \times \chi^\alpha)L(1/2, \pi_2 \times \chi^\beta)}{L(1,\text{{\rm Ad }} \pi_1)L(1,\text{{\rm Ad }} \pi_2) }\Big)^{1/2} \ll_{\alpha,\beta,\varepsilon, c_0} (\log D)^{-1/4 + \varepsilon}
\]
for any $\varepsilon > 0$. 
\end{enumerate}
All implied constants depend polynomially on $\lambda_\pi$.
\end{theorem}
This theorem is the technical heart of the paper. The bound is best-possible, up to the power of $\varepsilon$, and in the following section we explain probabilistic model behind it. 

  As promised in the introduction, we give an immediate application to Fourier coefficients and Bessel periods of Yoshida lifts. Let $f, g$ be two distinct holomorphic cusp forms of weight $2$ and $2k-2$, respectively,  and level $N$ where $N$ is squarefree and has an odd number of prime factors. To $f$ and $g$ one can associate a non-zero holomorphic Siegel cusp form $F$ such that $L(s, F) = L(s, f)L(s, g)$ and for all $p \mid N$ the local representation $\pi_p$ associated with $F$ is of type VIb, cf.\ \cite{DPSS}. For a fundamental discriminant $D<0$ such that all $p \mid N$ remain inert in $\Bbb{Q}(\sqrt{D})$, let ${\rm Sym}_D^2(\Z)$ be the set of positive definite symmetric semi-integral matrices with integral diagonal in $M_2$ of determinant $D/4$. We may identify $\SL_2(\Z)$-equivalence classes of elements in ${\rm Sym}_D^2(\Z)$ with the class group ${\rm Cl}_E$ of $E = \Bbb{Q}(\sqrt{D})$. Let $a(F, S)$ denote the Fourier coefficient of $F$ at the matrix $S\in {\rm Sym}_D^2(\Z)$; as it depends only on the $\SL_2(\Z)$-equivalence of $S$ we may write $a(F,S_t)$ for $t\in {\rm Cl}_E$. For a fundamental discriminant $D < 0$ and a class group character $\chi\in {\rm Cl}_E^\vee$ we define the Bessel period
$$R(F, D, \chi) = \sum_{t \in \text{Cl}_E} a(F, S_t) \bar{\chi}(S_t).$$ From \cite[Proposition 3.14]{DPSS}, and Fourier inversion over ${\rm Cl}_E$, we obtain the following corollary of Theorem \ref{thm3}. 
\begin{cor}\label{cor4} Assume GRH. Let $F$ be a Yoshida lift and $D<0$   a fundamental discriminant satisfying the above assumptions. Then we have 
$$\max_{t\in {\rm Cl}_E}|a(F, S_t)| \leq \frac{1}{|{\rm Cl}_E|}\sum_{\chi \in {\rm Cl}_E^{\vee}} |R(F, D, \chi)|\ll_{\varepsilon, F} \frac{|D|^{(k-1)/2}}{(\log |D|)^{1/4 - \varepsilon}}$$
for any $\varepsilon > 0$. 
\end{cor}

On GRH, the trivial bound on the left-hand side and the sum in the middle is $|D|^{(k-1)/2 + \varepsilon}$. The  bound on the right-hand side is sharp, up to the value of $\varepsilon$, since the same is true of Theorem \ref{thm3}. One expects a bound of size $a(F, S_t)\ll_\varepsilon |D|^{k/2 - 3/4 + \varepsilon}$, based on both GRH and square-root cancellation in the $\chi$-sum in \eqref{planch}.

\section{Heuristics}\label{sec5}

In this section we give a heuristic argument for the proof of Theorem \ref{thm3}(a). For simplicity we assume that $E$ is imaginary quadratic (so that $\widetilde{\rm Cl}_E$ is just the ideal class group ${\rm Cl}_E$), and we also drop the adjoint $L$-values. We aim to explain what one should expect for the average
$$\frac{1}{|{\rm Cl}_E|} \sum_{\chi  \in {\rm Cl}_E^\vee}   \big(L(1/2, \pi_1 \times \chi)L(1/2, \pi_2 \times \chi)\big)^{1/2}.$$

For a function $f : {\rm Cl}_E^\vee \rightarrow \Bbb{C}$ let
\[
\Bbb{E}(f ) = \frac{1}{|{\rm Cl}_E|} \sum_{\chi \in {\rm Cl}_E^\vee} f(\chi).
\]
The basic idea is that on GRH we can express $\log L(1/2, \pi_j \times \chi)$ by a short sum over primes, namely
\begin{equation}\label{approx}
\log L(1/2, \pi_j \times \chi)  \approx  S_{\text{sp}}(\pi_j, \chi) + S_{\text{in}}(\pi_j, \chi) + O(1),
\end{equation}
where
\begin{displaymath}
\begin{split}
& S_{\text{sp}}(\pi_j, \chi) =  \sum_{\substack{(p) = \mathfrak{p} \bar{\mathfrak{p}}\\ p \leq x}} \frac{(\chi(\mathfrak{p}) + \bar{\chi}(\mathfrak{p}) )\lambda_j(p)}{p^{1/2}} + \frac{1}{2}\sum_{\substack{(p) = \mathfrak{p} \bar{\mathfrak{p}}\\ p^2 \leq x}} \frac{(\chi(\mathfrak{p})^2 + \bar{\chi}(\mathfrak{p})^2 )(\lambda_j(p^2)-1)}{p},\\
&   S_{\text{in}}(\pi_j, \chi)  =  \frac{1}{2}\sum_{\substack{\eta_E(p) = -1\\ p^2 \leq x}} \frac{2(\lambda_j(p^2)-1)}{p}.
\end{split}
\end{displaymath}
Here, $\lambda_j(n)$ denotes the Hecke eigenvalues of $\pi_j$ and we think of $x \approx \exp(\log{D}/\log\log {D})$. For the sake of exposition we ignore ramified primes. 

Asymptotically we have $\Bbb{E}( S_{\text{sp}}(\pi_j, \chi)) \sim 0$ by orthogonality of characters, and the variance of $S_{\text{sp}}(\pi_j, \chi)$ equals
\begin{align}\label{split-variance}
\Bbb{E} \sum_{\substack{\eta_E(p) = 1\\ p \leq x}} \frac{(\chi(\mathfrak{p}) + \bar{\chi}(\mathfrak{p}) )^2\lambda_j(p)^2}{p} + O(1) & = \Bbb{E}\sum_{ p \leq x} \frac{(\chi(\mathfrak{p})^2 + \bar{\chi}(\mathfrak{p})^2 + 2 )(\lambda_j(p^2)+1)\frac{1}{2}(1 + \eta_E(p))}{p} + O(1)\nonumber\\
&= \sum_{ p \leq x} \frac{ (\lambda_j(p^2)+1) (1 + \eta_E(p))}{p} + O(1).
 \end{align}
Recall that $\eta_E$ is the quadratic character associated to the extension $E/\Bbb{Q}$ by class field theory, and put $\theta_E = 1 \boxplus\eta_E$. Then \eqref{split-variance} is $\text{var}_{j, D}(x)+O(1)$, where 
$$\text{var}_{j, D}(x) = \log\log x + \log L(1, \eta_E) + \log(1, \text{Ad}\, \pi_j \times \theta_E). \quad $$
The inert sum $ S_{\text{in}}(\pi, \chi)$ is independent of $\chi$ and can be written as \[
  S_{\text{in}}(\pi_j, \chi)  = \sum_{ p^2 \leq x} \frac{(\lambda_j(p^2)-1)\frac{1}{2}(1 - \eta_E(p))}{p} = \mu_{j, D}(x) + O(1),
\]
where
\begin{equation}\label{mu-heur}
\mu_{j, D}(x) = \frac12\log L(1, \eta_E) +\frac12\log L(1, \text{Ad}\, \pi_j)-\frac12\log L(1, \text{Ad}\, \pi_j\times \eta_E) -\frac12 \log \log x.
\end{equation}
Assuming that the sums over $p$ behave like independent Gau{\ss}ian random variables, we conclude
\begin{displaymath}
\begin{split}
\Bbb{E}(L(1/2, \pi_j \times \chi))& \asymp \Bbb{E}\Big(\exp( S_{\text{sp}}(\pi_j, \chi))\Big) \Bbb{E}\Big( \exp ( S_{\text{in}}(\pi_j, \chi))\Big) \\
& \asymp  \exp\Big(\frac{1}{2}\text{var}_{j, D}(x)\Big) \exp\big( \mu_{j, D}(x) \big) = L(1, \eta_E)L(1, \text{Ad}\, \pi_j). 
 \end{split}
 \end{displaymath}
Assuming now that the two ``events'' $\log L(1/2, \pi_1 \times \chi)$ and $\log L(1/2, \pi_2\times \chi)$ are independent, we obtain by the same argument
\begin{displaymath}
\begin{split}
\Bbb{E}&\Big((L(1/2, \pi_1 \times \chi)L(1/2, \pi_2\times \chi))^{1/2}\Big)\\
& \asymp \Bbb{E}\Big(\exp(\textstyle\frac{1}{2} S_{\text{sp}}(\pi_1, \chi))\Big)\Bbb{E}\Big(\exp(\textstyle\frac{1}{2} S_{\text{sp}}(\pi_2, \chi))\Big) \Bbb{E}\Big(\exp(\textstyle\frac{1}{2} S_{\text{in}}(\pi_1, \chi))\Big)\Bbb{E}\Big(\exp(\textstyle\frac{1}{2} S_{\text{in}}(\pi_2, \chi))\Big) \\
& \asymp \exp\Big( \frac{\text{var}_{1, D}(x) + \text{var}_{2, D}(x)}{4}\Big) \exp\Big( \frac{\mu_{1, D}(x) + \mu_{2, D}(x)}{2}\Big) \asymp  \frac{(\mathcal{L}_1 \mathcal{L}_2)^{1/8}}{(\log x)^{1/4}},
 \end{split}
 \end{displaymath}
where
$$\mathcal{L}_j = \frac{L(1, \eta_E)^3 L(1, \text{Ad}\, \pi_j)^3}{L(1, \text{Ad}\, \pi_j \times \eta_E)}. $$ 
This probabilistic model suggests a final saving of $(\log x)^{-1/4} = (\log {D})^{1/4 + o(1)}$, in agreement with Theorem \ref{thm3}(a). 

\begin{remark}\label{rem4}
The above computations are sensitive to $\pi_j$ being cuspidal. In particular, the important term $-\frac{1}{2} \log\log x$ in \eqref{mu-heur} is a reflection of the fact that $\lambda_j(p^2)$ oscillates unlike, for instance, $\tau(p^2)$. 
\end{remark}

The probabilistic model just described shows what one can reasonably expect, and as such it is a useful tool in stress-testing a proof strategy. But of course these heuristics are far from a proof. In an important paper on moments of the Riemann zeta function \cite{So}, Soundararajan made the key observation that \eqref{approx} can in fact be made precise as an \textit{upper bound} for relatively small $x$. 
Then 
one proceeds by computing very high moments 
\[
\frac{1}{|{\rm Cl}_E|} \sum_{\chi  \in {\rm Cl}_E^\vee}   \big(\log L(1/2, \pi_1 \times \chi) + \log L(1/2, \pi_2 \times \chi)\big)^{k},
\]
for $k$ as large as about $\log\log {D}$. From the moments we get   sufficient information on the  distribution function of $\log L(1/2, \pi_1\times  \chi) + \log L(1/2, \pi_2 \times \chi)$ (which supports the Gau{\ss}ian heuristics), and hence also the desired $k=1/2$-moment. We will formalize this argument in Section \ref{sec9}.

\begin{remark}\label{rem:correction-factor}
We will see later that the probabilistic model of this section is not entirely correct: the events $\log L(1/2, \pi_1 \times \chi)$ and $\log L(1/2, \pi_2\times \chi)$ are \emph{not} independent, and their correlation is measured by the $L$-value $L(1, \pi_1 \times\pi_2 \times \theta_E)$ of degree 8. That this value is well-defined is a consequence of our assumption in Theorem \ref{thm1} that both $\pi_j$ are distinct (and at least one is cuspidal). The presence of this correlation $L$-value has no influence on the power of $\log x$. We see, however, that it is  important to analyze the $L$-values at $1$ very carefully, which is even on GRH a subtle matter. We refer the reader to the discussion in Sections \ref{sec42} and \ref{end}.
\end{remark}

\section{$L$-functions}\label{sec6}

 In this section, we summarize (and in some cases prove) the various analytic properties of $L$-functions which will be necessary in the proof of Theorem \ref{thm3}. Throughout, we consider two cuspidal automorphic representations $\pi_1, \pi_2$ for ${\rm PGL}_2$ of squarefree levels $N_1, N_2$ and analytic conductors  $Q_{\pi_1}\asymp N_1 \lambda_{\pi_2}^2$, $Q_{\pi_2}\asymp N_2 \lambda_{\pi_1}^2$, where the spectral parameters $\lambda_{\pi_i}$ are defined in \S\ref{sec:Weyl}. 

\subsection{Generalities}\label{sec41} We denote by $\lambda_j(n)$, $j=1, 2$, the Hecke eigenvalues of $\pi_j$ and record the Hecke relation
\begin{equation}\label{mult}
\lambda_j(p^2) = \lambda_j(p)^2 - \psi_j(p)
\end{equation}
for a prime $p$ where $\psi_j$ is the trivial character modulo $N_j$. This relation will be used repeatedly throughout our argument.

Recall the Arakelov class group $\widetilde{\rm Cl}_E$ discussed in Section \ref{sec:ACG}.  Let $\chi \in \widetilde{\rm Cl}_E^\vee$ be an everywhere unramified idele class character. By automorphic induction, we may view $\chi$, when convenient, as a theta series of weight 1 if $E$ is imaginary and of weight 0 if $E$ is real. We denote its Dirichlet coefficients by $a_{\chi}(n)$. Note that $\chi$ is in particular trivial on ideals $(n) \subseteq \mathcal{O}_E$ with $n \in \Bbb{N}$, i.e.\ ideals  induced from $\Bbb{Q}$. Since $N_1, N_2$ are squarefree,  the Dirichlet   series expansions are given by 
\begin{equation}\label{dir1}
\begin{split}
  & L(s, \pi_j \times \chi) = L(2s, \eta_E)\sum_n \frac{\lambda_{j}(n)a_{\chi}(n)}{n^s},\\
   & L(s, \text{Ad} \, \pi_j) = \zeta^{(N_j)}(2s)\sum_{n} \frac{\lambda_j(n^2)}{n^s}, \quad L(s, \text{Ad} \, \pi_j \times \eta_E) = \zeta^{(DN_j)}(2s)\sum_{n} \frac{\lambda_j(n^2)\eta_E(n)}{n^s}
   \end{split}
   \end{equation}
 (see \cite[Section 2]{WLi} and \cite[Section 2.3.3]{BFKMMS}),  and we have the Euler product
\begin{equation}\label{RS-euler}
L(s, \pi_j \times \chi) = \prod_{p} \prod_{i , k= 1, 2}  \Big( 1 - \frac{\alpha_j(p, i) \xi_{\chi}(p, k)  }{p^s}\Big)^{-1}
\end{equation}
where 
\begin{equation}\label{xi} 
\{\xi_{\chi}(p, 1), \xi_{\chi}(p, 2)\} = \begin{cases} \{\chi(\mathfrak{p}), \bar{\chi}(\mathfrak{p})\}, & (p) = \mathfrak{p} \bar{\mathfrak{p}}, \mathfrak{p} \not= \bar{\mathfrak{p}},\\\{\chi(\mathfrak{p}), 0\}, & (p) = \mathfrak{p}^2,\\ \{-1, 1\}, & \eta_E(p) = -1,\end{cases}
\end{equation}
 are the Satake parameters of $\chi$ and $\alpha_j(p, i)$ are the Satake parameters of $\pi_j$. If $(p) = \mathfrak{p}^2$, then $\chi(\mathfrak{p})^2 = 1$, and we conclude
 \begin{equation}\label{sat}
   \xi_{\chi}(p, 1)^2 + \xi_{\chi}(p, 2)^2 = \begin{cases} \xi_{\chi^2}(p, 1) + \xi_{\chi^2}(p, 2), & p \text{ split},\\
   1, & p \text{ ramified},\\
   2, & p \text{ inert}. \end{cases}
 \end{equation}

 Let $\theta_E = 1 \boxplus \eta_E$. We use the explicit computations in \cite{WLi} again to conclude
\begin{displaymath}
\begin{split}
L(s, \pi_1 \times \pi_2 \times \theta_E)& = L(s, \pi_1 \times \pi_2) L(s, \pi_1 \times \pi_2 \times \eta_E) \\
&= \zeta^{(N_1N_2)}(2s)\zeta^{(N_1N_2D)}(2s)\sum_{n} \frac{\lambda_1(n) \lambda_2(n)}{n^s} \sum_{n} \frac{\lambda_1(n) \lambda_2(n)\eta_E(n)}{n^s}\\
&\quad \prod_{\substack{p \mid (N_1, N_2)\\ p \nmid D}} \Big(1 - \frac{\lambda_1(p)\lambda_2(p)}{p^{s-1}}\Big)\Big(1 - \frac{\lambda_1(p)\lambda_2(p)\eta_E(p)}{p^{s-1}}\Big).
\end{split}
\end{displaymath}
Note that $\lambda_p := \lambda_1(p)\lambda_2(p) p \in \{\pm 1\}$ for $p \mid (N_1, N_2)$, so that 
\begin{align}\label{dir2}
L(s, \pi_1 \times \pi_2 \times \theta_E) = &\sum_{n} \frac{\lambda_1(n) \lambda_2(n)}{n^s} \sum_{n} \frac{\lambda_1(n) \lambda_2(n)\eta_E(n)}{n^s} \nonumber\\
&\zeta^{(N_1N_2)}(2s)\zeta^{(N_1N_2D)}(2s)\prod_{\substack{p \mid (N_1, N_2)\\ p \nmid D}} \Big(1 - \frac{\lambda_p(1 + \eta_E(p))}{p^{s}} + \frac{\eta_E(p)}{p^{2s}}\Big).
\end{align}

\subsection{$L$-functions at $s=1$ on GRH}\label{sec42}

We must be very careful with bounds for $L$-functions on the 1-line. On GRH \emph{and} the Ramanujan conjecture all of them are $(\log\log D)^{O(1)}$ from above and below. Without the Ramanujan conjecture, existing bounds are exponential in $\log D$ which is problematic. Luckily,   we do have good bounds in some situations and this suffices for our application. The following lemma is well-known (see e.g. \cite[Lemma 5.3]{LR} for a special case) and goes essentially back to Littlewood. For convenience we provide a complete proof.

\begin{lemma}\label{tit} Let $L(s, \pi)$ be a holomorphic $L$-function of fixed degree $d$  and analytic conductor $Q_{\pi}$ in the extended\footnote{i.e.\ without assuming the Ramanujan bounds} Selberg class, not necessarily primitive, with Dirichlet coefficients satisfying $\lambda_{\pi}(n)$.  Assume GRH for $L(s, \pi)$ and assume that the Satake parameters $\alpha_{\pi}(p, j)$, $1 \leq j \leq d$, satisfy $|\alpha_{\pi}(p, j)| \leq p^{1/2 - \rho}$ for some $\rho > 0$. Then
\begin{equation}\label{log-series}
  \sum_{p \leq x} \frac{\lambda_{\pi}(p)}{p} = \log L(1, \pi) + O(1), \quad x \geq (\log Q_{\pi})^{2+\varepsilon}.
\end{equation}
Moreover, for fixed $\alpha > 0$ we have 
\begin{equation}\label{titch1}
\begin{split}
&  L(1, \pi) \gg (\log\log Q_{\pi})^{-\alpha}, \quad \text{if } \lambda_{\pi}(p) \geq -\alpha \text{ for all } p,\\
 & L(1, \pi) \ll (\log\log Q_{\pi})^{\alpha}, \quad \text{if } \lambda_{\pi}(p) \leq \alpha \text{ for all } p.\\
\end{split}  
\end{equation}
\end{lemma}

\begin{proof}  Let $T > 2$ and $s = \sigma + it$. We start with Perron's formula
$$\frac{1}{2\pi i}\int_{1-iT}^{1 + i T} \log L(s+1, \pi) \frac{x^s}{s} ds= \sum_{p \leq x} \frac{\lambda_{\pi}(p)}{p}  + O\Big(1 + \frac{x \log x}{T}\Big). $$
By the Borel-Caratheodory inequality and the convexity bound we have
$$\log L(s+1, \pi) \ll (\sigma +1/2)^{-1} \log \big(Q_{\pi}(1 + |t|)\big)$$
for $\sigma > -1/2$. We shift the contour to $\Re s = -1/2 + \delta$ for some small $\delta > 0$. At $s = 0$ we collect a simple pole with residue $\log L(1, \pi)$. The horizontal contours contribute
$$\ll  \int_{-1/2 + \delta  }^{1  } |\log L(\sigma +1 + i T, \pi)| \frac{x^{\sigma}}{T}d\sigma \ll \frac{x \log (Q_{\pi}T)}{\delta T}.$$
The vertical contour contributes 
$$ \ll \int_{-T  }^T |\log L(1/2 + \delta + i t, \pi)| \frac{x^{1/2 + \delta}}{1 + |t|}dt \ll \frac{ \log (Q_{\pi}T) \log T}{\delta x^{1/2 + \delta}}.$$
With $T = x^2 \log Q_{\pi}$ and $\delta = 1/\log x$ we conclude 
\begin{displaymath}
\begin{split}
\log L(1, \pi) & =  \sum_{p \leq x} \frac{\lambda_{\pi}(p)}{p}  +O\Big(    1 + \frac{x\log x}{T} +   \frac{x\log (Q_{\pi}T)}{\delta T}+ \frac{  \log (Q_{\pi}T) \log T}{\delta x^{1/2 - \delta}}\Big)\\
& = \sum_{p \leq x} \frac{\lambda_{\pi}(p)}{p}  +O\Big(  1 + \frac{\log x( \log Q_{\pi} + \log x)(\log x + \log\log Q_{\pi})}{x^{1/2 }}\Big). 
\end{split}
\end{displaymath}
The error term is  $O(1)$ if $x \geq (\log Q_{\pi})^{2+\varepsilon}$. If $\lambda_{\pi}(p) \geq -\alpha$ we choose $x =( \log Q_{\pi})^3$ getting
$$ L(1, \pi) \gg \exp (- \alpha \log\log x  ) \gg(\log\log Q_{\pi})^{-\alpha}.$$
An analogous argument works for $\lambda_{\pi}(p) \leq \alpha$.  \end{proof}

 In particular, we have
\begin{equation}\label{titch2}
\begin{split}
 \frac{1}{\log\log {D}} \ll L(1, \eta_E) \ll \log\log {D}
\end{split}  
\end{equation}
(which was proved by Littlewood). Moreover, for $\pi_j$ as in the beginning of this section, a key observation is that the Hecke relations \eqref{mult} imply  that $\lambda_j(p^2) \geq -1$, and then obviously also $\lambda_j(p^2)(1+\eta_E(p))\geq -2$. Thus  
\begin{equation}\label{titch3}
\begin{split}
  & \frac{1}{\log\log  Q_{\pi_j}} \ll L(1, \text{Ad} \, \pi_j), \quad  \frac{1}{(\log\log  Q_{\pi_j}{D})^2} \ll L(1, \text{Ad} \, \pi_j \times \theta_E). \\
\end{split}  
\end{equation}
We complement this with the additional bounds (cf.\ \cite[Lemma 5.5]{LR})
\begin{equation}\label{add}
\begin{split} 
  & \sum_{p \leq x} \frac{\lambda_j(p)^2}{p}, \quad  \sum_{p \leq x} \frac{\lambda_j(p^2)}{p}  \ll \log\log x + (\log Q_{\pi_j})^{1/3}.    
\end{split}   
\end{equation}

\subsection{$L$-functions at $s=1/2$ on GRH}

We now prove a precise version (on GRH) of the heuristic formula \eqref{approx}. A crucial ingredient will be the following special case of \cite[Theorem 2.1]{Ch}.

\begin{lemma}\label{lem:Chandee}
Assume GRH for $L(1/2, \pi_j\times \chi)$. 
Then for any $x>1$ we have
\[
\log L(1/2, \pi_j\times \chi) \leq \sum_{p^n \leq x}\sum_{i=1}^4 \frac{ \alpha_{\pi_j \times \chi}(p, i)^n}{n p^{n(1/2 + 1/\log x)}} \frac{\log(x/p^n)}{\log x} + 10 \frac{\log Q_{\pi_j \times \chi}}{\log x},
\]
where $\{\alpha_{\pi_j \times \chi}(p, i)\}$ are the four Satake parameters of $\pi_j \times \chi$ which can be read off from \eqref{RS-euler}--\eqref{xi}. 
\end{lemma}


Upon choosing $x = \log Q_{\pi_j\times \chi}$ and using the prime number theorem, we conclude 
\begin{equation}\label{54}
 \log L(1/2, \pi_j\times \chi) \ll  \frac{\log Q_{\pi_j\times \chi}}{\log\log Q_{\pi_j\times \chi}}.
\end{equation}


\begin{cor}\label{3line}
Assume GRH for $L(1/2, \pi_j\times \chi)$. Let $\varepsilon>0$ and suppose that 
\begin{equation*}
(\log Q_{\pi_j \times \chi})^{4+\varepsilon}\leq x. 
\end{equation*}
Then 
\begin{equation}\label{3displ}
\begin{split}
\log L(1/2, \pi_j\times \chi) 
\leq & \sum_{  p \leq x}  \frac{a_{\chi}(p)\lambda_j(p)}{ p^{1/2 + 1/\log x}} \frac{\log(x/p)}{\log x}+\frac{1}{2}\sum_{\substack{\eta_E(p) = 1\\ p^2 \leq x}}  \frac{a_{\chi^2}(p)(\lambda_j(p^2)-\psi_j(p))}{ p^{1 + 2/\log x}} \frac{\log(x/p^2)}{\log x}\\ 
+&\mu_{j,D}(x)+10 \frac{\log Q_{\pi_j \times \chi}}{\log x}+O\big( \log\log\log Q_{\pi_j}\big),
\end{split}
\end{equation}
where 
\begin{equation}\label{defmuj}
\mu_{j, D}(x)=\frac12\log L(1, \eta_E) +\frac12\log L(1, \text{{\rm Ad}}\, \pi_j)-\frac12\log L(1, \text{{\rm Ad}}\,  \pi_j \times \eta_E) -\frac12 \log \log x
\end{equation}
  satisfies
\begin{equation}\label{boundmu}
\mu_{j, D}(x)  \ll \log\log x  + (\log Q_{\pi_j})^{1/3}. 
\end{equation}
We also have
\begin{equation}\label{also}
\sum_{\substack{\eta_E(p) = 1\\ p^2 \leq x}}  \frac{a_{\chi^2}(p)(\lambda_j(p^2)-\psi_j(p))}{ p^{1 + 2/\log x}} \frac{\log(x/p^2)}{\log x}    \ll \log\log x  + (\log Q_{\pi_j})^{1/3}.
\end{equation}
\end{cor}

\begin{proof}
We can spell out the main term in Lemma \ref{lem:Chandee} explicitly. Indeed, the contribution from $n=1$ is the first term on the right hand side of \eqref{3displ}. The terms corresponding to $n \geq 3$ contribute $O(1)$  in view of the bound $\alpha_j(p, i)\ll p^{\delta}$ with $\delta < 1/6$ \cite{KS}.  
We use the Hecke relations \eqref{mult} for the  terms corresponding to $n=2$. Using \eqref{sat}, the split primes contribute the second term on the right hand side of \eqref{3displ}, while the ramified and inert primes contribute
\begin{equation*}
\begin{split}
&\frac{1}{2}\sum_{\substack{p\mid D\\ p^2 \leq x}}  \frac{(\lambda_j(p^2)-\psi_j(p))}{ p^{1 + 2/\log x}} \frac{\log(x/p^2)}{\log x} + \frac12\sum_{\substack{\eta_E(p) = -1\\ p^2 \leq x}}  \frac{2( \lambda_j(p^2)-\psi_j(p))}{ p^{1 + 2/\log x}} \frac{\log(x/p^2)}{\log x}\\
& =  \sum_{  p^2 \leq x}  \frac{( \lambda_j(p^2)-\psi_j(p)) \frac{1}{2}(1 - \eta_E(p)) + 1 -1}{ p^{1 + 2/\log x}}\cdot  \frac{\log(x/p^2)}{\log x}.
 \end{split}
 \end{equation*}
 Since $( \lambda_j(p^2)-\psi_j(p)) \frac{1}{2}(1 - \eta_E(p)) + 1 \geq 0$, the previous display is at most
 \begin{equation}\label{log} 
\sum_{  p \leq \sqrt{x}}  \frac{( \lambda_j(p^2)-\psi_j(p)) \frac{1}{2}(1 - \eta_E(p))}{ p} +\sum_{  p \leq \sqrt{x}}  \frac{1}{ p} - \sum_{p \leq \sqrt{x}} \frac{ \log (x/p^2)}{p^{1 + 2/\log x} \log x}.
\end{equation}
From \eqref{log-series} and $\log\log\sqrt{x}=\log\log x +O(1)$, the first term in \eqref{log} is
\[
\frac{1}{2}\log \frac{L(1, \eta_E) L(1, \mathrm{Ad}\, \pi_j)}{L(1, \mathrm{Ad}\, \pi_j \times \eta_E) \log x}  +O\Big(1 + \sum_{p \mid N_j} \frac{1}{p}\Big),
\]
provided that $x \geq (\log Q_{\pi_j \times \chi})^{4+\varepsilon}$, while the second and third  terms  in \eqref{log} are
\[
 \log \log \sqrt{x} - \sum_{p \leq \sqrt{x}} \frac{1}{p} \Big( 1 - O\Big( \frac{\log p}{\log x}\Big)\Big) + O(1) = O(1).
\]
This establishes \eqref{3displ}, observing that $\sum_{p \mid N_j} 1/p \ll \log\log \log N_j \leq \log\log \log Q_{\pi_j} $. Reversing the analysis, we obtain from \eqref{log-series} and \eqref{add} that 
\begin{displaymath}
\begin{split}
|\mu_{j, D}(x)|  &= \bigg|\sum_{  p \leq \sqrt{x}}  \frac{( \lambda_j(p^2)- \psi_j(p)) \frac{1}{2}(1 - \eta_E(p))}{ p}  + O(1)\bigg|  \leq \sum_{  p \leq \sqrt{x}} \frac{\lambda_j(p)^2}{p} +O(1)\\
& \ll \log\log x + (\log Q_{\pi_j})^{1/3}
\end{split}
\end{displaymath}
which establishes \eqref{boundmu}, and by the same argument we conclude \eqref{also}. 
\end{proof}

\section{Orthogonality}\label{sec7}

As outlined in Section \ref{sec5}, our methods are ultimately based on computing high moments of $\log L(1/2, \pi_1 \times \chi) L(1/2, \pi_2 \times \chi)$, or more generally $\log L(1/2, \pi_1 \times \chi^{\alpha}) L(1/2, \pi_2 \times \chi^{\beta})$, and by the results of the previous section these values can be approximated by short sums over primes. Starting from basic orthogonality relations, the lemmas in this section   estimate averages of increasing complexity over the Arakelov class  group $\widetilde{\rm Cl}_E^\vee$ of these short sums over primes. The principal results here are Lemmas \ref{lemma43} and \ref{lemma12}, which serve for Theorems \ref{thm1} and \ref{thm2}, respectively. We emphasize that this section does not invoke GRH, nor in fact do the cusp forms $\pi_j$ play any role.


For the rest of this paper we fix an even, non-negative Schwartz class function $F$  on $\R$ whose Fourier transform has support in $[-\frac{1}{2\pi} , \frac{1}{2\pi}]$. This will be of use in smoothly truncating the non-compact part of $\widetilde{\rm Cl}_E^\vee$ for real quadratic fields $E$. We recall the notation $\lambda_{\chi}$ as defined in Section \ref{sec:fractional},  {which in the case of real quadratic fields satisfies $\chi_{\infty}(x_1, x_2) = |x_1|^{i\lambda_{\chi}} |x_2|^{-i\lambda_{\chi}}.$
 
\begin{lemma}\label{orth} There exists a constant $c > 1$ with the following property. If  $\{0\} \not= \mathfrak{a} \subseteq \mathcal{O}_E$ is an ideal with $N\mathfrak{a} < {D}/4$ and  ${D} > c$, then 
$$\sum_{\chi \in \widetilde{\rm Cl}_E^\vee}F(\lambda_{\chi}) \chi(\mathfrak{a}) = 0$$
unless $\mathfrak{a} = (a)$, with $a \in \Bbb{N}$, is an ideal induced from $\Bbb{Q}$. 
\end{lemma}

\begin{proof} We put  $\sigma  = \eta_E(-1)$, so that the discriminant of $E$ is $\sigma D$. Recall that $\mathcal{O}_K = \Bbb{Z} + \frac{1}{2}(\sqrt{\sigma D} + \kappa) \Bbb{Z}$ where $\kappa = 0$ if $D \equiv 0$ (mod 4) and $\kappa = 1$ if $\sigma D \equiv 1$ (mod 4). The lemma is  easy to see if $E$ is imaginary. The sum can only be non-zero if $\mathfrak{a} = (\alpha)$ is a principal ideal. Now   $$\frac{{D}}{4} > N\alpha = N\Big(a + b \frac{\sqrt{\sigma D} + \kappa}{2}\Big) = a^2 + \kappa ab +   b^2\frac{{D}+\kappa}{4} \geq b^2 \frac{D}{4}$$ implies $b = 0$.

Let us now assume that $E$ is real. Again the sum can only be non-zero if $\mathfrak{a} = (\alpha)$ is principal.  Thus  our sum becomes
\begin{equation}\label{sum1}
\sum_{n \in \Bbb{Z}} F\Big( \frac{\pi n}{\log \epsilon}\Big) \Big| \frac{\alpha}{\alpha'}\Big|^{i n\pi/\log \epsilon}.
\end{equation}
Changing $\alpha = a + b (\sqrt{D} + \kappa)/2$ by a sign and replacing $\alpha$ with $\alpha'$ if necessary (without loss of generality since $F$ is even), we may assume that $b \geq 0$ and $a + \kappa b/2 \geq 0$. Now changing $\alpha$ by a power of $\epsilon$, we can also assume 
 $\epsilon^{-1} \leq |\alpha/\alpha'| < \epsilon$. 
By Poisson summation the previous display equals
$$ \frac{\log \epsilon}{\pi} \sum_{m \in \Bbb{Z}} \widehat{F} \Big( \frac{m}{\pi}\log \epsilon  - \frac{1}{2\pi} \log\Big| \frac{\alpha}{\alpha'} \Big|\Big) .$$
Since $\log \epsilon \geq \log D + O(1)$ and $D > c$, the support of $\widehat{F}$ implies that the sum consists only of one term and $\epsilon^{-1} \leq |\alpha/\alpha'| < \epsilon$ implies that this  term corresponds to $m=0$. 
We have 
$$\frac{\alpha}{\alpha'} = \frac{a + b(\sqrt{D}+\kappa)/2}{a - b(\sqrt{D}-\kappa)/2} = 1 + \frac{b \sqrt{D} (a + b( \sqrt{D}+\kappa)/2)}{N\alpha} \geq 1 + \frac{b^2D}{N\alpha} \geq 1 + 4b^2,$$
so the term $m=0$ is outside the support of $\widehat{F}$ for $b \not= 0$ (since $\log 5 > 1$). Hence \eqref{sum1} vanishes, unless $b=0$, in which case $\mathfrak{a} = (\alpha)$ is induced from $\Bbb{Q}$. \end{proof}

\subsection{High moments of short Dirichlet polynomials}

Let $\tilde{a}_{\chi}$ be the completely multiplicative function with $\tilde{a}_{\chi}(p) = a_{\chi}(p)$ where as in the beginning of Section \ref{sec41} we write $a_{\chi}(n)$ for the Dirichlet coefficients of the theta function induced by $\chi$.  Let $R$ be the multiplicative function with
$$R(p^{\alpha}) = \begin{cases} \binom{\alpha}{\alpha/2}, &  \alpha \text{ even}, p\text{ split},\\
1, & \alpha \text{ even}, p \text{ ramified},\\
0, & \text{otherwise.}\end{cases}$$
In the following we assume that ${D}$ is sufficiently large in order to apply the previous lemma.

\begin{lemma}\label{lem:chi-sum}
Let $\nu \in \Bbb{N}$. Assume that $n<({D}/4)^{1/\nu}$ and if $\nu$ is even suppose that $n$ is only composed of split primes. Then we have
\[
 \sum_{\chi \in \widetilde{\rm Cl}_E^\vee} F(\lambda_{\chi}) \tilde{a}_{\chi^\nu}(n) = R(n)  \sum_{\chi \in \widetilde{\rm Cl}_E^\vee} F(\lambda_{\chi}).
\]
\end{lemma}

\begin{proof}
If $n$ is divisible by an inert prime, then $\tilde{a}_{\chi}(n) = 0$ for any $\chi\in\widetilde{\rm Cl}_E^\vee$ so the sum over $\chi$ clearly vanishes. We may therefore assume that $n$ is supported on non-inert primes and write $n = \prod p_j^{\alpha_j} \prod q_k^{\beta_k}$ with pairwise distinct split primes $p_j = \mathfrak{p}_j \bar{\mathfrak{p}}_j$ and pairwise distinct  ramified primes $q_k = \mathfrak{q}_k^2$. We obtain
\begin{displaymath}
\begin{split}
\sum_{\chi \in \widetilde{\rm Cl}_E^\vee} F(\lambda_{\chi}) \tilde{a}_{\chi^\nu}(n)  & =\sum_{\chi \in \widetilde{\rm Cl}_E^\vee} F(\lambda_{\chi}) \prod (\chi^\nu(\mathfrak{p}_j) + \bar{\chi}^\nu(\mathfrak{p}_j))^{\alpha_j}\prod \chi^\nu(\mathfrak{q}_k)^{\beta_k}  \\&=  \sum_{\gamma_1 \leq \alpha_1, \, \gamma_2 \leq \alpha_2, \ldots } \binom{\alpha_1}{\gamma_1} \binom{\alpha_2}{\gamma_2} \cdots \sum_{\chi \in \widetilde{\rm Cl}_E^\vee} F(\lambda_{\chi})\chi^\nu(\mathfrak{p}_1^{\gamma_1}\bar{\mathfrak{p}}_1^{\alpha_1 - \gamma_1} \cdots \mathfrak{q}_1^{\beta_1} \cdots).
\end{split}
\end{displaymath}
By  Lemma \ref{orth}, the inner sum vanishes unless   $\mathfrak{p}_1^{\gamma_1}\bar{\mathfrak{p}}_1^{\alpha_1 - \gamma_1} \cdots \mathfrak{q}_1^{\beta_1} \cdots$ is a rational integer.  This is the case precisely when $2\gamma_j =\alpha_j$ for all $j$, and the $\beta_k$ are even for all $k$ (here we use the assumption $\beta_k = 0$ if $\nu$ is even). These conditions give exactly the definition of $R(n)$.
\end{proof}

The following is inspired by  \cite[Lemma 4.3]{LR}.  It is a central ingredient in the proof of Theorem \ref{thm1}.

\begin{lemma}\label{lemma43} Let $\nu \in \Bbb{N}$. Let  $x \geq 2$, $k \in \Bbb{N}$ with $x^{2k} < ({D}/4)^{1/\nu}$. 
For any sequence of complex numbers $b(p)$ indexed by primes (split primes if $\nu$ is even) we have
$$\sum_{\chi \in \widetilde{\rm Cl}_E^\vee} F(\lambda_{\chi}) \Big( \sum_{p \leq x} \frac{a_{\chi^\nu}(p)b(p) }{\sqrt{p}} \Big)^{2k} \leq    \frac{(2k)!}{k!} \Big(\frac12 \sum_{p \leq x} \frac{(1 + \eta_E(p))b(p)^2 }{p} \Big)^{k} \sum_{\chi \in \widetilde{\rm Cl}_E^\vee} F(\lambda_{\chi}).$$
 \end{lemma}

\begin{proof} We extend $b(p)$ to all integers as a completely multiplicative function. Let $p_j(n)$ be the characteristic function on numbers with $j$ prime factors (counted with multiplicity) and $\nu$  the multiplicative function with $\nu(p^{\alpha}) = \alpha!$. With these notational conventions we have  
 \begin{displaymath}
\begin{split}
&\sum_{\chi \in \widetilde{\rm Cl}_E^\vee} F(\lambda_{\chi}) \Big( \sum_{p \leq x} \frac{a_{\chi^\nu}(p)b(p) }{\sqrt{p}} \Big)^{2k} = \sum_{\substack{n\geq 1\\ p \mid n \Rightarrow p \leq x}} \frac{ (2k)!}{\nu(n)}\frac{b(n)  p_{2k}(n)  }{\sqrt{n}}\sum_{\chi \in \widetilde{\rm Cl}_E^\vee} F(\lambda_{\chi}) \tilde{a}_{\chi^\nu}(n). 
  \end{split}
\end{displaymath}
 We conclude from Lemma \ref{lem:chi-sum} that the inner sum vanishes unless $n$ is a square, so that the previous display equals
 $$  (2k)!\sum_{\substack{n\geq 1\\ p \mid n \Rightarrow p \leq x}} \frac{b(n)^2   p_k(n)  }{ n} \frac{R(n^2)}{\nu(n^2)}\sum_{\chi \in \widetilde{\rm Cl}_E^\vee} F(\lambda_{\chi}).$$
It follows from the definitions that
$$\frac{R(p^{2\alpha})}{\nu(p^{2\alpha})} = \left\{\begin{array}{ll}(\alpha!)^{-2}\leq (\alpha!)^{-1}=  \nu(p^{\alpha})^{-1}, & p\text{ split};\\  (2\alpha)!^{-1}\leq (2^{\alpha} \alpha!)^{-1}=(2^{\alpha} \nu(p^\alpha))^{-1}, & p \text{ ramified}\end{array}\right\} \leq \frac{1 + \eta_E(p)}{2\nu(p^\alpha)}. $$
Denoting by $r(n)$ the completely multiplicative function extending $r(p) = 1 + \eta_E(p)$, we obtain
\[
\sum_{\chi \in \widetilde{\rm Cl}_E^\vee} F(\lambda_{\chi}) \Big( \sum_{p \leq x} \frac{a_{\chi^\nu}(p)b(p) }{\sqrt{p}} \Big)^{2k} \leq   \frac{(2k)!}{k!} \bigg(\frac{1}{2^k}\!\!\!\sum_{\substack{n\geq 1\\ p \mid n \Rightarrow p \leq x}}\!\! \frac{k!}{\nu(n)}\frac{r(n)b(n)^2 p_k(n) }{n}\bigg) \sum_{\chi \in \widetilde{\rm Cl}_E^\vee} F(\lambda_{\chi}).
\]
The claim follows.
\end{proof}

\subsection{High moments of short mixed Dirichlet polynomials} For the proof of Theorem \ref{thm2} we need slightly more advanced combinatorics. We fix two \emph{distinct} positive integers $\alpha, \beta \in \Bbb{N}$. 
For $n \in \Bbb{N}$, $0 \leq m \leq n$ we define
$$B_{\alpha, \beta}(n, m) = \underset{2\alpha r + (\beta - \alpha) m - \beta n + 2\beta s = 0}{\sum_{r=0}^m \sum_{s = 0}^{n-m}} \binom{n}{m}\binom{m}{r}\binom{n-m}{s}.$$
One checks directly that
\begin{equation}\label{50}
 B_{\alpha, \beta}(1, 0)= B_{\alpha, \beta}(1, 1) = 0.
\end{equation}
Let $v_2$ the denote the usual 2-adic valuation. Then
\begin{equation}\label{50a}
B_{\alpha, \beta}(n, m) = 0 \quad  \text{unless}  \quad \begin{cases} 2\mid n, & v_2(\alpha) = v_2(\beta), \\ 2\mid m, & v_2(\alpha) < v_2(\beta),\\
2 \mid n-m, & v_2(\alpha) > v_2(\beta). \end{cases}
\end{equation}
We record the trivial bound
\begin{equation}\label{triv}
B_{\alpha, \beta}(n, m) \leq \binom{n}{m}2^{n}.
\end{equation}


The following lemma should be compared to Lemma \ref{lem:chi-sum}.

\begin{lemma}\label{lem10}
Let $b(p)$, $c(p)$ be any sequences indexed by split primes.  Let $f_\chi$ be the completely multiplicative function whose values at primes are given by $f_{\chi}(p) = a_{\chi^{\alpha}}(p) b(p) + a_{\chi^{\beta}}(p)c(p)$. For $n < ({D}/4)^{1/\max(\alpha, \beta)}$ we have
 \begin{equation}\label{sum}
 \sum_{\chi \in \widetilde{\rm Cl}_E^\vee} F(\lambda_{\chi})  f_{\chi}(n) = H_{\alpha, \beta}(n)\sum_{\chi \in \widetilde{\rm Cl}_E^\vee} F(\lambda_{\chi}) 
 \end{equation}
where  $H_{\alpha, \beta}$ is multiplicative  and given by 
$$H_{\alpha, \beta}(p^{\nu}) = \sum_{m=0}^{\nu} B_{\alpha, \beta}(\nu, m) b(p)^{m} c(p)^{\nu-m}. $$ 
\end{lemma}

\begin{remark}\label{rem6} For $n$ consisting only of split primes, \eqref{50} ensures that \eqref{sum} is supported only on squarefull $n$, but this property fails if $n$ has ramified prime factors. This would make later estimates in Section \ref{77} more cumbersome. For simplicity we exclude ramified primes factors which is reflected in the assumption on trivial 2-torsion in Theorem \ref{thm2}.  
 \end{remark}

\begin{proof} For a split prime $p = \mathfrak{p}\bar{\mathfrak{p}}$ and $\nu \in \Bbb{N}$ we have
\begin{displaymath}
\begin{split}
f_{\chi}(p^{\nu})&= \sum_{m=0}^{\nu} \binom{\nu}{m}  \big(a_{\chi^{\alpha}}(p) b(p)\big)^m \big(a_{\chi^{\beta}}(p)c(p)\big)^{\nu-m}\\
&= \sum_{m=0}^{\nu} \binom{\nu}{m} b(p)^m c(p)^{\nu-m} \sum_{r=0}^m \binom{m}{r} 
 \sum_{s=0}^{\nu-m} \binom{\nu-m}{s} \chi\big( \mathfrak{p}^{\alpha r+ \beta s} \overline{\mathfrak{p}}_j^{\alpha(m-r) + \beta(\nu-m-s)}\big) . 
 \end{split}
\end{displaymath}
We write $n = \prod_{j}  p_j^{\nu_j} $ with pairwise distinct split primes $p_j = \mathfrak{p}_j\bar{\mathfrak{p}}_j$, so that $f_{\chi}(n)$ equals 
\begin{displaymath}
\begin{split}
  \prod_j\sum_{m_j=0}^{\nu_j} \sum_{r_j=0}^{m_j} \sum_{s_j=0}^{\nu_j-m_j} \binom{m_j}{r_j} \binom{\nu_j}{m_j} \binom{\nu_j-m_j}{s_j}  b(p_j)^{m_j} c(p_j)^{\nu_j-m_j} \chi\big(\mathfrak{p}_j^{\alpha r_j + \beta s_j} \overline{\mathfrak{p}}_j^{\alpha(m_j-r_j) + \beta(\nu_j-m_j-s_j)} \big).
   \end{split}
\end{displaymath}
Summing $ f_{\chi}(n) $ over $\chi$, we see by Lemma \ref{orth} that only those terms with 
$$2\alpha r_j + (\beta - \alpha) m_j - \beta n_j + 2\beta s_j = 0$$
survive, so that the left hand side of \eqref{sum} equals
\begin{displaymath}
\begin{split}
 \prod_j \sum_{m_j=0}^{\nu_j} B_{\alpha, \beta}(\nu_j, m_j) b(p_j)^{m_j} c(p_j)^{\nu_j-m_j} \sum_{\chi \in \widetilde{\rm Cl}_E^\vee} F(\lambda_{\chi}). 
\end{split}
\end{displaymath}
This features precisely the function $H_{\alpha, \beta}$ specified in the lemma. \end{proof}

The following result should of course be compared to Lemma \ref{lemma43}.
\begin{lemma}\label{lemma12}
Let $\alpha, \beta \in \Bbb{N}$ be two distinct positive integers.  Let $b(p)$, $c(p)$ be any sequences of real numbers indexed by split primes. Let $x \geq 2$, $k \in \Bbb{N}$  with $x^{2k} <({D}/4)^{1/\max(\alpha, \beta)}$. Then 
$$\sum_{\chi \in \widetilde{\rm Cl}_E^\vee} F(\lambda_{\chi}) \Big(\sum_{p \leq x} \frac{a_{\chi^{\alpha}}(p) b(p) + a_{\chi^{\beta}}(p)c(p)}{\sqrt{p}}\Big)^{2k}$$
is bounded by 
$$\sum_{\chi \in \widetilde{\rm Cl}_E^\vee}  F(\lambda_{\chi}) \begin{cases} 
\displaystyle\sum_{2\ell_1 + 3\ell_2 = 2k} \frac{(2k)!}{(\ell_1)! (\ell_2)!} \Big(  \sum_{ p \leq x} \frac{ b(p)^2 + c(p)^2}{p} \Big)^{\ell_1}\Big(\sum_{p \leq x} \frac{  b(p)^2|c(p)|}{p^{3/2}}\Big)^{\ell_2}, & \!\!\!  v_2(\alpha) \not= v_2(\beta),\\
\displaystyle \sum_{2\ell_1 + 4\ell_2 = 2k} \frac{(2k)!}{(\ell_1)! (\ell_2)!} \Big(  \sum_{ p \leq x} \frac{ b(p)^2 + c(p)^2}{p} \Big)^{\ell_1}\Big(\sum_{p \leq x} \frac{ (|b(p)| + |c(p)|)^4}{p^{2}}\Big)^{\ell_2}, & \!\!\! v_2(\alpha) = v_2(\beta). 
 \end{cases}$$
 \end{lemma}
 
\begin{proof} We use the notation from Lemmas \ref{lemma43} and \ref{lem10}. We have 
$$\sum_{\chi \in \widetilde{\rm Cl}_E^\vee} F(\lambda_{\chi})  \Big(\sum_{p \leq x} \frac{a_{\chi^{\alpha}}(p) b(p) + a_{\chi^{\beta}}(p)c(p)}{\sqrt{p}}\Big)^{2k} = (2k)!\sum_{p \mid n \Rightarrow p \leq x}  \frac{p_{2k}(n)}{\nu(n)\sqrt{n}} H_{\alpha, \beta}(n)\sum_{\chi \in \widetilde{\rm Cl}_E^\vee} F(\lambda_{\chi}) .$$
Let us now consider the right hand side of the claimed inequality according to  the two different cases. 

 \textit{Case 1.} Suppose that $v_2(\alpha) \not= v_2(\beta)$, without loss of generality $v_2(\alpha) < v_2(\beta)$. 
 Let $\Phi, \Psi$ be the completely multiplicative function extending $\Phi(p) = b(p)^2 + c(p)^2$ and $\Psi(p) = b(p)^2|c(p)|$. Note that $\Phi, \Psi$ are non-negative. Now the right hand side equals 
\begin{displaymath}
\begin{split}
  (2k)! \sum_{2\ell_1+3\ell_2 = 2k}   \sum_{ p \mid n_1 \Rightarrow p \leq x} \frac{p_{\ell_1}(n_1) \Phi(n_1)}{\nu(n_1)n_1}    \sum_{ p \mid n_2 \Rightarrow p \leq x} \frac{p_{\ell_2}(n_2)\Psi(n_2)}{\nu(n_2)n_2^{3/2}}.  \end{split}
\end{displaymath}
Note that $H_{\alpha, \beta}$ is supported only on squarefull numbers, cf.\ \eqref{50}.  Decompose uniquely $n= n_1^2 n_2^3$ with $\mu^2(n_2) = 1$. It then suffices to show
$$\frac{H_{\alpha, \beta}(n)}{\nu(n)}   \leq \frac{\Phi(n_1) \Psi(n_2)}{\nu(n_1)\nu(n_2)}$$
and again by multiplicativity  
$$\frac{H_{\alpha, \beta}(p^{2r})}{(2r)!} \leq \frac{(b(p)^2 + c(p)^2)^{r}}{r!}, \quad \frac{H_{\alpha, \beta}(p^{2r+3})}{(2r+3)!} \leq \frac{(b(p)^2 + c(p)^2)^{r}\cdot b^2(p)|c(p)|}{r!}. $$
Both inequalities of the previous display follow from \eqref{50a} and  \eqref{triv}:
$$\frac{r!}{(2r)!} H_{\alpha, \beta}(p^{2r}) \leq \frac{r!}{(2r)!}  \sum_{m=0}^{r} \binom{2r}{2m} 2^{2r} b(p)^{2m} c(p)^{2r - 2m} \leq \sum_{m=0}^{r} \binom{r}{m} b(p)^{2m} c(p)^{2r - 2m}  = (b(p)^2 + c(p)^2)^{r}$$
and
\begin{displaymath}
\begin{split}
\frac{r!}{(2r+3)!} H_{\alpha, \beta}(p^{2r+3}) &\leq \frac{r!}{(2r+3)!} \sum_{m=0}^{r} \binom{2r+3}{ 2m+2} 2^{2r + 3} b(p)^{2m+2} c(p)^{2r+1 - 2m} \\
&\leq \sum_{m=0}^{r} \binom{r}{m} b(p)^{2m} c(p)^{2r - 2m}  b^2(p)|c(p)|= (b(p)^2 + c(p)^2)^{r}b^2(p)|c(p)|.
\end{split}
\end{displaymath} 

 \textit{Case 2.} Next suppose that $v_2(\alpha) = v_2(\beta)$. We argue similarly. Let $\tilde{\Psi}$ be the completely multiplicative function extending $\tilde{\Psi}(p) = (|b(p)| + |c(p)|)^4$. The   right hand side equals 
\begin{displaymath}
\begin{split}
  (2k)! \sum_{2\ell_1+4\ell_2 = 2k}   \sum_{ p \mid n_1 \Rightarrow p \leq x} \frac{p_{\ell_1}(n_1) \Phi(n_1)}{\nu(n_1)n_1}    \sum_{ p \mid n_2 \Rightarrow p \leq x} \frac{p_{\ell_2}(n_2)\tilde{\Psi}(n_2)}{\nu(n_2)n_2^{2}} . \end{split}
\end{displaymath}
Note that $H_{\alpha, \beta}$ is supported only on squares, cf.\  \eqref{50a}.   It then suffices to show
$$\frac{H_{\alpha, \beta}(p^2)}{2!} \leq  b(p)^2 + c(p)^2 ,
\quad \frac{H_{\alpha, \beta}(p^{2r})}{(2r)!} \leq \frac{(b(p)^2 + c(p)^2)^{r-2}\cdot (|b(p)| + |c(p)|)^4}{(r-2)!} $$
for $r \geq 2$. We have $H_{\alpha, \beta}(p^2) = 2( b(p)^2 + c(p)^2)$ [this uses $\alpha \not= \beta$ and is actually the only point where this assumption is used] and 
$$H_{\alpha, \beta}(p^{2r}) \leq \sum_{m=0}^{2r} \binom{2r}{m} 2^{2r}|b(p)|^m |c(p)|^{2r-m} \leq 2^{2r} (|b(p)| + |c(p)|)^{2r}. $$
As $b(p)^2 + c(p)^2 \geq  \frac{1}{2}(|b(p)| + |c(p)|)^2$, the last inequality follows from
$$\frac{2^{2r}}{(2r)!} \leq \frac{1}{2^{r-2} (r-2)!}$$
for $r \geq 2$. This completes the proof. 
\end{proof}

\section{A bound on the second moment}\label{sec8}

In preparation for the proof of Theorem \ref{thm3}, we start with a bound for the second moment of $L(1/2,\pi_j\times\chi)$. This will only serve a technical purpose to exclude very large values of $L$-functions. We continue to denote by $F$ be an even non-negative Schwartz-class function whose Fourier transform has support in $[-\frac{1}{2\pi}, \frac{1}{2\pi}]$. Note that for $Q > 1$ the function $x \mapsto F(x/Q)$ is still an even non-negative Schwartz-class function whose Fourier transform has support in $[-\frac{1}{2\pi}, \frac{1}{2\pi}]$. 
For later purposes we record for $k \in \Bbb{N}$ the elementary estimate
\begin{equation}\label{fac}
\frac{(2k)!}{k!} \leq \sqrt{2}  \Big(\frac{4}{e}k\Big)^k.
\end{equation}
(The constant $\sqrt{2}$ play no role in the following.) 
We observe that,  when $E$ is imaginary, the conductor of $\pi \times \chi$ is constant within the family of  $\chi \in \widetilde{\rm Cl}_E^\vee$ and depends only on $D$. If $E$ is real, the conductor of  $\pi \times \chi$ does depends on $\chi$, via its archimedean component $\chi_{\infty}$.  In either case, it is a consequence of the class number formula, and the fact that $\lambda_{\chi}$ runs through a one-dimensional lattice of volume $\pi/\log \epsilon$  when $E$ is real, that
\begin{equation}\label{HD}
 \sum_{ \chi \in \widetilde{\rm Cl}_E^\vee}F(\lambda_{\chi}/Q)  \ll Q H_E, \quad H_E :=  L(1, \eta_E) \sqrt{{D}}  
  \end{equation}
for $Q \geq 1$. Note that  $H_E\asymp \text{vol}(\widetilde{\rm Cl}_E)$ by \eqref{vol1}.

Since we allow polynomial dependence on the representations $\pi_1, \pi_2$ in the bound of Theorem \ref{thm3}, we will assume throughout this section that
$$Q_{\pi}  \leq {D}^{1/10}, $$
so that   $$ \log Q_{\pi_j \times \chi} \ll  \log {D},$$
provided that  $\lambda_{\chi} \leq {D}$. Moreover, we assume for the rest of the paper that ${D}$ is sufficiently large.  We also fix $j \in \{1, 2\}$ and drop it from the notation. 
 With this in mind we define
\begin{equation}\label{curlyC}
\mathscr{C} = \log\log {D} + (\log Q_{\pi})^{1/3}.
\end{equation}
We keep the general assumptions and notations from Section \ref{sec41}. 

\begin{lemma}\label{second} Assume that
$Q_\pi \leq {D}^{1/10}$, and let $1 \leq Q \leq {D}^{1/5}$. For $\varepsilon > 0$ we have 
$$\sum_{\chi \in \widetilde{\rm Cl}_E^\vee}F\Big(\frac{\lambda_{\chi}}{Q}\Big) L(1/2, \pi \times \chi)^2 \ll \exp(O_{\varepsilon}(\mathscr{C}^{1+\varepsilon}))QH_E.$$
\end{lemma}

\begin{proof}
For $\mathcal{V}\in \Bbb{R}$ we define
\[
\mathcal{S}(\mathcal{V}):=   \sum_{  \log L(1/2, \pi \times \chi) > \mathcal{V} }F^{\ast}\Big(\frac{\lambda_{\chi}}{Q}\Big) 
\]
(with the convention $\log 0 = -\infty$), where $F^{\ast}(x) = F(x) \delta_{|x| \leq {D}/Q}$. 
Using the convexity bound for $L(1/2, \pi \times \chi)$ and the rapid decay of $F$ and then  partial summation, we have
\begin{displaymath}
\begin{split}
\sum_{\chi \in \widetilde{\rm Cl}_E^\vee}F(\lambda_{\chi}) L(1/2, \pi \times \chi)^2 &= \sum_{\chi \in \widetilde{\rm Cl}_E^\vee}F^{*}(\lambda_{\chi}) L(1/2, \pi \times \chi)^2 + O({D}^{-10}) \\
&=  \int_{\Bbb{R}} e^\mathcal{V}  \mathcal{S}(\mathcal{V}/2) d\mathcal{V} + O({D}^{-10}) .
\end{split}
\end{displaymath}
 We may truncate the integral at $\mathcal{V}\ll\log {D}/\log\log {D}$, in view of  \eqref{54}, since otherwise $\mathcal{S}(\mathcal{V}/2) = 0$.  Moreover, using the trivial bound $ \mathcal{S}(\mathcal{V}/2)\leq \sum F(\lambda_{\chi}/Q)$  and \eqref{HD}, we have 
 \[
\int_{-\infty}^{A\mathscr{C}^{1+\varepsilon}}e^\mathcal{V} \mathcal{S}(\mathcal{V}/2) d\mathcal{V}\leq \sum_{\chi \in \widetilde{\rm Cl}_E^\vee}F\Big(\frac{\lambda_{\chi}}{Q}\Big)\int_{-\infty}^{A\mathscr{C}^{1+\varepsilon}}e^\mathcal{V}d\mathcal{V}\ll QH_E e^{A\mathscr{C}^{1+\varepsilon}}
\]
for $A > 0$. Put $V=\mathcal{V}/2$. 
  From the above considerations, we may now assume that $V$ verifies 
\begin{equation}\label{AA}
A\mathscr{C}^{1+\varepsilon} \leq V \leq B \frac{\log {D}}{\log\log {D}}
\end{equation}
 for some sufficiently large constants $A, B$. We shall show that
\begin{equation}\label{2nd-moment-bd}
  \mathcal{S}(V) \leq   \exp(-c(\varepsilon) V\log V)QH_E
\end{equation}
for $V$ satisfying \eqref{AA} and some $c(\varepsilon) > 0$. In this way,
\[
\int_{2A\mathscr{C}^{1+\varepsilon}}^{\infty} e^\mathcal{V}  \mathcal{S}(\mathcal{V}/2) d\mathcal{V} \ll QH_E,
\]
 which suffices for the proof of the lemma. The rest of the proof is devoted to \eqref{2nd-moment-bd}. Choose 
 \begin{equation}\label{x}
   x = {D}^{5B/  V} = \exp\Big( \frac{5B}{V} \log {D}\Big),
   \end{equation}
which by \eqref{AA} implies $x \geq  \exp(5 \log\log {D}) = (\log {D})^{5} \geq (\log Q_{\pi \times 
    \chi})^{4+\epsilon}$. For $\chi$ counted by $\mathcal{S}(V)$  we apply Corollary \ref{3line}, and  conclude from \eqref{3displ}, \eqref{boundmu} and \eqref{also} that 
   \begin{equation}\label{cor7}
   V \leq  \log L(1/2, \pi \times \chi) \leq \sum_{ p \leq x}  \frac{a_{\chi}(p)\lambda_{\pi}(p)}{ p^{1/2 + 1/\log x}} \frac{\log(x/p)}{\log x} + O\Big(\log\log x + (\log Q_{\pi})^{1/3} +  \frac{\log {D}}{\log x}\Big)
   \end{equation}
 Recalling \eqref{curlyC} and taking $B$ sufficiently large in \eqref{x} we find 
   \[
   V\leq \sum_{ p \leq x}  \frac{a_{\chi}(p)\lambda_{\pi}(p)}{ p^{1/2 + 1/\log x}} \frac{\log(x/p)}{\log x}  +  \frac{1}{4} V + O(\mathscr{C}).
   \]
Hence  if $A$  in \eqref{AA} is sufficiently large, we have
$$\frac{1}{2}V < \sum_{ p \leq x}  \frac{a_{\chi}(p)\lambda_{\pi}(p)}{ p^{1/2 + 1/\log x}} \frac{\log(x/p)}{\log x}.$$
For any $k \geq 0$, this implies (by positivity)
\[
  \mathcal{S}(V) \leq  \sum_{\chi \in \widetilde{\rm Cl}_E^\vee} F\Big(\frac{\lambda_{\chi}}{Q}\Big)\frac{2^{2k}}{V^{2k}} \Big(\sum_{  p \leq x}  \frac{a_{\chi}(p)\lambda_{\pi}(p)}{ p^{1/2 + 1/\log x}} \frac{\log(x/p)}{\log x}\Big)^{2k}.
\]
Lemma \ref{lemma43} (with $\nu=1$) then shows that, as long as $k$ verifies $x^{2k} < {D}/4$,  we have
 \[
     \mathcal{S}(V)\leq \sum_{\chi \in \widetilde{\rm Cl}_E^\vee} F\Big(\frac{\lambda_{\chi}}{Q}\Big) \frac{(2k)!}{k!}  \frac{2^{2k}}{V^{2k}}  \Big( \sum_{ p \leq x } \frac{\lambda_{\pi}(p)^2 }{p} \Big)^{k},
\]
where we have used the simple inequality $|1 + \eta_E(p)|\leq 2$. Now by   \eqref{add} and \eqref{x} we have
\[
\sum_{ p \leq x } \frac{\lambda_{\pi}(p)^2 }{p}\ll \mathscr{C}.
\]
Recalling \eqref{fac} and \eqref{HD} we deduce
   \begin{displaymath}
 \begin{split}
    \mathcal{S}(V) &  \ll QH_E  \frac{(2k)!}{k!}  \frac{2^{2k}}{V^{2k}}  \big(O(\mathscr{C})\big)^{k} \leq QH_E \Big(\frac{16k}{eV^2} O(\mathscr{C})\Big)^k. 
 \end{split}
 \end{displaymath}
 The condition $x^{2k} < {D}/4$ allows us to take $$k <  \frac{V}{11B} \asymp V.$$
For such $k$ we have $k \mathscr{C}/V^2  \ll V^{-\varepsilon/(1+\varepsilon)}$ by \eqref{AA}. Inserting this proves \eqref{2nd-moment-bd}.
\end{proof}

\section{Proof of Theorem \ref{thm3}}\label{sec9}

 In this section we finally prove Theorem \ref{thm3}. We first prove part (a)  of Theorem \ref{thm3}, and then in Section \ref{77} pass to the proof of part (b). In both cases, the proof scheme is similar to the proof of Lemma \ref{second}, but the  estimates are more delicate.

\subsection{Setting up the proof of Theorem \ref{thm3}(a)}\label{sec:setting-up}

For $\chi \in \widetilde{\rm Cl}_E^\vee$ let
\begin{equation}\label{defL}
\mathcal{L}(\chi)=L(1/2,  \pi_1\times \chi) L(1/2,  \pi_2\times \chi).
\end{equation}
We continue to assume that $D$ is sufficiently large and $Q_{\pi_1}, Q_{\pi_2} \leq {D}^{1/10}$.   
For $x > 0$ we define
\begin{equation}\label{defmu}
\begin{split}
\mu_{D}(x)  = &\frac{1}{2} \log L(1, \eta_E) + \frac{1}{4} \log L(1, \text{Ad}\, \pi_1)  + \frac{1}{4} \log L(1, \text{Ad}\, \pi_2)\\
& -  \frac{1}{4} \log L(1, \text{Ad}\, \pi_1 \times \eta_E)  - \frac{1}{4} \log L(1, \text{Ad}\, \pi_2 \times \eta_E) - \frac{1}{2}\log\log x
\end{split}
\end{equation}
and note that by \eqref{defmuj} we have $\mu_D(x) = \frac{1}{2}(\mu_{1, D}(x) + \mu_{2, D}(x))$. We also define
\begin{equation}\label{defsigma}
\begin{split}
  {\rm var}_{ D}(x) =&  \frac{1}{2} \log\log x + \frac{1}{2}\log L(1, \eta_E) + \frac{1}{4} \log L(1, \text{Ad} \pi_1 \times \theta_E) \\
  &+  \frac{1}{4} \log L(1, \text{Ad} \pi_2 \times \theta_E) + \frac{1}{2} \log L(1, \pi_1 \times \pi_2 \times \theta_E).
   \end{split}
\end{equation}
We write $\mu_D = \mu_D({D})$ and $\text{var}_D = \text{var}_D({D})$. 

Recalling the notation $H_E$ from \eqref{HD}, our primary goal is to prove the bound
\begin{equation}\label{primary}
 \sum_{\chi \in \widetilde{\rm Cl}_E^\vee}F\Big(\frac{\lambda_{\chi}}{Q_{\pi_1}Q_{\pi_2}}\Big) \mathcal{L}(\chi)^{1/2} \ll  H_E\Big(\exp\Big(\mu_{D} + \Big(\frac{1}{2} +  \varepsilon\Big) {\rm var}_{ D}\Big) + \exp(- \textstyle\frac{1}{3}\sqrt{\log {D}})\Big),
 \end{equation}
 provided that $Q_{\pi_1}, Q_{\pi_2} \leq {D}^{1/10}$. Here and henceforth all implied constants are allowed to depend polynomially on $Q_{\pi_1}$ and $Q_{\pi_2}$. 
 The proof of \eqref{primary} extends over  the next few subsections and will be completed in Section \ref{compl}. We then show in Section \ref{end} how to deduce Theorem \ref{thm3}(a) from \eqref{primary}.

\subsection{Some useful bounds}
We now relate $\mu_D$ and ${\rm var}_D$ to short Dirichlet polynomials. Let
$$\mathscr{L}(s) := \frac{1}{2}\log L(s, \eta_E) + \frac{1}{4} \log L(s, \text{Ad}\, \pi_1 \times \theta_E) +  \frac{1}{4} \log L(s, \text{Ad}\, \pi_2 \times \theta_E) + \frac{1}{2} \log L(s, \pi_1 \times \pi_2 \times \theta_E).$$
 By \eqref{dir1}, \eqref{dir2} and \eqref{mult} the Dirichlet series coefficients of $\mathscr{L}(s)$ at primes $p \nmid N_1N_2$ are
\begin{equation*}
\begin{split}
& \frac{2\eta_E(p) + (1 + \eta_E(p))(\lambda_1(p^2) + 2 \lambda_1(p) \lambda_2(p) + \lambda_2(p^2))}{4}=  \frac{  (1 + \eta_E(p))(\lambda_1(p) +   \lambda_2(p))^2 - 2}{4}. 
\end{split}
\end{equation*}
Similar formulae hold if $p$ divides exactly one of $N_1, N_2$ and if it divides both $N_1$ and $N_2$. We write this collectively as
\begin{equation*}
\frac{  (1 + \eta_E(p))(\lambda_1(p) +   \lambda_2(p))^2 - \kappa_1(p)}{4}, \quad \kappa_1(p) \in \{2, 1 - \eta_E(p), 4\eta_E - 2\}.
\end{equation*}
 We deduce from this that
\begin{equation}\label{boundsigma}
\frac12 \log\log {D} + O(\log\log\log {D} )  \leq {\rm var}_D  \ll (\log {D})^{1/3},
\end{equation}
where we use  \eqref{titch1} for the lower bound and \eqref{log-series}, \eqref{titch2} and \eqref{add} for the upper bound. Moreover, using \eqref{log-series}, we deduce that
\begin{equation}\label{dirich}
\begin{split}
\sum_{p \leq x} \frac{  (1 + \eta_E(p))(\lambda_1(p) +   \lambda_2(p))^2 }{2p} &= 2 \log \mathscr{L}(1)  + O(1) + \log\log x + O\Big( \sum_{p \mid N_1N_2} \frac{1}{p}\Big)\\
& = 2 \text{var}_D(x) + O(\log\log\log N_1N_2),
\end{split}
\end{equation}
provided that $x  > (\log Q_{\pi_1} Q_{\pi_2} {D})^{2+\varepsilon}$. 



By a similar computation, the Dirichlet coefficients of 
\begin{equation*}
  - \log L(1, \text{Ad} \,\pi_1 \times \theta_E) - \log L(1, \text{Ad}\, \pi_2 \times \theta_E) + 2  \log L(1, \pi_1 \times \pi_2 \times \theta_E)
  \end{equation*}
are 
$$\kappa_2(p) -(1 + \eta_E)(\lambda_1(p) + \lambda_2(p))^2, \quad \kappa_2(p) \in \{2, 1 - \eta_E(p), 4\eta_E(p) - 2\},  $$
in particular, since $\kappa_2(p) \leq 2$, we obtain by \eqref{titch1} that 
\begin{equation}\label{dirich2}
 \frac{ L(1, \pi_1 \times \pi_2 \times \theta_E)}{L(1, \text{Ad} \,\pi_1 \times \theta_E)L(1, \text{Ad} \,\pi_2 \times \theta_E)} \ll (\log\log Q_{\pi_1} Q_{\pi_2}{D})^2.
\end{equation}
Analogously the Dirichlet coefficients of 
\begin{equation*}
  - \log L(1, \text{Ad} \,\pi_1) - \log L(1, \text{Ad}\, \pi_2 ) +   \log L(1, \pi_1 \times \pi_2 \times \theta_E)
  \end{equation*}
are
$$\kappa_3(p) + (1+\eta_E(p)) \lambda_1(p) \lambda_2(p) - \lambda_1(p)^2   - \lambda_2(p)^2$$
with $\kappa_3(p) \leq 2$, so that 
\begin{equation}\label{dirich3}
 \frac{ L(1, \pi_1 \times \pi_2 \times \theta_E)}{L(1, \text{Ad} \,\pi_1  )L(1, \text{Ad} \,\pi_2  )} \ll (\log\log Q_{\pi_1} Q_{\pi_2}{D})^2.
\end{equation}

\subsection{Preliminary reductions}\label{prelim} As in the proof of Lemma \ref{second} we can restrict to characters with $|\lambda_{\chi}| \leq {D}$ by the rapid decay of $F$. Again we write $F^{\ast}(x) = F(x) \delta_{|x| \leq {D}/(Q_{\pi_1}Q_{\pi_2})}$.  
Recalling the notation \eqref{defL}, we define, similarly to the previous section,
$$\mathcal{T}(V)  := \sum_{\log \mathcal{L}(\chi) > V} F^{\ast}\Big(\frac{\lambda_{\chi}}{Q_{\pi_1}Q_{\pi_2}}\Big).$$
 We first deal with large values of $V$. For $U \geq 1$ we have 
\begin{displaymath}
\begin{split}
 \sum_{\mathcal{L}(\chi) \geq U} F\Big(\frac{\lambda_{\chi}}{Q_{\pi_1}Q_{\pi_2}}\Big)\mathcal{L}(\chi)^{1/2} &\leq \frac{1}{U^{1/2}} \sum_{\mathcal{L}(\chi) \geq U}F\Big(\frac{\lambda_{\chi}}{Q_{\pi_1}Q_{\pi_2}}\Big) \mathcal{L}(\chi)\\
 &\leq \frac{1}{U^{1/2}} \sum_{ \chi \in \widetilde{\rm Cl}_E^\vee}F\Big(\frac{\lambda_{\chi}}{Q_{\pi_1}Q_{\pi_2}}\Big) \big(L(1/2, \pi_1 \times \chi)^2 +L(1/2, \pi_2 \times \chi)^2\big).
\end{split}
\end{displaymath}
We now invoke Lemma  \ref{second}, and recall the definition of $\mathscr{C}$ from \eqref{curlyC}, to obtain
\begin{displaymath}
\begin{split}
 \sum_{\mathcal{L}(\chi) \geq U}F\Big(\frac{\lambda_{\chi}}{Q_{\pi_1}Q_{\pi_2}}\Big)  \mathcal{L}(\chi)^{1/2} &  \ll U^{-1/2} e^{(\log {D})^{2/5}}   H_E .
\end{split}
\end{displaymath}
We may therefore treat all $\chi $ with $\log \mathcal{L}(\chi) > (\log {D})^{1/2}$ trivially and estimate their contribution by
\begin{equation*}
\sum_{\log \mathcal{L}(\chi) > (\log {D})^{1/2}}F\Big(\frac{\lambda_{\chi}}{Q_{\pi_1}Q_{\pi_2}}\Big) \ll  H_E \exp(- \textstyle\frac{1}{3}\sqrt{\log {D}}).
\end{equation*}
This is admissible for \eqref{primary}. 

By partial summation it now remains to estimate
\begin{displaymath}
\begin{split}
  \int_{-\infty}^{(\log {D})^{1/2}}e^{V/2} \mathcal{T}(V) dV =  \exp(\mu_{D})  \int_{-\infty}^{(\log {D})^{1/2}- \mu_{ D}}e^{V/2} \mathcal{T}(V + \mu_{D}) dV.
\end{split}
\end{displaymath}
 
The contribution of $V\leq \varepsilon \log\log {D}$ can be estimated trivially by 
\begin{equation*}
\begin{split}
 \exp(\mu_{D}) &\int_{-\infty}^{\varepsilon \log\log {D}}e^{V/2} \mathcal{T}(V + \mu_{D}) dV \leq  \sum_{ \chi \in \widetilde{\rm Cl}_E^\vee}F\Big(\frac{\lambda_{\chi}}{Q_{\pi_1} Q_{\pi_2}}\Big)\exp(\mu_{D})\int_{-\infty}^{\varepsilon \log\log {D}}e^{V/2}  dV \\
&\ll H_E\exp\Big(\mu_{D} + \frac{\varepsilon}{2} \log\log {D}\Big) \ll   H_E\exp\big(\mu_{D} + O(\varepsilon){\rm var}_{D}\big),
\end{split}
\end{equation*}
where we used   \eqref{boundsigma}  in the last step. This is again admissible for \eqref{primary}, perhaps after redefining $\varepsilon$. 
   
The hardest part is to  estimate
\begin{equation}\label{int}
 \exp(\mu_{D})  \int_{\varepsilon \log\log V}^{(\log {D})^{1/2}- \mu_{ D}}e^{V/2}\mathcal{T}(V + \mu_{ D}) dV. 
\end{equation}
Recalling the definition of $\mu_D$ in \S\ref{sec:setting-up}, and applying \eqref{boundmu},  we have 
\begin{equation}\label{3star}
(\log {D})^{1/2}- \mu_{D} \asymp (\log {D})^{1/2}.
\end{equation}
Henceforth we restrict   $V$ to the interval 
\begin{equation}\label{V} 
   \varepsilon \log\log {D} \leq V \leq (\log {D})^{1/2}- \mu_{D}\ll (\log {D})^{1/2}.
   \end{equation}
   
   \subsection{Application of Corollary \ref{3line}} 
For   $V$ as in \eqref{V} we choose 
\begin{equation}\label{new-x}
   x = {D}^{A/\varepsilon V} = \exp\Big( \frac{A}{\varepsilon V} \log{D}\Big)
   \end{equation}
for a sufficiently large constant $A$, so that in particular $\log x \gg (\log D)^{1/2}$ and so $x \gg  (\log {D})^{5} \gg (\log Q_{\pi_j \times \chi})^{4+\varepsilon}$ for all $\chi$ in the support of $F^{\ast}(\lambda_{\chi}/Q_{\pi_1}Q_{\pi_2})$. We may now apply Corollary \ref{3line} for $j=1$ and $j=2$ in a similar way as in \eqref{cor7} to conclude that
\begin{equation}\label{sumprimes}
\begin{split}
\mathcal{L}(\chi) - \mu_{D} & \leq    \sum_{  p \leq x}  \frac{a_{\chi}(p)(\lambda_1(p) + \lambda_2(p))}{ p^{1/2 + 1/\log x}} \frac{\log(x/p)}{\log x} \\
&+\frac{1}{2}\sum_{\substack{\eta_E(p) = 1\\ p^2 \leq x}}  \frac{a_{\chi^2}(p)(\lambda_1(p^2) + \lambda_2(p^2) -\psi_1(p)- \psi_2(p))}{ p^{1 + 2/\log x}} \frac{\log(x/p^2)}{\log x} + O(\varepsilon V). 
\end{split}
\end{equation}
Here we used that $\mu_D(x) - \mu_D = - \frac{1}{2}(\log\log x + \log\log {D}) = \log V + O(1) \leq \frac{1}{10}\varepsilon V$, and also \eqref{V} and \eqref{new-x} to bound the remaining terms in Corollary \ref{3line} by $O(\varepsilon V)$. Hence  if $\chi$ is counted by  $\mathcal{T}(V+ \mu_{D})$ we have 
\begin{displaymath}
\begin{split}
(1-\varepsilon)V  &\leq    \sum_{  p \leq x}  \frac{a_{\chi}(p)(\lambda_1(p) + \lambda_2(p))}{ p^{1/2 + 1/\log x}} \frac{\log(x/p)}{\log x} \\
&+\frac{1}{2}\sum_{\substack{\eta_E(p) = 1\\ p^2 \leq x}}  \frac{a_{\chi^2}(p)(\lambda_1(p^2) + \lambda_2(p^2) -\psi_1(p)- \psi_2(p))}{ p^{1 + 2/\log x}} \frac{\log(x/p^2)}{\log x}. 
\end{split}
\end{displaymath}
We write the right hand side as
$$\sum_{p \leq z} + \sum_{z < p \leq x} + \sum_{p^2 \leq x} = S_1(\chi) + S_2(\chi) + S_3(\chi),$$
say, for some $z \geq (\log {D})^{2+\varepsilon}$. We choose 
\begin{equation}\label{z}
   \Delta = (\log\log {D})^{1/3}, \quad z = x^{1/\Delta} = \exp\Big( \frac{A}{\varepsilon V} \frac{\log {D}}{\Delta}\Big),
\end{equation} 
which is clearly $ \geq (\log {D})^{2+\varepsilon}$  in view of \eqref{new-x} and \eqref{V}. We now estimate, using Lemma \ref{lemma43} as a crucial input, the quantities
\begin{displaymath}
M_1 = \sum_{S_1(\chi) \geq (1 - 4\varepsilon)V} F\Big(\frac{\lambda_{\chi}}{Q_{\pi_1} Q_{\pi_2}}\Big), \quad M_2 = \sum_{S_2(\chi) \geq  \varepsilon V} F\Big(\frac{\lambda_{\chi}}{Q_{\pi_1} Q_{\pi_2}}\Big), \quad M_3 = \sum_{S_3(\chi) \geq  \varepsilon V} F\Big(\frac{\lambda_{\chi}}{Q_{\pi_1} Q_{\pi_2}}\Big)
\end{displaymath}
so that $\mathcal{T}(V + \mu_{ D}) \leq M_1+M_2 + M_3$.

\subsection{Bounding $M_1, M_2, M_3$} 
We decompose the interval \eqref{V} as $I_1\cup I_2$, where
\begin{displaymath}
\begin{split}
&I_1=\left[ \varepsilon \log\log {D} ,\frac{\varepsilon }{A}\Delta \cdot {\rm var}_{ D} \right],\quad I_2 =\left[ \frac{\varepsilon}{A}\Delta \cdot {\rm var}_{D}  , (\log {D})^{1/2}- \mu_{ D}\right],
\end{split}
\end{displaymath}
and we recall \eqref{3star} and \eqref{HD}. 
\begin{lemma}\label{n1lemma}
We have
\begin{equation*}
M_1 \ll  \begin{cases}H_E \exp\Big( -\frac{((1 - 5\varepsilon)V)^2}{8 {\rm var}_{D}} \Big),& V\in I_1,\\ H_E \exp\Big(-c(\varepsilon) V\log \frac{V}{{\rm var}_{ D}}\Big), & V \in I_2 .\end{cases}
 \end{equation*}
 \end{lemma}

\begin{proof}
By Lemma \ref{lemma43} (with $\nu=1$) we have
\[
M_1 \leq  \frac{2k!}{k!} \frac{1}{((1-4\varepsilon) V)^{2k}} \Big( \frac{1}{2}\sum_{ p \leq z} \frac{(\lambda_1(p) + \lambda_2(p))^2(1 + \eta_E(p)) }{p} \Big)^k \sum_{ \chi \in \widetilde{\rm Cl}_E^\vee}F\Big(\frac{\lambda_{\chi}}{Q_{\pi_1} Q_{\pi_2}}\Big) ,
\]
provided that $z^{2k} < {D}/4$. 
  Using \eqref{fac}, \eqref{dirich} and \eqref{boundsigma} along with the obvious fact  ${\rm var}_{ D}(z)  \leq {\rm var}_{D}$, we conclude
\begin{equation}\label{n1}
M_1 \ll H_E\Big( \frac{8k ({\rm var}_{ D}(z) + O(\log\log\log {D})) }{e((1-4\varepsilon)V)^2 }\Big)^k \ll H_E\Big( \frac{8k\, {\rm var}_{ D}}{e((1-5\varepsilon)V)^2 }\Big)^k.
\end{equation}
Our choice of   $z$ in \eqref{z} allows us to take
\begin{equation}\label{k}
k \leq  \frac{\varepsilon V \Delta}{3A}.
\end{equation}
  We now choose  
\begin{equation}\label{choose-k}
 k = \begin{cases}[ \frac{((1 - 5\varepsilon)V)^2}{8 {\rm var}_{D}}] , & V \in I_1, \\ [V], & V \in I_2 \end{cases}
\end{equation}
in agreement with \eqref{k} which completes the proof of the lemma. 
\end{proof}

The following bounds are  similar, but easier. 
\begin{lemma}\label{n2n3lemma}
There is $c(\varepsilon) > 0$ such that
\begin{equation}\label{N-bounds}
M_2,M_3\ll H_E  \exp(-c(\varepsilon) V\log V)
\end{equation}
for $V \in I_1 \cup I_2$. 
\end{lemma}

\begin{proof}
By Lemma \ref{lemma43} (with $\nu=2$ noting that the sum contains only split primes) we obtain  
$$M_3 \leq \frac{2k!}{k!} \frac{1}{(\varepsilon V)^{2k}} \Big( \sum_{p \leq \sqrt{x}} \frac{(\lambda_1(p^2) + \lambda_2(p^2) - \psi_1(p) - \psi_2(p))^2}{4p^2} \Big)^k \sum_{ \chi \in \widetilde{\rm Cl}_E^\vee}F\Big(\frac{\lambda_{\chi}}{Q_{\pi_1} Q_{\pi_2}}\Big),$$
whenever $x^{k}< ({D}/4)^{1/2}$. Thus
$$M_3 \ll  H_E \Big(\frac{c k}{(\varepsilon V)^2}\Big)^k $$
for some constant $c> 0$. Our choice  \eqref{new-x} implies that we can choose $k = [\frac{\varepsilon}{2A} V]$, yielding the stated bound for $M_3$ in \eqref{N-bounds}.

Next, by Lemma \ref{lemma43} (with $\nu=1$), the Hecke relations \eqref{mult}, and \eqref{log-series} we have (using $1 + \eta_E(p) \leq 2$ and $(r+s)^2 \leq 2(r^2 + s^2)$ for $r, s \in \Bbb{R}$) 
\begin{equation}\label{n2}
\begin{split}
M_2 &\ll H_E \frac{(2k)!}{k!} \frac{1}{(\varepsilon V)^{2k}} \Big( \sum_{z < p \leq x} \frac{(\lambda_1(p)+ \lambda_2(p))^2}{p} \Big)^k \\
&\leq  H_E \frac{(2k)!}{k!(\varepsilon V)^{2k}}\Big( \sum_{z < p \leq x} \frac{2(2 + \lambda_1(p^2) + \lambda_2(p^2))}{p} \Big)^k =  H_E  \frac{(2k)!}{k!(\varepsilon V)^{2k} }  \Big(4 \log\frac{\log x}{\log z} + O(1)\Big)^k,
\end{split}
\end{equation}
provided that $x^{2k} \leq {D}/4$ and $z \geq (\log D)^{2 + \varepsilon}$. We choose $k = [\varepsilon V/(2A)]$ and recall \eqref{fac} and \eqref{z} getting
\[
M_2 \ll H_E \Big( \frac{16 k}{e \varepsilon^2V^2} (\log\log \log {D} + O(1))\Big)^k.
\]
By \eqref{V} this is at most $H_E  \exp(-c(\varepsilon) V\log V)$ for some $c(\varepsilon) > 0$, again confirming \eqref{N-bounds}.
\end{proof}

\subsection{Completion of the proof of \eqref{primary}}\label{compl}

 We substitute the bounds of the previous two lemmas back into \eqref{int}. We start with the contribution of $M_1$. 
 The interval $I_2$ contributes 
\begin{equation}\label{I2}
\begin{split}
 H_E\exp(\mu_{D}) \int_{I_2} &\exp\Big( \frac{1}{2} V -c(\varepsilon) V\log \frac{2V}{{\rm var}_{ D}}\Big)dV \\
\ll&  H_E\exp(\mu_{D}) \int_{I_2} \exp\big(-c(\varepsilon) V \log\log\log {D}\big) dV\ll  H_E\exp(\mu_{D}),
\end{split}
\end{equation}
which is  admissible for \eqref{primary}. The contribution of $I_1$ is
\[
H_E\exp(\mu_{D}) \int_{I_1} \exp\Big( \frac{1}{2}V  -\frac{((1 - 5\varepsilon)V)^2}{8{\rm var}_{ D}}\Big) dV.
\]
We extend the range of integration to all of $\Bbb{R}$ and use the formula
$$\int_{\Bbb{R}} e^{-\alpha x^2 + \beta x} dx = \sqrt{\frac{\pi}{\alpha}} \exp\Big(\frac{\beta^2}{4\alpha}\Big), \quad \alpha, \beta > 0,$$
getting the upper bound
\begin{equation*}
\begin{split}
&H_E\exp(\mu_{D}) \left(\frac{{\rm var}_{D}}{2}\right)^{1/2}\exp\Big(\frac{{\rm var}_{ D}}{2(1 - 5\varepsilon)^2}\Big)\ll H_E\exp\Big(\mu_{D} + \Big(\frac{1}{2} + O(\varepsilon)\Big) {\rm var}_{ D}\Big)
\end{split}
\end{equation*}
in agreement with \eqref{primary}, potentially after re-defining $\varepsilon$. 

For $M_2, M_3$   we obtain the same contribution as in \eqref{I2}, which  completes the proof of \eqref{primary}. 


\subsection{The endgame}\label{end} We have now prepared the scene to complete the proof of Theorem \ref{thm3}(a). Recall that we need to establish the bound
\[
\frac{1}{H_E} \sum_{\chi \in \widetilde{\rm Cl}^\vee_E} \exp\Big(-\frac{c_0|\lambda_{\chi}|}{Q_{\pi_1} Q_{\pi_2}}\Big)\Big(\frac{L(1/2, \pi_1 \times \chi)L(1/2, \pi_2 \times \chi)}{L(1,\text{{\rm Ad }} \pi_1)L(1,\text{{\rm Ad }} \pi_2) }\Big)^{1/2} \ll (\log {D})^{-1/4 + \varepsilon},
\]
with polynomial dependence on $Q_{\pi_1}, Q_{\pi_2}$.

 By \eqref{titch3}, \eqref{HD} and  the convexity bound, the left hand side is trivially (and crudely) bounded by $({D} Q_{\pi_1} Q_{\pi_2})^{10} \ll (Q_{\pi_1} Q_{\pi_2})^{10^3} (\log {D})^{-1/4}$  if $\max(Q_{\pi_1}, Q_{\pi_2}) \geq {D}^{1/10}$. So from now on we can assume 
 \begin{equation}\label{assume}
    Q_{\pi_1}, Q_{\pi_2} \leq {D}^{1/10}.
    \end{equation} 
 
 Next we majorize $x \mapsto \exp(-c_0|x|)$ by a non-negative, even Schwartz class function $F$ whose Fourier transform has support in $[-\frac{1}{2\pi}, \frac{1}{2\pi}]$. 
Recalling \eqref{HD}, \eqref{defmu} and \eqref{defsigma} and using  \eqref{titch3} for the error term, the   bound \eqref{primary} yields
\begin{equation*}
\begin{split}
&\frac{1}{H_E}\sum_{\chi \in \widetilde{\rm Cl}_E^\vee}F\Big(\frac{\lambda_{\chi}}{Q_{\pi_1} Q_{\pi_2}}\Big) \Big( \frac{L(1/2, \pi_1 \times \chi)L(1/2, \pi_2\times \chi)}{L(1, \text{Ad}\, \pi_1) L(1, \text{Ad}\, \pi_1) }\Big)^{1/2}\\
& \ll \frac{L(1, \eta_E)^{\frac{3}{4} + \frac{\varepsilon}{2}} L(1, \pi_1\times \pi_2 \times \theta_E)^{\frac{1}{4} + \frac{\varepsilon}{2}}  }{ L(1, \text{Ad}\, \pi_1 \times \theta_E)^{\frac{1}{8} - \frac{\varepsilon}{4}}L(1, \text{Ad}\, \pi_1 \times \theta_E)^{\frac{1}{8}- \frac{\varepsilon}{4} }(\log {D})^{\frac{1}{4} -\frac{\varepsilon}{2}}}  + \exp(- \textstyle\frac{1}{4}\sqrt{\log {D}})
\end{split}
\end{equation*}
whenever $Q_{\pi_1}, Q_{\pi_2} \leq {D}^{1/10}$.  By \eqref{dirich2} and \eqref{titch2}, the right hand side is
\begin{equation}\label{RHS}
\ll \frac{L(1, \pi_1 \times \pi_2 \times \theta_E)^{\varepsilon}}{(\log {D})^{1/4 - \varepsilon}} + \exp(- \textstyle\frac{1}{4}\sqrt{\log {D}}).
\end{equation}
Let us temporarily make the additional assumption
\begin{equation}\label{temp}
Q_{\pi_1}, Q_{\pi_2} \leq (\log {D})^{10}.
\end{equation}
In this case, \eqref{log-series} and \eqref{add} with $x = (\log \log D)^3$, say,  imply the existence of a constant $C$ such that 
$$L(1, \text{Ad }\pi_j) \ll \exp(C(\log\log {D})^{1/3}) \ll (\log {D})^{\varepsilon}.$$
Together with \eqref{dirich3} we obtain
$$L(1, \pi_1 \times \pi_2 \times \theta_E) = L(1, \text{Ad }\pi_1)L(1, \text{Ad }\pi_2) \frac{L(1, \pi_1 \times \pi_2 \times \theta_E)}{L(1, \text{Ad }\pi_1)L(1, \text{Ad }\pi_2)} \ll (\log {D})^{\varepsilon}.$$
which when inserted in \eqref{RHS} is admissible for Theorem \ref{thm3}a.  

Let us now assume that \eqref{temp} fails, but \eqref{assume} holds, so that $  (\log {D})^{10} \leq \max(Q_{\pi_1}, Q_{\pi_2}) \leq {D}^{1/10}.$ 
 Then we use the Cauchy-Schwarz inequality and  Lemma \ref{second} together with \eqref{curlyC} and \eqref{titch3}  to obtain
\begin{equation*}
\begin{split}
&\frac{1}{H_E}\sum_{\chi \in \widetilde{\rm Cl}_E^\vee}F\Big(\frac{\lambda_{\chi}}{Q_{\pi_1} Q_{\pi_2}}\Big) \Big( \frac{L(1/2, \pi_1 \times \chi)L(1/2, \pi_2\times \chi)}{L(1, \text{Ad}\, \pi_1) L(1, \text{Ad}\, \pi_1) }\Big)^{1/2}\\
& \ll Q_{\pi_1} Q_{\pi_2} (\log\log {D})  \exp \big(  \max_j \log Q_{\pi_j}^{2/5}\big) \ll (Q_{\pi_1} Q_{\pi_2} )^2 (\log {D})^{-1/4} 
\end{split}
\end{equation*}
as desired. This completes the proof of Theorem \ref{thm3}(a) in all cases.

\subsection{Proof of Theorem \ref{thm3}(b)}\label{77} This is similar and we highlight only the relevant changes. We keep the definition \eqref{defmu} of $\mu_D(x)$, but we redefine \eqref{defsigma} as follows:
\begin{equation*}
\begin{split}
  {\rm var}^{\ast}_{ D}(x) =&  \frac{1}{2} \log\log x + \frac{1}{2}\log L(1, \eta_E) + \frac{1}{4} \log L(1, \text{Ad} \pi_1 \times \theta_E) +  \frac{1}{4} \log L(1, \text{Ad} \pi_2 \times \theta_E).    \end{split}
\end{equation*}

\begin{remark} This differs from ${\rm var}_D(x)$ by the term  $\frac{1}{2}\log L(1, \pi_1 \times \pi_2 \times \theta_E)$. The reason for this can be traced back to a comparison of Lemma \ref{lemma43} and Lemma \ref{lemma12}. For $\alpha \not= \beta$ we have $(B_{\alpha, \beta}(2, 0), B_{\alpha, \beta}(2, 1), B_{\alpha, \beta}(2, 2)) = (2, 0, 2)$ whereas for $\alpha= \beta$ this is $(2, 4, 2)$. Consequently, Lemma \ref{lemma12} features a ``main term'' $b(p)^2 + c(p)^2$, whereas the analogous situation in Lemma \ref{lemma43} gives $(b(p) + c(p))^2 = b(p)^2 + 2b(p) c(p) + c(p)^2$. It is the extra mixed term $2 b(p)c(p)$ that is responsible for the term $\frac{1}{2}\log L(1, \pi_1 \times \pi_2 \times \theta_E)$, which of course only makes sense in the situation $\pi_1 \not= \pi_2$ of Theorem \ref{thm1}, but not in the potentially allowed situation $\pi_1 = \pi_2$ of Theorem \ref{thm2}. 
\end{remark}

We write $\mu_D = \mu_D({D})$ and $\text{var}^{\ast}_D = \text{var}^{\ast}_D({D})$. The analogues of \eqref{boundsigma} and \eqref{dirich} are
\begin{equation*}
\frac12 \log\log {D} + O(\log\log\log {D} )  \leq {\rm var}^{\ast}_{  D}  \ll (\log {D})^{1/3}
\end{equation*}
and\begin{equation*}
\begin{split}
\sum_{p \leq x} \frac{  (1 + \eta_E(p))(\lambda_1(p)^2 +   \lambda_2(p)^2 )}{2p}  = 2 \text{var}^{\ast}_D(x) + O(\log\log\log M_1M_2)
\end{split}
\end{equation*}
provided that $x  > (\log Q_{\pi_1} Q_{\pi_2} D)^{2+\varepsilon}$. We generalize \eqref{defL} to 
$$\mathcal{L}^{\ast}(\chi)=L(1/2, f\times \chi^{\alpha}) L(1/2, g\times \chi^{\beta}).$$
We   follow the argument up to \eqref{sumprimes} which now reads
\begin{equation*}
\begin{split}
\mathcal{L}^{\ast}(\chi) - \mu_{D} &  \leq S_1^{\ast}(\chi) + S_2^{\ast}(\chi) + S_3^{\ast}(\chi)+ O(\varepsilon V),
\end{split}
\end{equation*}
where
\begin{displaymath}
\begin{split}
&S_1^{\ast}(\chi) = \sum_{  p \leq z}  \frac{a_{\chi^{\alpha}}(p)\lambda_1(p) + a_{\chi^{\beta}}(p)\lambda_2(p)}{ p^{1/2 + 1/\log x}} \frac{\log(x/p)}{\log x} ,\\
&S_2^{\ast}(\chi) = \sum_{ z <  p \leq x}  \frac{a_{\chi^{\alpha}}(p)\lambda_1(p) + a_{\chi^{\beta}}(p)\lambda_2(p)}{ p^{1/2 + 1/\log x}} \frac{\log(x/p)}{\log x} ,\\
& S_3^{\ast}(\chi) = \frac{1}{2}\sum_{\substack{\eta_E(p) = 1\\ p^2 \leq x}}  \frac{a_{\chi^{2\alpha}}(p)(\lambda_1(p^2)  - \psi_1(p))+ a_{\chi^{2\beta}}(p)( \lambda_2(p^2)- \psi_2(p))}{ p^{1 + 2/\log x}} \frac{\log(x/p^2)}{\log x}. 
\end{split}
\end{displaymath}
Correspondingly we define $M^{\ast}_j$ for $j = 1, 2, 3$. In order to bound $M^{\ast}_j$ we apply Lemma \ref{lemma12} instead of Lemma \ref{lemma43}. For notational simplicity we study the case $v_2(\alpha) = v_2(\beta)$, the other case being almost identical. In the following let $c > 0$ denote a sufficiently large constant,  \emph{not necessarily the same at every occurrence.}

We have
$$M_3^{\ast} \ll H_E \sum_{2\ell_1 + 4\ell_2 = 2k} \frac{(2k)!}{(\ell_1)! (\ell_2)!}  \frac{1}{(\varepsilon V)^{2k}} c^k$$
if $x^k < ({D}/4)^{1/\max(2\alpha, 2\beta)}$.  Since $(\ell_1)! (\ell_2)! \geq [k/3]!, $
we obtain by Stirling's formula
$$M_3^{\ast} \ll H_E \Big( \frac{ck^{5/3}}{(\varepsilon V)^2}\Big)^{k}.$$
 Choosing $k \asymp V$, we obtain the analogue of \eqref{N-bounds} for $M_3^{\ast}$.  In the same way, we obtain
$$M_2^{\ast} \ll H_E \sum_{2\ell_1 + 4\ell_2 = 2k} \frac{(2k)!}{(\ell_1)! (\ell_2)!}  \frac{1}{(\varepsilon V)^{2k}} \Big( 2 \log\frac{\log x}{\log z} + O(1)\Big)^{\ell_1} c^{\ell_2}$$
as an analogue of \eqref{n2},  and it is easy to confirm \eqref{N-bounds} also for $M_2^{\ast}$.

The estimation of $M_1^{\ast}$ is only slightly more difficult. As in \eqref{n1} we obtain
$$M_1^{\ast}  \ll H_E \sum_{2\ell_1 + 4\ell_2 = 2k} \frac{(2k)!}{(\ell_1)! (\ell_2)!}  \frac{(2 \text{var}^{\ast}_D  )^{\ell_1} c^{\ell_2}}{((1-5\varepsilon) V)^{2k}}.$$
We write $\ell_1 = a k$, $\ell_2 = b k$ with $a + 2b = 1$, so that 
by Stirling's formula we have  (with the convention $0^0 = 1$)
$$\frac{(2k)!}{(\ell_1)! (\ell_2)!} \ll \Big( \frac{4 k^{2-a-b}}{e^{2-a-b} a^{a} b^{b}}\Big)^k,$$
uniformly in $a, b$ (one can take $2$ as an implied constant).  We conclude
$$M_1^{\ast}  \ll H_E \sum_{2\ell_1 + 4\ell_2 = 2k}  c^{\ell_2} \Big( \frac{4 k^{2-a-b}(2 \text{var}^{\ast}_D  )^{a}}{e^{2-a-b} a^{a} b^b ((1 - 5\varepsilon)V)^{2}}\Big)^k .$$
Note that $1/2 \leq \xi^\xi \leq 1$ for $0 \leq \xi \leq 1$. 
We make the same choice for $k$ as in \eqref{choose-k}.

 If $V \in I_2$, then with $k = [V]$ we obtain
$$M_1^{\ast} \ll H_E  \sum_{2\ell_1 + 4\ell_2 = 2k}   \frac{(\text{var}^{\ast}_D)^{\ell_1}}{V^{\ell_1+\ell_2}}c^k \leq    H_E  \sum_{2\ell_1 + 4\ell_2 = 2k}   \frac{(\text{var}^{\ast}_D)^{\ell_1+\ell_2}}{V^{\ell_1+\ell_2}}c^k \ll H_E \exp\Big(-c(\varepsilon) V\log \frac{V}{{\rm var}^{\ast}_{ D}}\Big)$$
since $k/2 \leq \ell_1 + \ell_2 \leq k$. 

If $V \in I_1$ then with $k = [((1 - 5\varepsilon)V)^2/(8\text{var}^{\ast}_D)]$ and $b = (1 - a)/2$ we obtain
$$M_1^{\ast} \ll H_E \sum_{2\ell_1 + 4\ell_2 = 2k}    \Big( \frac{1}{ea^{a}b^b}\Big( \frac{cV}{    (\text{var}_D^{\ast})^{3/2}} \Big)^{1-a}\Big)^k \ll H_E \sum_{2\ell_1 + 4\ell_2 = 2k}   (e a^{a}b^b)^{-k}  (\log\log {D})^{-\frac{1}{6}\cdot 2\ell_2} $$
as $cV (\text{var}_D^{\ast})^{-3/2} \ll   \Delta  (\text{var}_D^{\ast})^{-1/2} \ll  (\log\log {D})^{-1/6}$ for $V \in I_1$. Altogether we conclude
 $$M_1^{\ast} \ll H_E \exp\Big( - \frac{((1 - 5\varepsilon)V)^2}{8\text{var}_D^{\ast}}\Big)$$
for $V \in I_1$. 
Having recovered the bounds from Lemmas \ref{n1lemma} and \ref{n2n3lemma}, the analogue of the basic bound \eqref{primary} is now
\begin{equation*}
 \sum_{\chi \in \widetilde{\rm Cl}_E^\vee}F\Big(\frac{\lambda_{\chi}}{Q_{\pi_1} Q_{\pi_2}}\Big)\mathcal{L}^{\ast}(\chi)^{1/2} \ll H_E\Big(\exp\Big(\mu_{D} + \Big(\frac{1}{2} +  \varepsilon\Big) {\rm var}^{\ast}_{ D}\Big) + \exp(- \textstyle\frac{1}{3}\sqrt{\log {D}})\Big),
 \end{equation*}
 so that finally
 \begin{equation*}
\begin{split}
&\frac{1}{H_E}\sum_{\chi \in \widetilde{\rm Cl}_E^\vee}F\Big(\frac{\lambda_{\chi}}{Q_{\pi_1} Q_{\pi_2}}\Big) \Big( \frac{L(1/2, \pi_1 \times \chi^{\alpha})L(1/2, \pi_2\times \chi^{\beta})}{L(1, \text{Ad}\, \pi_1) L(1, \text{Ad}\, \pi_1) }\Big)^{1/2}\\
& \ll \frac{L(1, \eta_E)^{\frac{3}{4} + \frac{\varepsilon}{2}}  }{ L(1, \text{Ad}\, \pi_1 \times \theta_E)^{\frac{1}{8} - \frac{\varepsilon}{4}}L(1, \text{Ad}\, \pi_1 \times \theta_E)^{\frac{1}{8}- \frac{\varepsilon}{4} }(\log {D})^{\frac{1}{4} -\frac{\varepsilon}{2}}}  + \exp(- \textstyle\frac{1}{4}\sqrt{\log {D}})  \ll (\log {D})^{1/4 + \varepsilon}
\end{split}
\end{equation*}
by \eqref{titch2} and \eqref{titch3}. This completes the proof of Theorem \ref{thm3}(b).

\appendix

\section{Some explicit computations related to Waldspurger's formula}\label{App-A} 

 Our aim in this appendix is to justify the bound \eqref{boundF}. We begin by explicating the shape of Waldspurger's formula, in the explicit form given by \cite{FMP}, which leads to the expression \eqref{twisted-Walds}.

 Let $\bB$, $\mathscr{O}$ and $K$ be as in Section \ref{sec:adelic-to-classical}. Let $N$ denote the discriminant of $\mathscr{O}$; then $N$ is square-free. Fix an optimal embedding $\iota$ of the quadratic field $E$ into $\bB(\Q)$ satisfying \eqref{eq:optimal} and \eqref{conjugation-Tinfty}. As usual we write $\bT_\iota$ for the associated torus in $\bPB^\times$.

 Let $\sigma\subset L^2_{\rm disc}([\bPB^\times])$ be irreducible and have non-zero invariants by $K$. Let $\phi_\sigma$ be a non-zero vector in the line $\sigma^K$, normalized to have $L^2$-norm 1. Fix $\chi\in \widetilde{\rm Cl}_E^\vee$ and assume that ${\rm Hom}_{\bT_\iota}(\sigma,\chi)\neq 0$. Recall from the discussion in \S\ref{sec:fractional} that $\phi^\circ_\sigma={\rm g}.\phi_\sigma$ is the global Gross--Prasad vector, with respect to the pair $(\bT_\iota,\chi)$.

 We recall the twisted adelic torus period $\mathcal{P}^\chi_\bT(\phi_\sigma^\circ)$ from \eqref{twisted}, where the measure is normalized to have volume 1. Note that the toric period in \cite{FMP} is taken with respect to the Tamagawa measure $\mu_{\rm Tam}^\bT$ on $[\bT]$. Moreover, the $L^2$-normalization of the test vectors in \cite{FMP} is itself taken with respect to Tamagawa measure $\mu_{\rm Tam}^{\bG}$ on $[\bG]$. 

  Let $\pi$ be the irreducible cuspidal automorphic representation of $\PGL_2(\A_\Q)$ of level $N$ corresponding to $\sigma$ via the Jacquet--Langlands correspondence.  We apply \cite[Theorem 1.1]{FMP} with
\[
S(\Omega) = S_2(\pi) = \emptyset,\qquad S(\pi) = S_1(\pi)=\{p \mid N \},\quad\textrm{and}\quad  S_0(\pi) = \{p \mid (D, N)\}.
\]
From \cite[Theorem 1.1]{FMP} we obtain  
\[
|\mathcal{P}_\bT^\chi(\phi_{\sigma}^\circ)|^2 =\frac{\mu_{\rm Tam}^\bG ([\bG])^2}{\mu_{\rm Tam}^\bT([\bT])^2}\frac12\frac{1}{\sqrt{D}}L_{S(\pi)}(1,\eta_E)\zeta^{S(\pi)}(2)\prod_{p\mid N}e(E_p/\Q_p)C_\infty(E,\pi,\chi)\frac{L^{S_0}(1/2,\pi\times\chi)}{L^{S_0}(1,{\rm Ad}\, \pi)},
\]
where $C_{\infty}(E, \pi, \chi)$ is defined in \cite[Section 7B]{FMP} and recalled below. Note that, if $S$ is a finite (possibly empty) set of primes, the superscript notation  $L^S$ includes the local factor at infinity. Using $\mu_{\rm Tam}^\bT([\bT])={\rm Res}_{s=1}\zeta_E(s)=L(1,\eta_E)$ and reorganizing we obtain
\[
|\mathcal{P}_\bT^\chi(\phi_{\sigma}^\circ)|^2 =C_{\bG}C_{\rm Ram}( \pi, \chi)\frac{1}{L(1,\eta_E)^2}\frac{1}{\sqrt{D}}\frac{L(1/2,\pi\times\chi)}{L(1,{\rm Ad}\, \pi)}F(\pi_\infty,\chi_\infty),
\]
where $C_{\bG}=\mu_{\rm Tam}^\bG(\bG)^2\frac12\xi(2)$,
\[
C_{\rm Ram}( \pi, \chi)=\frac{L_{S_0}(1,{\rm Ad}\, \pi)}{L_{S_0}(1/2,\pi\times\chi)}\prod_{p\mid N}e(E_p/\Q_p) \frac{1 - p^{-2}}{1 - \eta_E(p) p^{-1}},
\]
and
\[
F(\pi_\infty,\chi_\infty)=C_\infty(\pi,\chi)\frac{L_\infty(1/2,\pi\times\chi)}{L_\infty(1,{\rm Ad}\, \pi)}.
\]

Recall the notation $\lambda_\pi$ for the spectral parameter of $\pi$ in Section \ref{sec:Weyl} and $\lambda_\chi\in\R$ for the frequency of $\chi$ from Section \ref{sec:fractional}.

\begin{lemma}
We have
\[
C_{\rm Ram}( \pi, \chi)\ll_\varepsilon N^{\varepsilon}\qquad\textrm{and}\qquad F(\pi_\infty,\chi_\infty)\ll \exp(-c_0|\lambda_\chi| /\lambda_\pi).
\]
\end{lemma}

\begin{proof}
At the finite places the local factors are given by \eqref{dir1} -- \eqref{RS-euler}.  Since $|\alpha(p, i)| \leq p^{-1/2}$ for $p \mid N$ we obtain $|L_p(1/2, \pi \times \chi)|^{-1} \leq (1 - p^{-1})^{-4}$ and $|L_p(1, \text{Ad} \,\pi)| \leq 1$, and so $|C_{\rm Ram}( \pi, \chi)| \leq \prod_{p\mid N} 2 (1 - 1/p)^{-5} \ll N^{\varepsilon}$. 

For the second estimate, recall that $\pi_{ \infty}$ can be either discrete series of weight $k$ or principal series with spectral parameter $t$. In the latter case, it suffices to assume that $t\in\R$, so that $\pi_\infty$ is tempered. In this notation, the archimedean $L$-factors take the form 
\[
L_{\infty}(s, \pi \times \chi) = \begin{cases} 4 (2\pi)^{-2s} \prod_{\pm}\Gamma(s \pm it), &  \pi_{\infty}\text{ principal series, } E \text{ imaginary;}\\ 
 \pi^{-2s} \prod_{\pm, \pm} \Gamma(\frac{1}{2}(s \pm it \pm i\lambda_{\chi})), &\pi_{\infty}\text{ principal series, } E \text{ real;}\\
  4(2\pi )^{-2s - (k-1)}\prod_{\pm} \Gamma(s + \frac{1}{2}(k-1) \pm i\lambda_{\chi}), & \pi_{\infty}\text{ discrete series, } 
 \end{cases}
 \]
 and
 \[
 L_{\infty}(s, \text{Ad } \pi) = \begin{cases}  \pi^{-3s/2} \Gamma(s/2) \prod_{\pm}\Gamma(s/2 \pm it), & \pi_{\infty}\text{ principal series;}\\2^{2-k-s} \pi^{(1 - 2k - 3s)/2} \Gamma(s + k - 1) \Gamma((s+1)/2),& \pi_{\infty}\text{ discrete series.} \end{cases}
 \]
 Moreover, by definition, we have
\[
C_\infty(E, \pi , \chi) = \begin{cases} 1, & \pi_{ \infty}\text{ principal series};\\ \frac{\Gamma(k )}{\pi \Gamma(k /2)^2}, & \pi_{\infty} \text{ discrete series, } E \text{ imaginary};\\
2^k, & \pi_{ \infty} \text{ discrete series, }  E \text{ real}.\end{cases}
\]
By Stirling's formula we have
\[
C_\infty(E, \pi , \chi) \frac{L_{\infty}(1/2, \pi \times \chi)}{L_{\infty}(1, \text{Ad } \pi)} \ll \begin{cases}   1, & E \text{ imaginary};\\  e^{-\pi \max(0,|\lambda_{\chi}| - |t|)} , &\pi_{\infty}\text{ principal series, } E \text{ real;}\\
e^{-c\min(|\lambda_{\chi}|, |\lambda_{\chi}|^2/k)}, &\pi_{\infty}\text{ discrete series, } E \text{ real},
 \end{cases}
\]
for some absolute constant $c > 0$. Using $\lambda_\pi^2=k(k+1)$ or $\lambda_\pi^2=1/4+t^2$ according to whether $\pi_\infty$ is discrete or principal series, we estimate this very crudely by $\exp(-c_0|\lambda_\chi|/ \lambda_\pi)$.
\end{proof}

\end{document}